\documentclass[12pt]{article}   
\usepackage[utf8]{inputenc}
\usepackage{amsmath,amssymb,amsthm, geometry, wasysym}
\usepackage{graphicx, verbatim, caption, subcaption}
\usepackage{xcolor, svg}
\usepackage{fullpage}
\usepackage{hyperref}

\usepackage[psamsfonts]{eucal}
\usepackage{microtype}
\usepackage{mathtools}

\newtheorem{theorem}{Theorem}[section] 
\newtheorem{lemma}[theorem]{Lemma}     
\newtheorem{corollary}[theorem]{Corollary} 
\newtheorem{proposition}[theorem]{Proposition} 

\theoremstyle{definition} 
\newtheorem{definition}{Definition} 

\theoremstyle{remark} 
\newtheorem{remark}[theorem]{Remark} 

\newtheorem{observation}{Observation}

\newcommand{\E}{\ensuremath{\mathbb E}}
\newcommand{\R}{\ensuremath{\mathbb R}}

\newcommand{\lab}{\label}  \newcommand{\ra}{\ensuremath{\rightarrow}}   \def\de{{\mathbf{\delta}}} \def\De{{{\Delta}}}  \def\m{{\mathbf{\mu}}}
 \def\var{{\mathrm{var}}} \def\beq{\begin{eqnarray}} \def\eeq{\end{eqnarray}} \def\ben{\begin{enumerate}}
\def\een{\end{enumerate}} \def\bit{\begin{itemize}}
\def\bel{\begin{lemma}}
\def\eel{\end{lemma}}
\def\eit{\end{itemize}} \def\beqs{\begin{eqnarray*}} \def\eeqs{\end{eqnarray*}} \def\bel{\begin{lemma}} \def\eel{\end{lemma}}
\newcommand{\N}{\mathbb{N}} \newcommand{\Z}{\mathbb{Z}}  \newcommand{\C}{\mathcal{C}} 
\newcommand{\T}{\mathbb{T}}  \newcommand{\x}{\mathbf{x}} \newcommand{\y}{\mathbf{y}}    

    \newcommand{\p}{\mathbb{P}}
\newcommand{\PP}{\mathcal P}    \newcommand{\one}{\mathrm{1}}
 \newcommand{\MM}{\mathcal M}\newcommand{\NN}{\mathcal N} \newcommand{\la}{\lambda}  
  \def\eps{{\epsilon}}  \def\ie{i.\,e.\,}

\newcommand{\RR}{\mathbb{R}}
\newcommand{\dist}{\mathbf{d}}

\renewcommand{\L}{\mathbb{L}}

\renewcommand{\L}{\mathbb L}
\renewcommand{\C}{\mathcal C}

\newtheorem{claim}{Claim}
\newtheorem{obs}[theorem]{Observation}
\numberwithin{equation}{section}
\numberwithin{figure}{section}

\newcommand{\Spec}{\mathtt{Spec}}

\newcommand{\diag}{\mathrm{diag}}

\newcommand{\tilh}{h}

\newcommand{\HIVE}{\mathtt{HIVE}}
\newcommand{\AHIVE}{\mathtt{AUGHIVE}}
\newcommand{\HORN}{\mathtt{HORN}}

\newcommand{\GT}{\mathtt{GT}}

\newcommand{\Var}{\operatorname{var}}

\newcommand{\weight}{\mathbf{wt}}
\newcommand{\edge}{{\diamond}}

\renewcommand{\RR}{\mathcal R}

\renewcommand{\P}{\mathbb{P}}
\newcommand{\rel}{\to}
\newcommand{\HT}{\mathtt{HT}}
\newcommand{\AHT}{\mathtt{AHT}}

\newcommand{\FF}{\mathcal{F}}

\newcommand{\wt}{\mathbf{wt}}

\newcommand{\nnn}{{d}}
\newcommand{\sidelength}{\,\mathrm{len}}
\newcommand{\up}{\mathtt{up}}
\newcommand{\lo}{\mathtt{lo}}

\newcommand{\rhosc}{\rho_{\mathrm{sc}}}

\newcommand{\CZ}{\mathtt{CZ}}

\newcommand{\tiling}{\Xi}
\title{On the limit of random  hives with GUE boundary conditions}
\author{Hariharan Narayanan\thanks{School of Technology and Computer Science, Tata Institute of Fundamental Research, Mumbai, India}}

\date{\today}
\begin{document}
\maketitle
\begin{abstract}
We show that hives chosen at random with independent GUE boundary conditions on two sides, weighted by a Vandermonde factor depending on the third side (which is necessary in the context of the randomized Horn problem), when normalized so that the eigenvalues at the edge are asymptotically constant, converge in probability to a continuum hive as $n \ra \infty.$ It had previously been shown in joint work with Sheffield and Tao \cite{NST} that the variance of these scaled random hives tends to $0$ and consequently, from compactness, that they converge in probability subsequentially.   In the present paper, building on \cite{NST}, we prove convergence in probability to a single continuum hive, without having to pass to a subsequence. We moreover show that the value at a given point $v$ of this continuum hive equals the supremum of a certain functional acting on asymptotic height functions of lozenge tilings.
\end{abstract}
\tableofcontents
\section{Introduction}

\begin{figure}[!h]
  \centering
  \includegraphics[width=\linewidth]{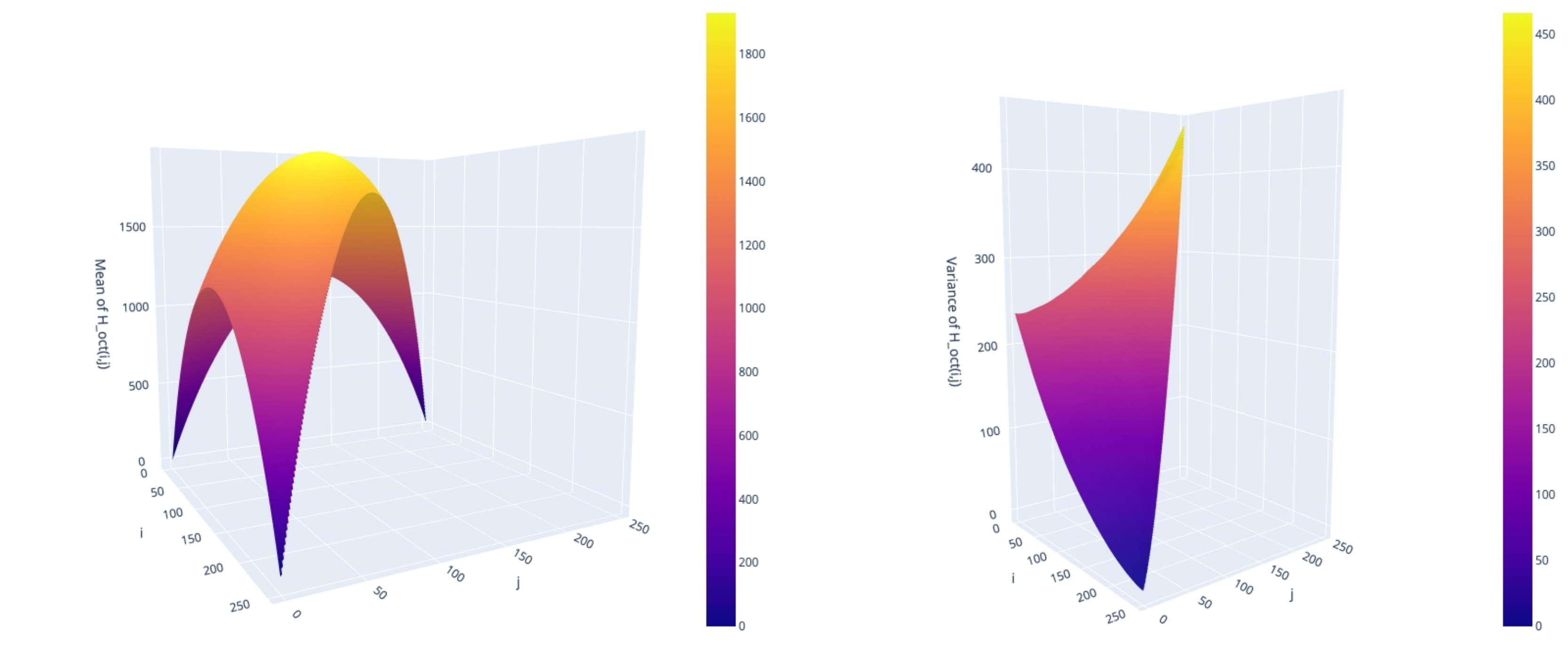}\\
  \caption{The mean and variance of a random hive corresponding to two GUE distributions with $n = 250$ from \cite{GangNar}.}
  \label{fig:GUE_hive}
\end{figure}

\subsection{Motivation}
Let $\la, \mu, \nu$ be vectors in $\R^n$ whose entries are non-increasing. Assume that $\sum_i \la_i + \sum_i\mu_i = \sum_i \nu_i.$

 Take an equilateral triangle $\T_n$ of side $n$. Tessellate it with equilateral triangles of side $1$. Assign boundary values to $\T_n$ as in Figure~\ref{fig:tri3}; Clockwise, assign the values $0, \la_1, \la_1 + \la_2, \dots, \sum_i \la_i, \sum_i \la_i + \mu_1, \dots, \sum_i \la_i + \sum_i \mu_i.$ Then anticlockwise, on the horizontal side, assign  $$0, \nu_1, \nu_1 + \nu_2, \dots, \sum_i \nu_i.$$

Knutson and Tao defined this hive model for Littlewood-Richardson coefficients with $\la, \mu, \nu \in \Z^n$, in \cite{KT1}. They showed that the Littlewood-Richardson coefficient (which plays an important role in the representation theory of $GL_n(\mathbb{C})$)
$c_{\la\mu}^\nu$ is given by the number of ways of assigning integer values to the interior vertices of the triangle, such that the piecewise linear extension to the interior of $\T_n$ is a concave function $f$ from $\T_n$ to $\R$.  
Such  an integral ``hive" $f$ can be described as an integer point in a certain polytope known as a hive polytope. The volumes of these polytopes shed light on the asymptotics of Littlewood-Richardson coefficients \cite{Nar, Okounkov, Greta}. More generally, when $\la, \mu, \nu \in \R^n$, 
they can be used to specify the Horn probability measure exactly \cite{KT2, Zuber}. Indeed,
the volume of the polytope of all real hives with fixed boundaries $\la, \mu, \nu$ is equal, up to known multiplicative factors involving Vandermonde determinants, to the probability density of obtaining a Hermitian matrix with spectrum $\nu$ when two Haar random Hermitian matrices with spectra  $\la$ and $\mu$   are added \cite{KT2}. The  relation was stated in an explicit form by Coquereaux and Zuber in \cite{Zuber} and has been reproduced here in (\ref{horn-prob}).

\begin{figure}
\begin{center}
\includegraphics[scale=0.40]{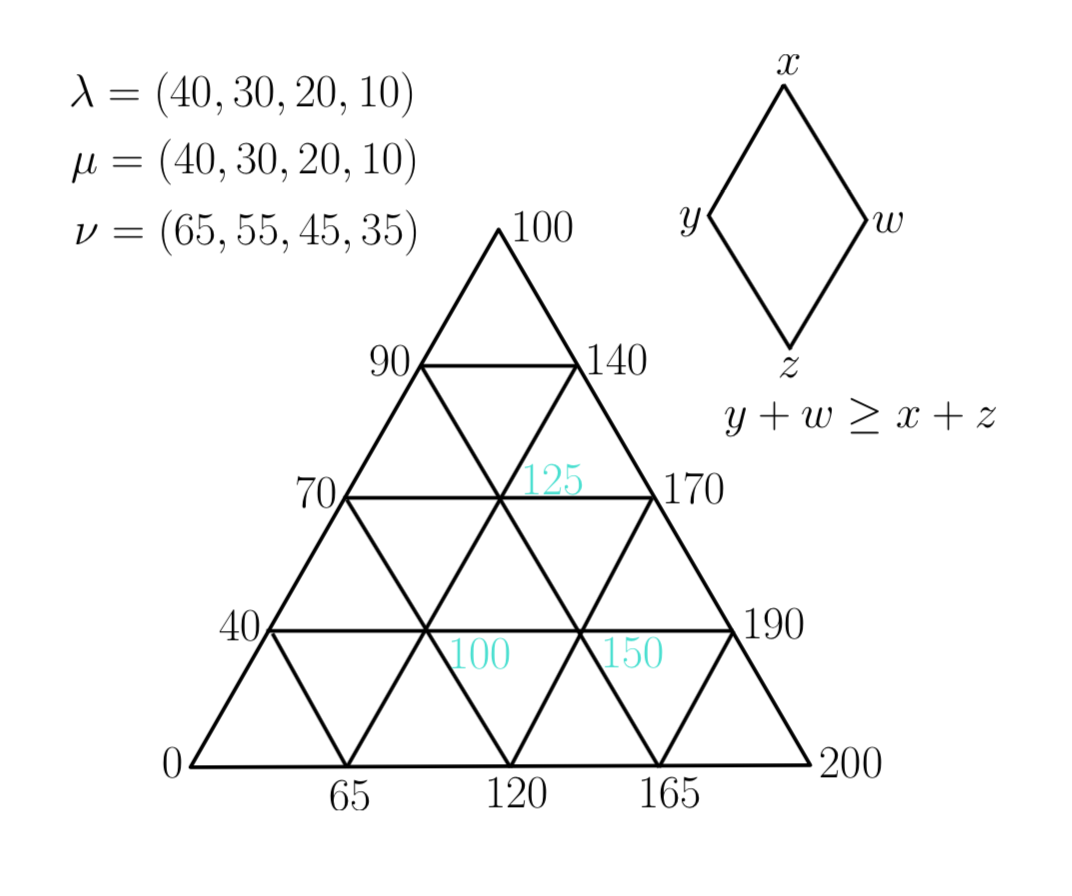}
\caption{Values taken at interior vertices in the hive model satisfy rhombus inequalities as shown above. A function which satisfies all possible rhombus inequalities is called {\it rhombus concave.}}\label{fig:tri3}
\end{center}
\end{figure}


The question of understanding the spectrum of the sum of two matrices with given spectra (known as the Horn problem) is related to a large body of research, including work in free probability. For this reason, and given the relationship between hives and this problem, it is natural to ask if random hives (appropriately weighted to reflect the connection with free additive convolution) have a limit as the size of the hives tends to infinity. In this paper we address this question in the special case when the spectra of the two matrices come from two independent scaled GUEs, and answer it in the affirmative. This paper builds on earlier joint work with Sheffield and Tao \cite{NST}, where the concentration of the hive around its expectation was proved, but whether the normalized expectation converged to a limit was left open. We also make crucial use of a result of Fefferman on quantitative differentiation (Theorem~\ref{thm:Feff}) and other results of Tao including one on the correlation decay of interlacing gaps in the GUE minor process on constant sized patches (Theorem~\ref{theorem:Tao4}).
\subsection{Background on hives}
 
\newcommand{\tih}{\tilde{h}}
 
Let $\Spec_n$ denote the cone of all possible vectors $x = (x_1, \dots, x_n)$ in $\R^n$ such that $$x_1 \geq x_2 \geq \dots \geq x_n.$$ 

We consider a $n \times n$ lattice square  $\Box_n$ with vertices $(0, 0)$, $(0, n)$, $(n, 0)$ and $(n, n)$. 
We use $[n]$ to denote $\Z \cap [1, n].$

Note that the eigenvalues of an $n \times n$ Hermitian matrix, ordered in non-increasing order, become an element of $\Spec_n$.  We continue with the notation of \cite{NST} below. 
We define a relation
\begin{equation}\label{rel}
 \lambda \boxplus \mu \rel \nu
\end{equation}
if there exist Hermitian matrices $A,B$ with eigenvalues $\la, \mu$ respectively such that $A+B$ has eigenvalues $\nu$.  

 In \cite{Weyl}, Weyl asked the question of determining necessary and sufficient conditions on $\la, \mu, \nu \in \Spec_n$ for the relation \eqref{rel} to hold.  As conjectured by Horn \cite{Horn} and proven in \cite{KT1} (building on \cite{Kly}), the set 
$$ \HORN_{\la \boxplus \mu} \coloneqq \{ \nu \in \Spec: \lambda \boxplus \mu \rel \nu \}$$
of possible $\nu$ arising from a given choice of $\la, \mu$ forms a polytope (known as the \emph{Horn polytope}), given by the trace condition
\begin{equation}\label{trace}
 \sum \lambda + \sum \mu = \sum \nu
\end{equation}
(where we abbreviate $\sum \lambda \coloneqq \sum_{i=1}^n \lambda_i$)
together with a recursively defined set of linear inequalities known as the \emph{Horn inequalities}, which include for instance the Weyl inequalities
$$ \nu_{i+j-1} \leq \lambda_{i} + \mu_{j}$$
for $1 \leq i,j,i+j-1 \leq n$, as well as many others.  We refer the reader to \cite{KT2} for a survey of the history of this problem and its resolution.  For $\lambda,\mu \in \Spec^\circ$, the Horn polytope $\HORN_{\la \boxplus \mu}$ is $n-1$-dimensional.

One of the key tools used in the proof of the Horn conjecture in \cite{KT1} is that of a \emph{hive}, defined as follows.

Let $\T$ be the isosceles right triangle in $\R^2$ whose vertices are $(0, 0)$, $(0, 1)$ and $(1, 1)$.  Let $\T_n$ be the set of lattice points $n\T \cap (\Z^2)$, where $n$ is a positive integer.
For $i = 0, 1$ and $2$,  we define $E_i(\T_n)$ to be the sets of parallelograms corresponding to $\De_0, \De_1$ and $\De_2$ on $\T_n$, respectively.  

\begin{definition}[Discrete Hessian and the $\De_i$ on $\T_n$]
Let $f: \T_n \ra \R$ be a function.
\bit
\item Let $E_0(\T_n)$ be the set of all parallelograms $e_0\subseteq \T_n$ whose vertices are $\{(v_1, v_2),  (v_1 + 1, v_2),  (v_1 + 1, v_2 + 1),  (v_1 +2, v_2 + 1)\}.$

\item Let $E_1(\T_n)$ be the set of all parallelograms $e_1\subseteq \T_n$ whose vertices are $\{(v_1, v_2),  (v_1 + 1, v_2),  (v_1, v_2 + 1),  (v_1 +1, v_2 + 1)\}.$ 
\item Let $E_2(\T_n)$ be the set of all parallelograms $e_2\subseteq \T_n$ whose vertices are $\{(v_1, v_2),  (v_1 + 1, v_2+1),  (v_1, v_2 + 1),  (v_1 +1, v_2 + 2)\}.$ 
\eit

We define the discrete Hessian $\nabla^2(f):E(\T_n) \ra \R$ to be a  real-valued function on the set $E(\T_n) = E_0(\T_n) \cup E_1(\T_n) \cup E_2(\T_n)$ and the $\De_i$ from $\R^{V(\T_n)}$ to $\R^{E_i(\T_n)}$ by 
\beq\lab{eq:A}
\nabla^2 f(e_0) := \De_0 f(e_0) :=  f(v_1, v_2) - f(v_1 + 1, v_2) - f(v_1 + 1, v_2 + 1) + f(v_1 +2, v_2 + 1).\nonumber\\
\nabla^2f(e_1) := \De_1 f(e_1) :=  - f(v_1, v_2) + f(v_1 + 1, v_2) + f(v_1, v_2 + 1) - f(v_1 + 1,  v_2 + 1).\nonumber\\
\nabla^2f(e_2) := \De_2 f(e_2) :=  f(v_1, v_2) - f(v_1+1, v_2+1) - f(v_1, v_2 + 1) + f(v_1 + 1, v_2 + 2).\nonumber\\
\eeq
\end{definition}

\begin{definition}[Rhombus concavity]\lab{def:rhomb}
Given a function $h:\T \ra \R$, and a positive integer $n$,  let $h_n$ denote the function from $\T_n$ to $\R$ such that for $(nx, ny) \in \T_n$, $h_n(nx, ny) = n^2 h(x, y).$
A function $h:\T \ra \R$ is called {\bf rhombus concave} if for any positive integer $n$, 
and any $i$, $\De_i h_n$ is nonpositive on $E_i(\T_n),$ and $h$ is continuous on $\T$.
The corresponding function $h_n$ is called {\bf discrete (rhombus) concave}.  
Note that a necessary and sufficient condition for a function $h_n$ from $\T_n$ to $\R$ to be discrete concave,  is that the piecewise linear extension (which we denote $\tilh_n$) of $h_n$ to $n\T$ is  concave.
Here each piece is an isosceles right triangle with a $\sqrt{2}-$length  hypotenuse parallel to the vector $(1, 1).$
\end{definition}

\begin{definition}[Hive]Let $ \HIVE_{\lambda^{(n)} \boxplus \mu^{(n)} \rel \nu_n}$ denote the set of all discrete concave functions $h_n:\T_n \ra \R,$ (which,  following Knutson and Tao \cite{KT1},  we call hives)
such that \ben \item $\forall i \in [n]\cup\{0\},\quad h_n(0,  i) = \sum_{j = 1}^i \la_n(j).$
\item $\forall i \in [n]\cup\{0\},\quad h_n(i,  n) = \sum_{j = 1}^n \la_n(j) + \sum_{j = 1}^i \mu_n(j).$
\item $\forall i \in [n]\cup\{0\},\quad h_n(i,  i) = \sum_{j = 1}^i \nu_n(j).$
\een
Let $| \HIVE_{\lambda^{(n)} \boxplus \mu^{(n)} \rel \nu_n}|$ denote the ${n-1 \choose 2}-$dimensional Lebesgue measure of this hive polytope. 
\end{definition}

There is a natural probability measure on the Horn polytope $\HORN_{\la \boxplus \mu}$, referred to as the \emph{Horn probability measure} in \cite{Zuberhorn}, defined as the eigenvalues of $A+B$ when $A, B$ are chosen independently and uniformly (i.e., with respect to $U(n)$-invariant Haar measure) the space (essentially a coadjoint orbit) of all Hermitian matrices with eigenvalues $\la, \mu$ respectively.  This Horn measure turns out to be piecewise polynomial and was computed explicitly in \cite[(8), Proposition 4]{Zuber} (see also \cite{Zuberhorn}) to be given by the formula
\begin{equation}\label{horn-prob}
\frac{V(\nu) V(\tau)}{V(\la) V(\mu)} |\HIVE_{\la \boxplus \mu \rel \nu}| \ d\nu
\end{equation}
for $\la, \mu \in \Spec^\circ$, where $|\HIVE_{\lambda \boxplus \mu \rel \nu}|$ denotes the $\binom{n-1}{2}$-dimensional Lebesgue measure of the hive polytope $\HIVE_{\lambda \boxplus \mu \rel \nu}$, $d\nu = d\nu_1 \dots d\nu_{n-1}$ is $n-1$-dimensional Lebesgue measure on the hyperplane given by \eqref{trace}, 
$$ V(\la) = V_n(\la) \coloneqq \prod_{1 \leq i < j \leq n} (\la_i - \la_j)$$
is the Vandermonde determinant, and $\tau$ is the tuple
$$ \tau \coloneqq ( n, n-1, \dots, 1 ).$$

The factor of $V(\nu)$ in \eqref{horn-prob} is inconvenient, but can be removed through the device of \emph{Gelfand--Tsetlin patterns}, which can be viewed as a limiting (and much better understood) case of a hive.  Analogously to \eqref{rel}, we introduce the relation
$$ \diag(\la) \rel a$$
for $\la \in \Spec$ and $a \in \R^n$ to denote the claim that there exists a Hermitian matrix $A$ with eigenvalues $\la$ and diagonal entries $a_1,\dots,a_n$.  The classical \emph{Schur--Horn theorem} \cite{schur}, \cite{horn} asserts that the relation $\diag(\la) \rel a$ holds if and only if \emph{majorized} by $\lambda$ in the sense that one has the trace condition
\begin{equation}\label{al-trace}
\sum a = \sum \lambda
\end{equation}
and the majorizing inequalities
\begin{equation}\label{major}
 a_{i_1} + \dots + a_{i_k} \leq \lambda_1 + \dots + \lambda_k
\end{equation}
for all $1 \leq i_1 < \dots < i_k \leq n$; equivalently, $a$ lies in the \emph{permutahedron} formed by the convex hull of the image of $\lambda$ under the permutation group $S_n$.

Now define a Gelfand--Tsetlin pattern to be a pattern $\gamma = (\lambda_{j,k})_{1 \leq j \leq k \leq n}$ of real numbers obeying the interlacing conditions
\begin{equation}\label{interlacing}
 \lambda_{j,k+1} \geq \lambda_{j,k} \geq \lambda_{j+1,k+1}
\end{equation}
for $1 \leq j \leq k \leq n-1$; see Figure \ref{fig:gt}.  We say that this pattern has boundary condition $\diag(\lambda) \rel a$ for some $\lambda \in \Spec_n$ and $a \in \R^n$ if one has
$$ \lambda_{j,n} = \lambda_j$$
for $1 \leq j \leq n$ and
$$ \sum_{j=1}^k \lambda_{j,k} = \sum_{j=1}^k a_j$$
for $1 \leq k \leq n$; see Figure \ref{fig:gt-example} for an example.  The polytope of all Gelfand--Tsetlin patterns with boundary condition $\diag(\lambda) \rel a$ will be denoted $\GT_{\diag(\lambda) \rel a}$, and following \cite{NST}, we adopt the same wildcard convention as before, thus for instance
$$\GT_{\diag(\lambda) \rel \ast} \coloneqq \bigcup_a \GT_{\diag(\lambda) \rel a}$$
and
$$\GT_{\diag(\ast) \rel \ast} \coloneqq \bigcup_\lambda \GT_{\diag(\lambda) \rel \ast}$$
denote the sets of Gelfand--Tsetlin patterns with boundary conditions $\diag(\lambda) \rel \ast$ and $\diag(\ast) \rel \ast$ respectively.  We remark that for $\lambda \in \Spec^\circ_n$, $\GT_{\diag(\lambda) \rel \ast}$ is a $\binom{n}{2}$-dimensional polytope, which we call a \emph{Gelfand--Tsetlin polytope}, while $\GT_{\diag(\ast) \rel \ast}$ is a $\binom{n+1}{2}$-dimensional convex cone, which we call the \emph{Gelfand--Tsetlin cone}.

\begin{figure}
\begin{center}
\includegraphics[scale=0.50]{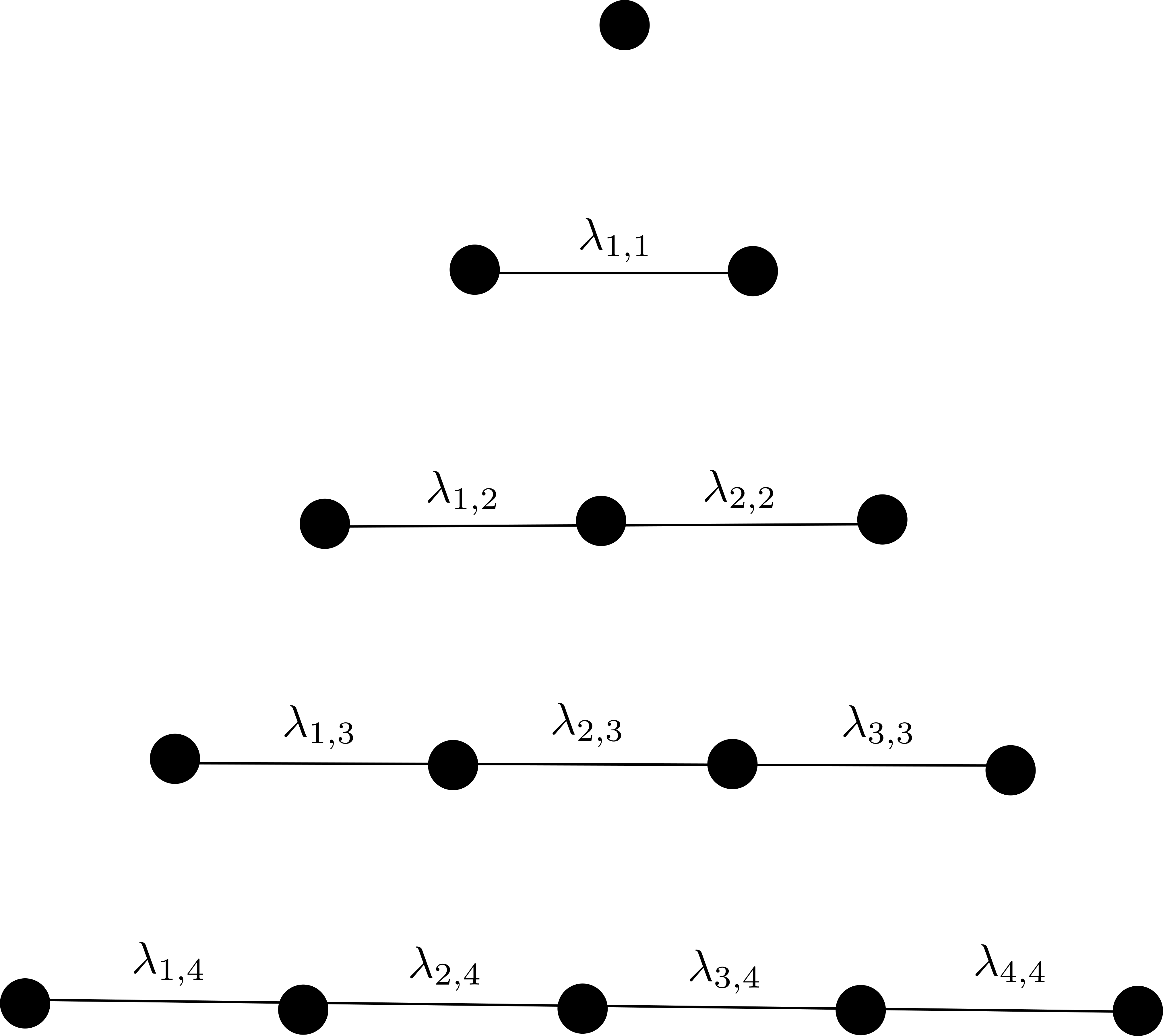}
\caption{An $n=4$ Gelfand--Tsetlin pattern.  Each number $\lambda_{i,j}$ in the pattern is greater than or equal to numbers immediately to the northeast or southeast of the pattern; in particular, every row of the pattern is decreasing. Note that such patterns are sometimes depicted as inverted pyramids instead of pyramids in the literature.}\label{fig:gt}
\end{center}
\end{figure}

\begin{figure}
\begin{center}
\includegraphics[scale=0.50]{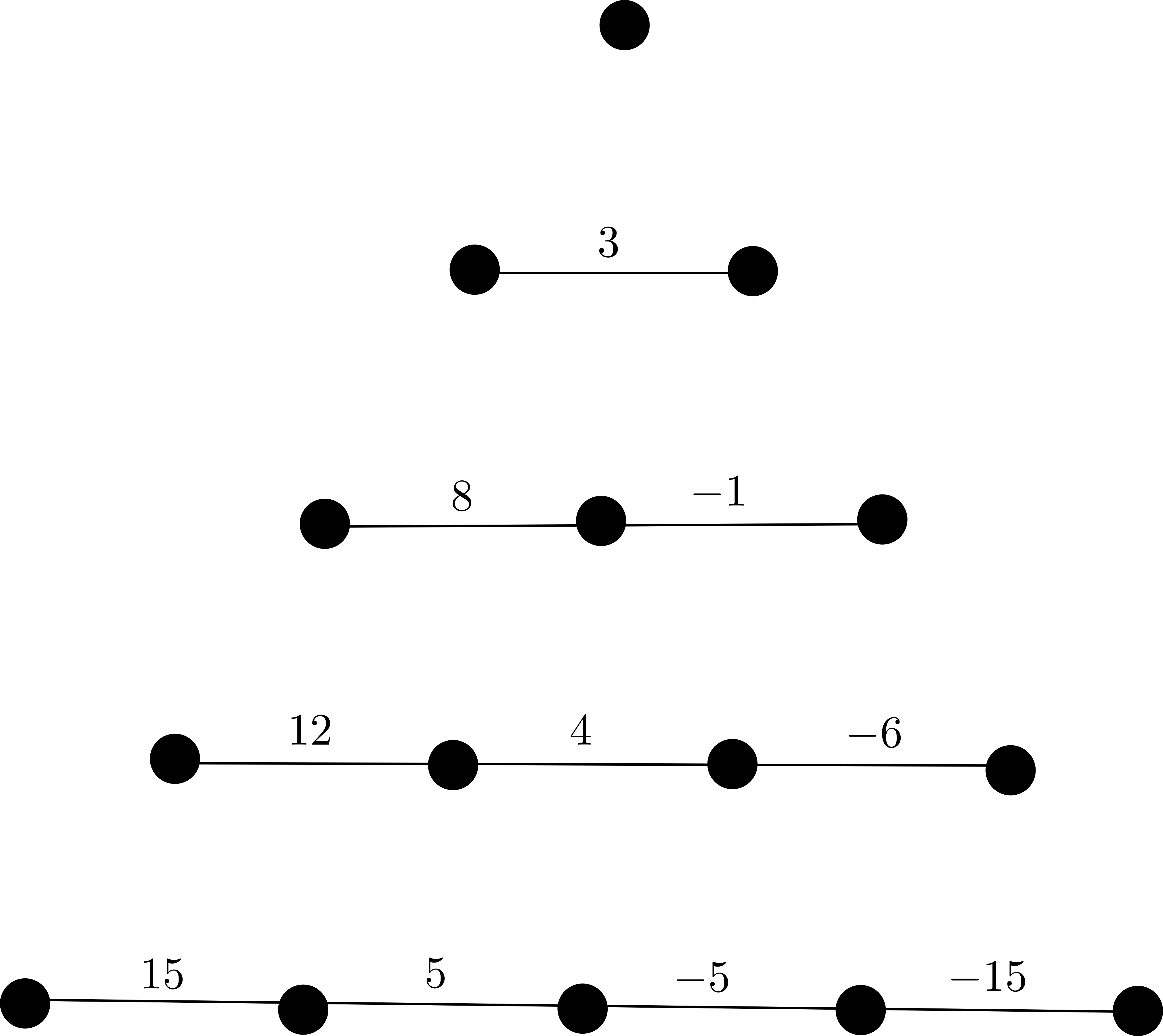}
\caption{A Gelfand--Tsetlin pattern with boundary $\diag(15,5,-5,-15) \rel (3,4,3,-10)$.}\label{fig:gt-example}
\end{center}
\end{figure}

We recall the following standard facts about Gelfand--Tsetlin polytopes from \cite[Proposition 2]{NST}:

\begin{proposition}\label{gt-rem}  Let $\lambda \in \Spec^\circ$.
\begin{itemize}
\item[(i)] If $a \in \R^n$, then $\diag(\lambda) \rel a$ holds if and only if $\GT_{\diag(\lambda) \rel a}$ is non-empty.
\item[(ii)]  The $\binom{n}{2}$-dimensional volume of $\GT_{\diag(\lambda) \rel \ast}$ is $V(\lambda) / V(\tau)$.
\item[(iii)]  Let $A$ be a random Hermitian matrix with eigenvalues $\lambda$, drawn using the $U(n)$-invariant Haar probability measure.  For $1 \leq k \leq n$, let $\lambda_{1,k} \geq \dots \geq \lambda_{k,k}$ be eigenvalues of the top left $k \times k$ minor of $A$.  Then $(\lambda_{j,k})_{1 \leq j \leq k \leq n}$ lies in the polytope $\GT_{\diag(\lambda) \rel \ast}$ with the uniform probability distribution; it has boundary data $\diag(\lambda) \rel a$ where $a = (a_{11},\dots,a_{kk})$ are the diagonal entries of $A$.
\item[(iv)]  If $\Lambda \in \Spec$ has \emph{large gaps} in the sense that 
\begin{equation}\label{large-gaps}
\min_{1 \leq i < n} \Lambda_i - \Lambda_{i+1} > \lambda_1 - \lambda_n,
\end{equation}
then there is a volume-preserving linear bijection between $\GT_{\diag(\lambda) \rel a}$ and $\HIVE_{\Lambda \boxplus \lambda \rel \Lambda+a}$ for any $a \in \R^n$, with a Gelfand--Tsetlin pattern $(\lambda_{j,k})_{1 \leq j \leq k \leq n}$ being mapped to the hive $h \colon T \to \R$ defined by the formula
\begin{equation}\label{hij}
 h(i,j) = \Lambda_1 + \dots + \Lambda_j + \lambda_{1,j} + \dots + \lambda_{i,j};
\end{equation}
see Figure \ref{fig:gt-hive}.
\end{itemize}
\end{proposition}

From this proposition, \eqref{horn-prob}, and the Fubini--Tonelli theorem, one can view the Horn probability measure associated to two tuples $\lambda,\mu \in \Spec^\circ$ to be the pushforward of $\frac{V(\tau)^2}{V(\la) V(\mu)}$ times Lebesgue measure on the ($(n-1)^2$-dimensional) \emph{augmented hive polytope}
$$ \AHIVE_{\diag(\lambda \boxplus \mu \rel \ast) \rel \ast} \coloneqq \bigcup_{\nu,a} \AHIVE_{\diag(\lambda \boxplus \mu \rel \nu) \rel a} \coloneqq \bigcup_{\nu,a} (\HIVE_{\lambda \boxplus \mu \rel \nu} \times \GT_{\diag(\nu) \rel a})$$
under the linear map that sends $\AHIVE_{\diag(\lambda\boxplus \mu \rel\nu) \rel a}$ to $\nu$ for each $\nu,a$.  We refer to the elements $(h, \gamma)$ of the augmented hive polytope $\AHIVE_{\diag(\lambda \boxplus \mu \rel \ast) \rel \ast}$, as \emph{augmented hives}. 

Using Proposition \ref{gt-rem}(iv), one can view an augmented hive as two hives glued together along a common edge, where two of the boundaries have large gaps: see Figure \ref{fig:augment}.

\begin{figure}
\begin{center}
\includegraphics[scale=0.40]{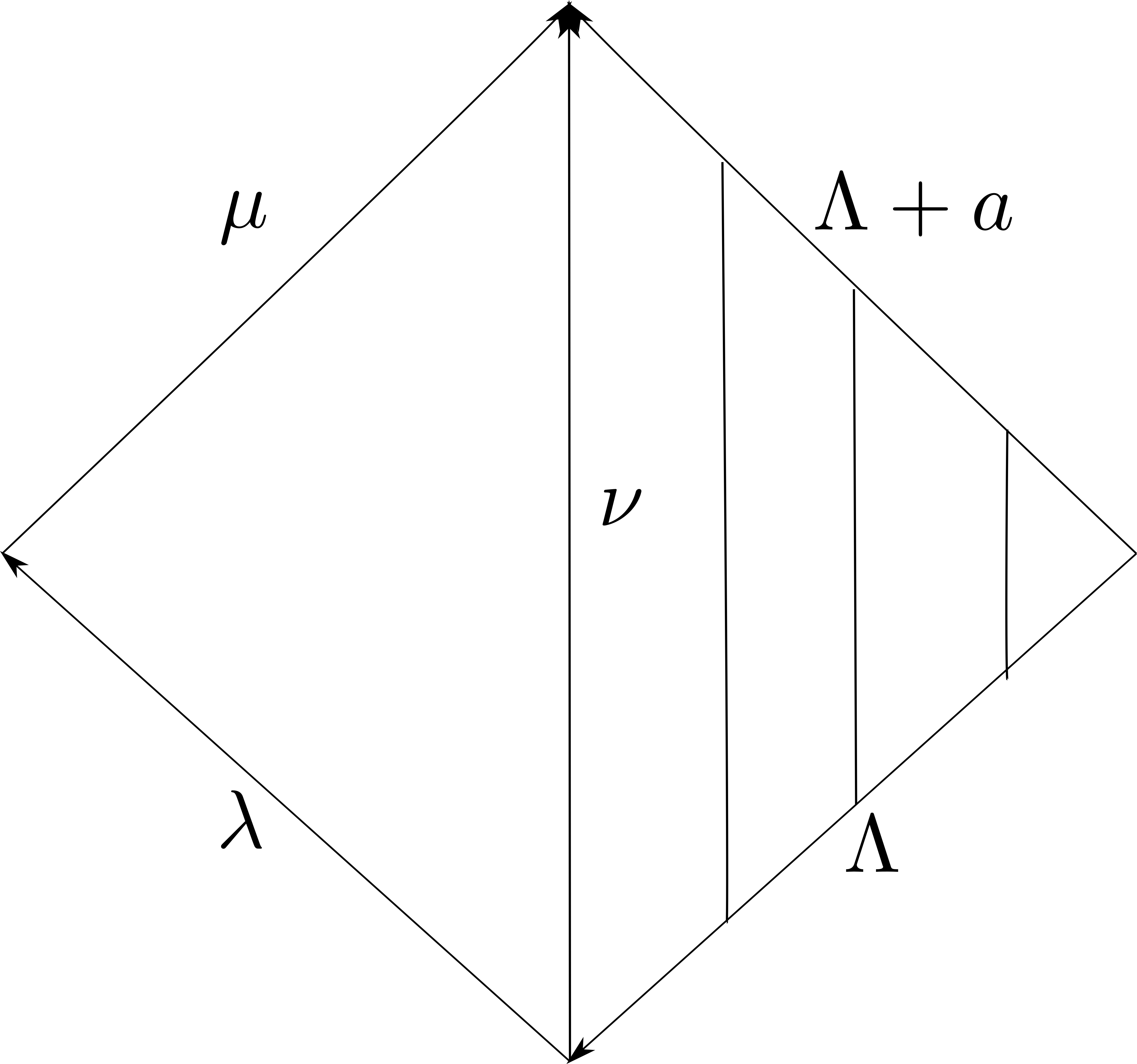}
\caption{A schematic depiction of an augmented hive in $\AHIVE_{\diag(\lambda \boxplus \mu \rel \nu) \rel a}$, where we artificially shift by a tuple $\Lambda$ with large gaps in order to create two hives, instead of a hive and a Gelfand--Tsetlin pattern.}\label{fig:augment}
\end{center}
\end{figure}

As noted in \cite{NST}, as just one illustration of the power of this characterization, we observe as an immediate corollary using Pr\'ekopa's theorem \cite{prekopa} (or the Pr\'ekopa--Leindler inequality, or the Brunn--Minkowski inequality) that the Horn probability measure is log-concave for any $\lambda,\mu \in \Spec^\circ$.

To summarize the discussion so far: the eigenvalues of the sum of two independent, uniformly distributed Hermitian matrices with eigenvalues $\lambda,\mu \in \Spec^\circ$ respectively, has the distribution of a linear projection of the uniform probability measure on the augmented hive polytope $\AHIVE_{\diag(\lambda \boxplus \mu \rel \ast) \rel \ast}$, which can be computed to be an $n(n-1)$-dimensional polytope.

\subsection{Large deviations for hives}

In principle, the behavior of the Horn probability measures in the limit $n \to \infty$ is governed by the theory of free probability: if the empirical distributions of $\la = \lambda^{(n)}, \mu = \mu^{(n)}$ converge (in an appropriate sense) to probability measures $\sigma, \sigma'$, then the empirical distribution of $\nu$ (drawn from the Horn probability measure) should similarly converge to the free convolution $\sigma  \boxplus \sigma'$; see for instance the seminal paper \cite{Voi} for some rigorous results in this direction.  Relating to this, results have emerged in recent years establishing large deviation inequalities for the Horn probability measure under suitable hypotheses on $\lambda, \mu$: see \cite{bgh}, \cite{NarSheff}.

The question of understanding the spectrum of $X_n + Y_n,$ in the setting of large deviations was studied  in \cite{bgh}, where upper and lower large deviation bounds were given which agreed for measures of a certain class that correspond to ``free products with  amalgamation" (see Theorem 1.3, \cite{bgh}). 

Suppose $\lambda^{cont}:[0,1]\rightarrow \R$ and $\mu^{cont}:[0,1]\rightarrow \R$ are $C^1$, strongly decreasing functions. Let $\lambda^{(n)}$ and $\mu^{(n)}$ be obtained from it by taking the slopes of the respective piecewise linear approximations to $\lambda^{cont}$ and $\mu^{cont}$, where the number of pieces is $n$.   In \cite[Theorem 8]{NarSheff}, a large deviation principle was obtained for the probability measure (as $n \ra \infty$) of the piecewise linear extension of $\frac{h_n}{n^2}$ to $T$, 
where $(h_n, \gamma_n)$ is an augmented hive sampled uniformly at random. 

The rate function $I_1$ from  \cite[Notation 1]{NarSheff} for a continuum hives $h^{cont}$ defined on a continuous triangle $T^{cont}$ is roughly, the integral of a certain surface tension function $\sigma$,  added to a term arising from the negative of the limiting value of the logarithm of a Vandermonde determinant associated with $h^{cont}$ restricted to the third side.  Given a discrete hive $h_n$, let $h' = \iota_n(h_n)$  be a rescaling of the  piecewise linear extension of $h_n$ that is a map from $T^{cont}$ to $\R$. 
Let $h_n$ be sampled from the measure on $ \HIVE_{\lambda^{(n)} \boxplus \mu^{(n)} \rel \ast}$  obtained by pushing forward the uniform measure on $\AHIVE_{\diag(\lambda^{(n)} \boxplus \mu^{(n)} \rel \ast) \rel \ast}$ via the natural map from $\AHIVE_{\diag(\lambda^{(n)} \boxplus \mu^{(n)} \rel \ast) \rel \ast}$ to $ \HIVE_{\lambda^{(n)} \boxplus \mu^{(n)} \rel \ast}$.
For any Borel set $E\subset L^\infty(T^{cont})$, let $\PP_n(E) := \p_n[\iota_n(h_n) \in E].$
For each Borel measurable set 
${\displaystyle E\subset L^\infty(T^{cont})}$,  it is shown in \cite[Theorem 8]{NarSheff} that 
$${\displaystyle -\inf_{h\in E^{\circ }}I_1(h)\leq \liminf_{n \ra \infty}\frac{2}{n^2} \log(\mathcal {P}_{n}(E))\leq \limsup_{n \ra \infty}\frac{2}{n^2}\log(\mathcal {P}_{n}(E))\leq -\inf_{h\in {\overline {E}}}I_1(h).}$$
It has not yet been proven that the rate function minimizer for the above large deviation principle is unique (although such uniqueness would be implied by the strict convexity of $\sigma$).   Further,  it has not been established  that the Lebesgue measure on  the polytope $\AHIVE_{\diag(\lambda^{(n)} \boxplus \mu^{(n)} \rel \ast) \rel \ast}$ exhibits concentration,  although this would be a consequence of the rate function minimizer being unique; however for GUE boundary data, concentration has been established in \cite{NST} as mentioned in Subsection~\ref{ssec:GUEhives}.

\subsection{Concentration and convergence of GUE hives}\lab{ssec:GUEhives}

\begin{figure}\label{fig:factorize}
\begin{center}
\includegraphics[scale=0.5, angle = 45]{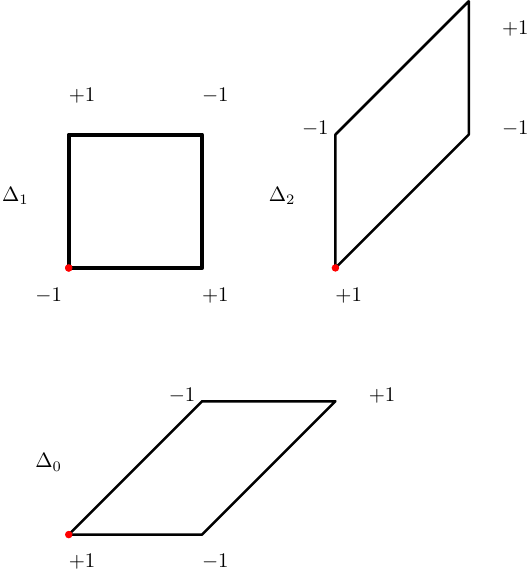}
\caption{A visual depiction of the the second order discrete operators $\De_i$.  A red dot indicates the vertex $(v_1, v_2)$.}
\end{center}
\end{figure}

To establish normalization conventions, we define a GUE random matrix to be a random Hermitian matrix $M = (\xi_{ij})_{1 \leq i,j \leq n}$ where $\xi_{ij} = \overline{\xi_{ji}}$ for $i<j$ are independent complex gaussians of mean zero and variance $1$, $\xi_{ii}$ are independent real gaussians of mean zero and variance $1$, independent of the $\xi_{ij}$ for $i<j$.  As is well known (see e.g., \cite{mehta}), if $\sigma>0$ and $A$ is a random matrix with $\frac{A}{\sqrt{\sigma^2 n}}$ drawn from the GUE ensemble, then the eigenvalues $\lambda \in \Spec$ of $A$ are distributed with probability density function
\begin{equation}\label{density}
 C_n \sigma^{-n^2} \exp\left( - \frac{|\lambda|^2}{2\sigma^2 n} \right) V(\lambda)^2
\end{equation}
for some constant $C_n>0$ depending only on $n$.  In particular, $\lambda$ will lie in $\Spec_n^\circ$ almost surely. From this the previous discussion, we see that if $\var_\lambda, \var_\mu > 0$ are fixed\footnote{In particular, we allow implied constants in the $O()$ notation to depend on these quantities.} and $A, B$ are independent random matrices with\footnote{This normalization is chosen so that the mean eigenvalue gaps of $A,B$ comparable to $1$ in the bulk of the spectrum.} $\frac{A}{\sqrt{\var_\la n}}, \frac{B}{\sqrt{\var_\mu n}}$ drawn from the GUE ensemble, then the the distribution of the eigenvalues of $A+B$ are the pushforward of the measure 
on the $n(n+1)$-dimensional \emph{augmented hive cone} 
$$\AHIVE_{\diag(\ast \boxplus \ast \rel \ast) \rel \ast} \coloneqq \bigcup_{\lambda,\mu,\nu,\pi} (\HIVE_{\lambda \boxplus \mu \rel \nu} \times \GT_{\diag(\nu) \rel \pi}),$$ 
where the probability density function of this measure $\p_n$ is given by
\begin{equation}\label{css}
 C_{n,\sigma_\lambda,\sigma_\mu} \exp\left( - \frac{|\lambda|^2}{2\var_\la n} - \frac{|\mu|^2}{2\var_\mu n} \right) V(\lambda) V(\mu)
\end{equation}
on the slices
\begin{equation}\label{ahive-slice}
\AHIVE_{\diag(\lambda \boxplus \mu \rel \ast) \rel \ast} \coloneqq \bigcup_{\nu,\pi} (\HIVE_{\lambda \boxplus \mu \rel \nu} \times \GT_{\diag(\nu) \rel \pi})
\end{equation}
and where $C_{n,\sigma_\lambda,\sigma_\mu} > 0$ is chosen to make this measure a probability measure.

Since GUE matrices have an operator norm of $O(\sqrt{n})$ with overwhelming probability (by which we mean with probability $1-O(n^{-C})$ for any fixed $C>0$), the boundary differences $\lambda, \mu, \nu$ of an augmented hive $(h,\gamma)$ drawn from the above measure will be of size $O(n)$ with overwhelming probability, and hence the entries $h(v), v \in T$ of the hive will be of size $O(n^2)$ with overwhelming probability.  From this fact (and some crude moment estimates to treat the contribution of the exceptional event), it is not difficult to show the ``trivial bound'' that the variance $\Var h(v)$ of any individual entry $h(v)$ of the hive is bounded by $O(n^4)$.

\noindent The main theorem of \cite{NST} is the following.
\begin{theorem}\lab{theorem:final3}
Let $\la_n:= (\la_{1, n}, \dots, \la_{n, n})$ and $\mu_n :=(\mu_{1, n}, \dots, \mu_{n, n})$ be respectively the eigenvalues of two independent random matrices $X_n$ and $Y_n$ such that $X_n/( \sqrt{\var_\la n})$ and $Y_n/( \sqrt{\var_\mu n})$ have the distribution of a GUE.
Let $a_n$ be a random sample from the normalized Lebesgue measure on $\AHIVE_{\diag(\lambda^{(n)} \boxplus \mu^{(n)} \rel \ast) \rel \ast}.$  Then,  for all $\eps > 0$, there exists an $n_0(\eps) \in \N$,  such that for all $n > n_0(\eps)$,  every coordinate  $(a_n)_{ij}$ of $a_n$,  satisfies  $\var (a_n)_{ij}  < \eps^2 n^4.$
\end{theorem}
Note that the typical value of a boundary coordinate has magnitude bounded below by $cn^2.$ By the concavity of hives  the same is true for the value at a typical interior vertex.

A random variable $Z$ in $\R$ that satisfies for some positive real $K$,  $$\E[\exp(|Z|/K)] \leq 2$$ is called subexponential.  Following  (Definition 2.7.5, \cite{Ver}), the subexponential (or $\psi_1$) norm is defined as follows.

\begin{definition}[$\psi_1$ norm]\lab{def:psi1}
$\|Z\|_{\psi_1}$ is defined to be $\inf\{ t > 0: \E[\exp(|Z|/t)] \leq 2\}.$
\end{definition}

It follows from  Theorem~\ref{theorem:prekopa} and the fact that $\p_n$ is log-concave, that $\left((a_n)_{ij} - \E_\P (a_n)_{ij}\right)$ has a log-concave density and is therefore subexponential. Further, we see that the standard deviation of a random variable with a log-concave density is within multiplicative positive universal constants, the same as its $\psi_1$ norm.   Thus, we have from Theorem~\ref{theorem:final3}, the following.
 
 \begin{corollary}Let $\eps$ be greater than $0$. Then, there exists a positive integer $n_0(\eps)$ depending on $\eps,$ such that 
 for all $n > n_0(\eps)$,  $$\left\|(a_n)_{ij} - \E_n(a_n)_{ij}\right\|_{\psi_1} < \eps n^2.$$\end{corollary}
In the setting of periodic hives, results with related statements were obtained in \cite{random_concave}.

The main theorem of this paper is Theorem~\ref{theorem:main}, which states the following.\\\\
\noindent{\bf Main Theorem:} Let $\la_n:= (\la_{1, n}, \dots, \la_{n, n})$ and $\mu_n :=(\mu_{1, n}, \dots, \mu_{n, n})$ be respectively the eigenvalues of two independent random matrices $X_n$ and $Y_n$ such that $X_n/( \sqrt{\var_\la n})$ and $Y_n/( \sqrt{\var_\mu n})$ have the distribution of a GUE.  
Let $a_n$ be a random sample from the normalized Lebesgue measure on $\AHIVE_{\diag(\lambda^{(n)} \boxplus \mu^{(n)} \rel \ast) \rel \ast}.$ Let $\lim_{n \ra \infty} v_n/n = v,$ where $v_n$ is a vertex (of the hive having $\la^{(n)}$ and $\mu^{(n)}$ as sides) of $\AHIVE_{\diag(\lambda^{(n)} \boxplus \mu^{(n)} \rel \ast) \rel \ast}.$ Then,
$\lim_{n \ra \infty}n^{-2} \E_n h_n(v_n)$ exists and equals  $\sup\limits_{f^\sharp \in \AHT_v^\infty} S_{v}(f^\sharp).$

Here $\AHT_v^\infty$ is the set of asymptotic height functions (see Definition~\ref{def:AHT}) for $\hexagon_v^\infty$, and $S_v$ is a certain functional from Definition~\ref{def:Sv}.

In the sequel, we will frequently use the notation $o_\eps(f(n))$ to denote a  quantity $A$, such that as $\eps \ra 0$, $\frac{|A|}{f(n)} \ra 0$.
We will use $c, C,$ etc to denote absolute constants, and $C_{|\rho|}, $ for example, to denote a constant controlled by an increasing function of $|\rho|$ alone.

 \section{Lozenge tilings of $\hexagon_v^n$}\label{recurrence-sec}
We refer the reader to  \cite[\S 3]{NST} for further background, but reproduce some of the material from there that we will need.

\begin{definition}[Lozenges and border triangles]  A \emph{lozenge} is a quadruple $ABCD$ in $U$ or $U'$ that is one of following three forms for some $i,j \in \Z$: (For the coordinates, see Figure~\ref{fig:typical}.)
\begin{itemize}
\item[(i)] $(A,B,C,D) = ((i,j), (i+1,j-1), (i+2,j-1), (i+1,j))$
\item[(ii)] $(A,B,C,D) = ((i,j), (i,j+1), (i-1,j+2), (i-1,j+1))$
\item[(iii)] $(A,B,C,D) = ((i,j), (i+1,j), (i+1,j+1), (i,j+1))$.
\end{itemize}
Lozenges of type (i) will be called \emph{blue} if they lie in the upper triangle $U$ and \emph{red} if they lie in the lower triangle $U'$; lozenges of type (ii) will be called \emph{red} if they lie in $U$ and \emph{blue} if they lie in $U'$; and lozenges of type (iii) that lie either in $U$ or in $U'$ will be called \emph{green}; see Figure \ref{fig:typical}.  A quadruple of the form (iii) that crosses the diagonal separating $U$ and $U'$ is \emph{not} considered to be a lozenge, but instead splits into two border triangles as defined below.  (The colors of lozenges will not be needed immediately, but will play a useful role later in this section.)

A \emph{border edge} is an edge $AC$ of the form $(A,C) = ((i,n-i), (i+1,n-i-1))$ for some $0 \leq i < n$; the border edges thus separate $U$ and $U'$.  Each border edge $(A,C) = ((i,n-i), (i+1,n-i-1))$ is bordered by two \emph{border triangles} $ABC$, defined as follows:
\begin{itemize}
\item (Upward triangle) $(A,B,C) = ((i,n-i), (i+1,n-i), (i+1,n-i-1))$.
\item (Downward triangle) $(A,B,C) = ((i,n-i), (i,n-i-1), (i+1,n-i-1))$.
\end{itemize}
Note that upward triangles lie (barely) in $U$, while downward triangles lie (barely) in $U'$.

Given a lozenge $\edge = ABCD$ and a function $\tilde k \colon \{0,\dots,n\}^2 \to \R$ defined as before, we define the \emph{weight} $\weight(\edge) = \weight(\edge, \tilde k)$ to be the quantity
$$\weight(\edge) \coloneqq \frac{1}{3} (\tilde k(A) + \tilde k(C) - \tilde k(B) - \tilde k(D)).$$
Similarly, given a border triangle $\Delta = ABC$, the weight $\weight(\Delta) = \weight(\Delta,  \tilde k)$ is defined as
$$\weight(\Delta) \coloneqq \frac{1}{3} (\tilde k(B) -  \tilde k(A)).$$ \end{definition}

\begin{definition}[Octahedron recurrence]\label{octa-def}  If $v = (i,j)$ lies in the interior of $\{0,\dots,n\}^2 = T \cup T'$, then the \emph{excavation hexagon}
$\hexagon_v^n = ABCDEF$ in $\{0,\dots,n\}^2 = U \cup U'$ centered at $v$ is defined as follows:
\begin{itemize}
\item If $v \in T$ (i.e., $i \leq j$), then
$$ (A,B,C,D,E,F) = ((0,n), (0,j), (i,j-i), (n+i-j,j-i), (n+i-j,j), (i,n)).$$
\item If $v \in T'$ (i.e., $i \geq j$), then
$$ (A,B,C,D,E,F) = ((i-j, n+j-i), (i-j,j), (i,0), (n,0), (n,j), (i,n+j-i)).$$
\end{itemize}
Note that these two definitions agree when $v \in T \cap T'$ (i.e., when $i=j$). The original point $v = (i,j)$ is then the intersection of the diagonals $BE$ and $CF$.  The line $AD$ is called the \emph{equator}; it lies on the border between $U$ and $U'$.  The \emph{weight} $\weight(\hexagon_v^n) = \weight(\hexagon_v^n,\tilde k)$ of this hexagon is defined as
\begin{equation}\label{hexagon-weight}
 \weight(\hexagon_v^n) \coloneqq \frac{1}{3} (\tilde k(B) + \tilde k(C) - \tilde k(D) + \tilde k(E) + \tilde k(F)).
\end{equation}

A \emph{lozenge tiling} $\Xi$ of the excavation hexagon $\hexagon_v^n$ is a partition of the (solid) hexagon into (solid) lozenges and (solid) border triangles, such that each border edge on the equator is adjacent to exactly one border triangle in the tiling; see Figure \ref{fig:typical}.  An example of a lozenge tiling is the \emph{standard lozenge tiling} $\Xi_0$, in which the trapezoid $ABEF$ is tiled by blue lozenges in $U$ and by green lozenges and downward border triangles in $U'$, while the opposite trapezoid $BCDE$ is tiled by green lozenges and upward border triangles in $U$ and by red lozenges in $U'$; see Figure \ref{fig:standard}.

The \emph{weight} $w_\Xi = w_\Xi(\tilde k)$ of such a tiling is defined to be the sum of the weights of all the lozenges $\edge$ and triangles $\Delta$ in the tiling, as well as the weight of the entire hexagon $\hexagon_v^n$:
\begin{equation}\label{weight-form}
w_\Xi \coloneqq \sum_{\edge \in \Xi} \weight(\edge) + \sum_{\Delta \in \Xi} \weight(\Delta) + \weight(\hexagon_v^n).
\end{equation}
Note that the $w_\Xi$ depend linearly on $\tilde k$, and hence on $k, k'$.  We then define
$$ \tilde h(v) \coloneqq \max_{\Xi\ \mathrm{tiles}\ \hexagon_v^n} w_\Xi.$$
\end{definition}

\begin{definition}
Let $v = (i, j) \in [0, 1]^2$ equal $\lim_{n \ra \infty} \frac{v_n}{n}$ where $v_n \in \{0, \dots, n\}^2$. We then define $$\hexagon_v^\infty := \lim_{n \ra \infty} (\frac{1}{n}) \hexagon_{v_n}^n.$$  
\end{definition}
As remarked on  \cite[page 15]{NST},  using Speyer's Theorem \cite{Speyer} and the volume preserving nature of the octahedron recurrence, it can be seen that when $k$ and $k'$ are hives arising via Proposition~\ref{gt-rem}(iv) from random Gelfand-Tsetlin patterns sampled independently from two scaled GUE minor processes, the distribution of $\tilde h(v)$ is exactly the same as the distribution of the value of the random augmented hive (from Theorem~\ref{theorem:final3}) at $v$.

\begin{figure}
\begin{center}
\includegraphics[scale=0.40]{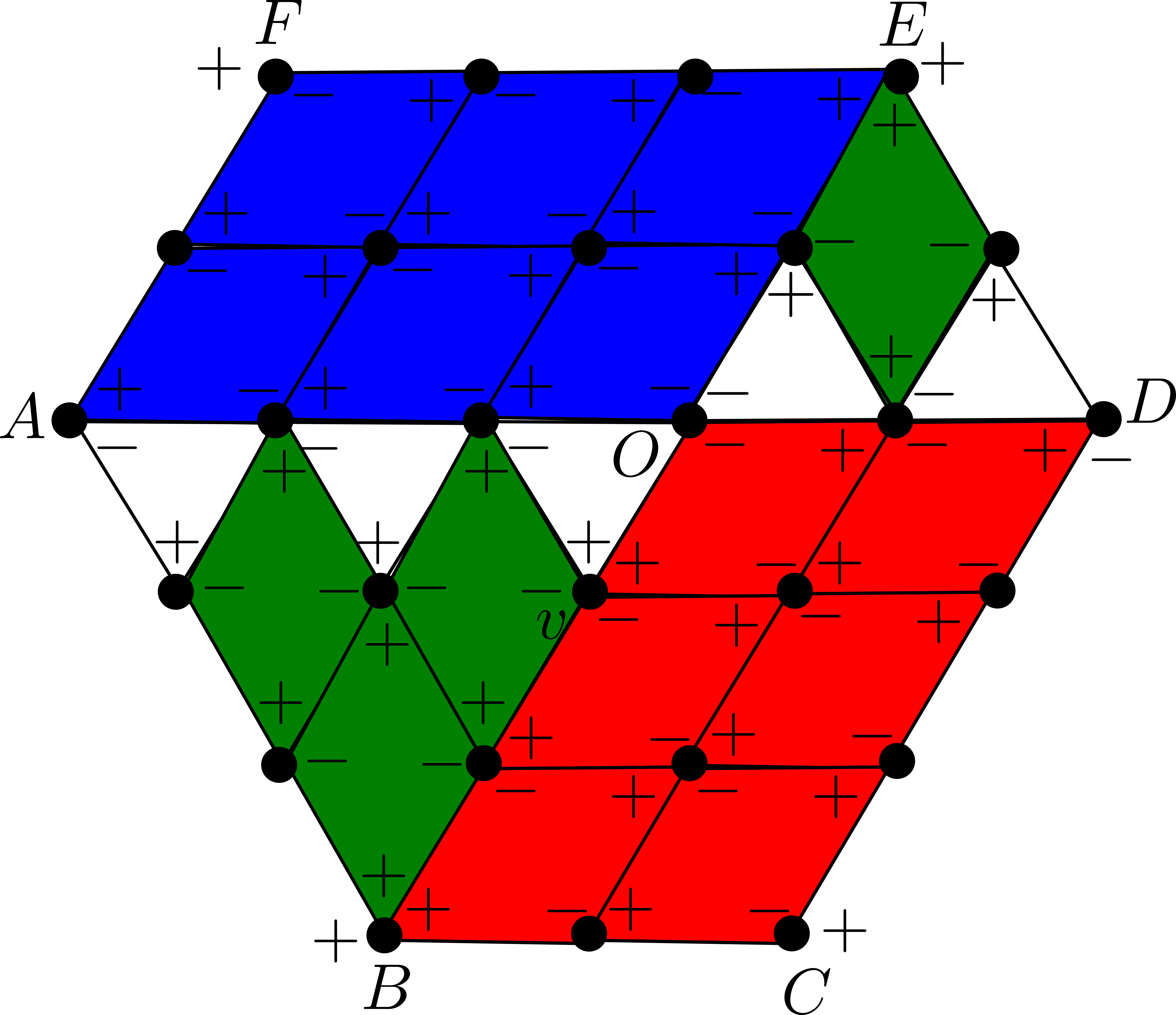}
\caption{The standard lozenge tiling of a hexagon $ABCDEF$ centered at $v$.  The total weight of this tiling is $\tilde k(E) + \tilde k(B) - \tilde k(O)$, where $O$ is the intersection of the diagonal $BE$ with the equator $AD$.}\label{fig:standard}
\end{center}
\end{figure}

\begin{figure}
\begin{center}
\includegraphics[scale=0.40]{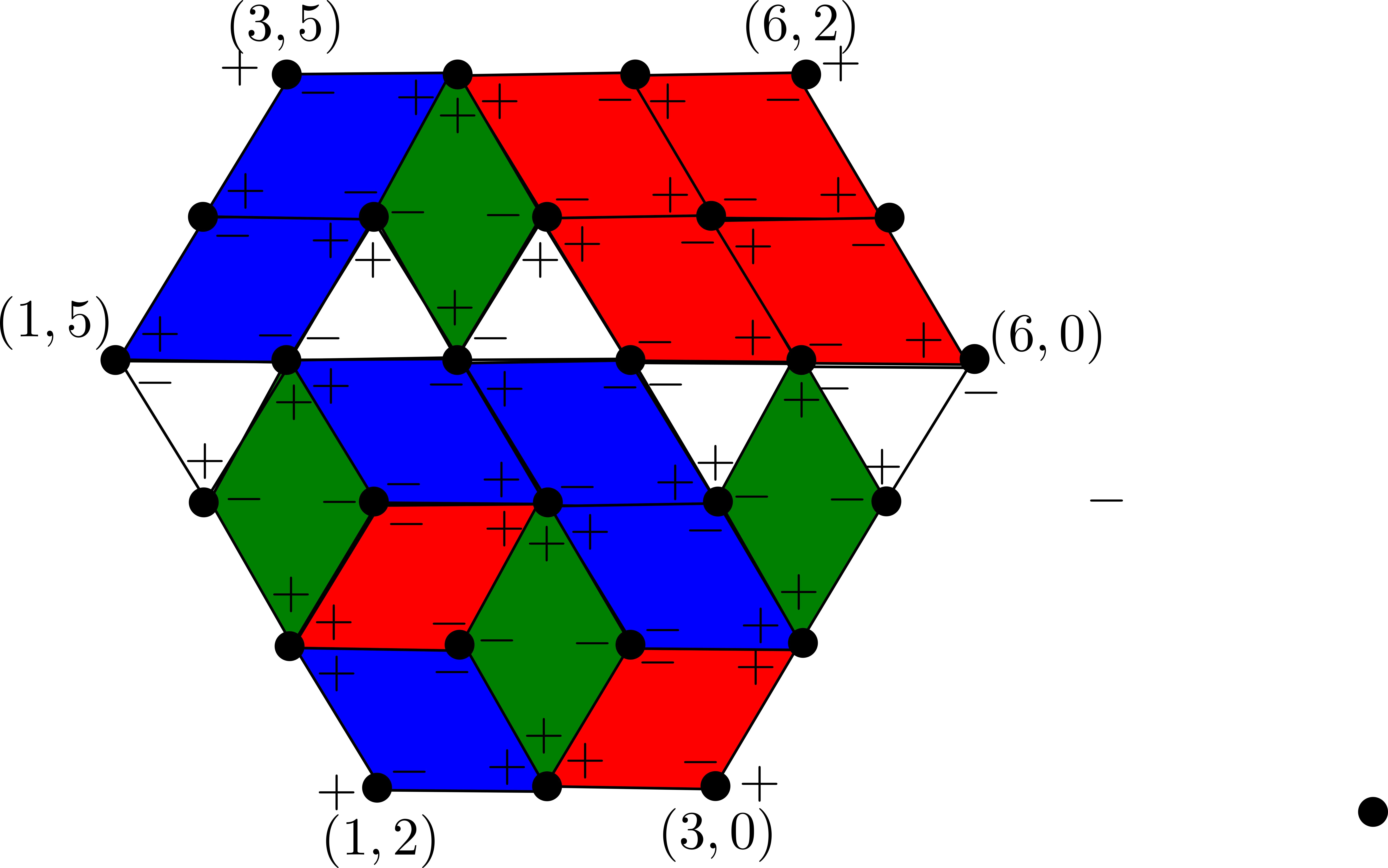}
\caption{A typical lozenge tiling of $\hexagon_{(3,2)}$, $n=6$. }\label{fig:typical}
\end{center}
\end{figure}

It was shown in \cite{NST} that we can also write the weight form $w_\Xi$ in \eqref{weight-form} in a red lozenge-avoiding form as
\begin{equation}\label{alt-weight}
w_\Xi \coloneqq 2 \sum_{\edge \in \Xi\text{, blue}} \weight(\edge) + \sum_{\edge \in \Xi\text{, green}} \weight(\edge) 
+ \sum_{\Delta \in \Xi} \weight(\Delta) + \weight'(\hexagon_v^n)
\end{equation}
where the modified weight $\weight'(\hexagon_v)$ of the hexagon $\hexagon_v$ is defined by the formula
$$
 \weight'(\hexagon_v^n) \coloneqq \frac{1}{3} (-\tilde k(A) + 2\tilde k(B) + 2\tilde k(F)).
$$
   
  \begin{figure}[ht]
    \centering
    \includegraphics[scale=0.40]{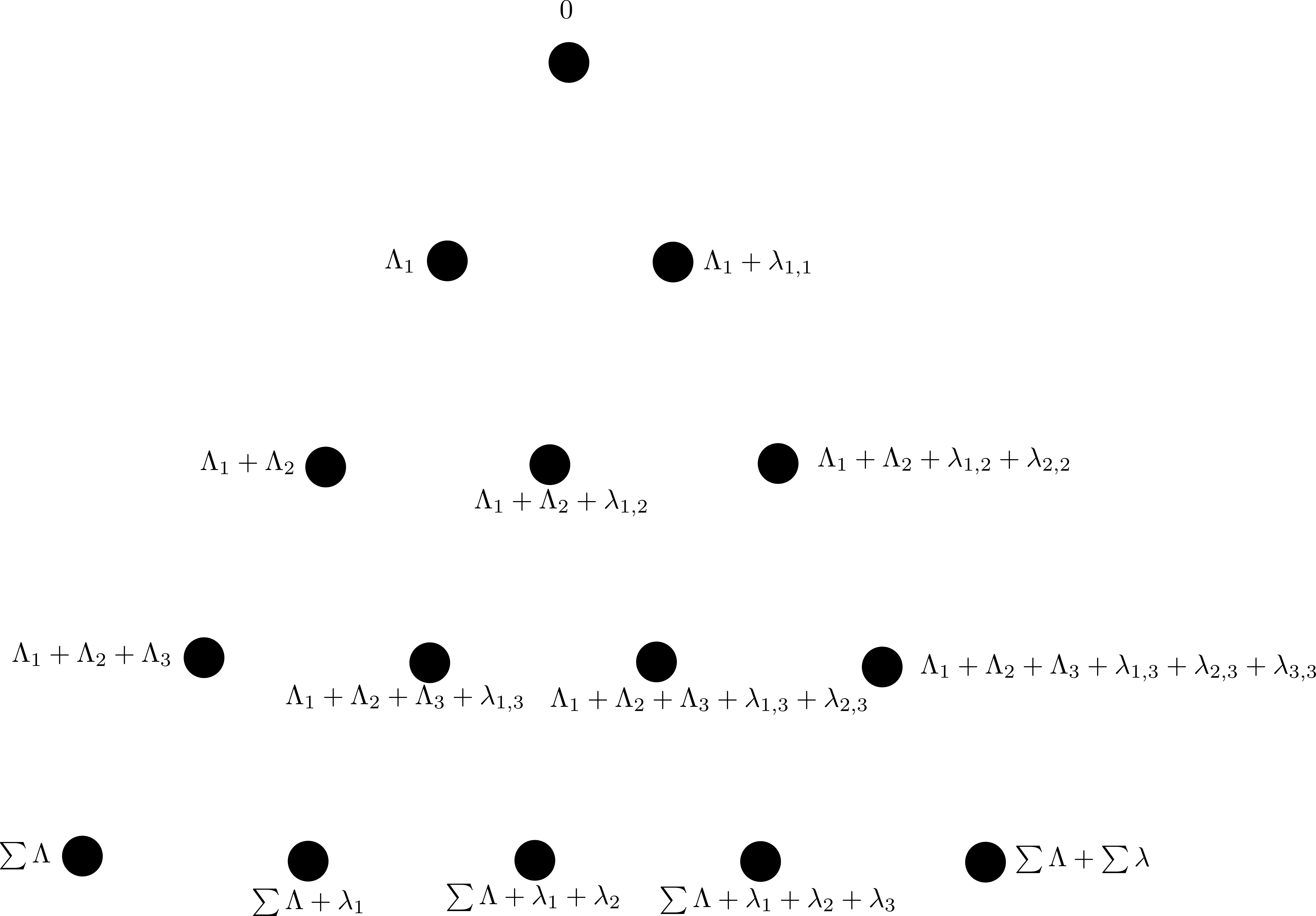}
    \caption{The hive associated with the Gelfand--Tsetlin pattern in Figure \ref{fig:gt} and some large gap tuple $\Lambda$. Note that for $1 \leq i \leq 4$, $\la_i = \la_{i, 4}$.}
    \label{fig:gt-hive}
  \end{figure}
  
  \begin{figure}[ht]
    \centering
    \includegraphics[scale=0.40, angle=270]{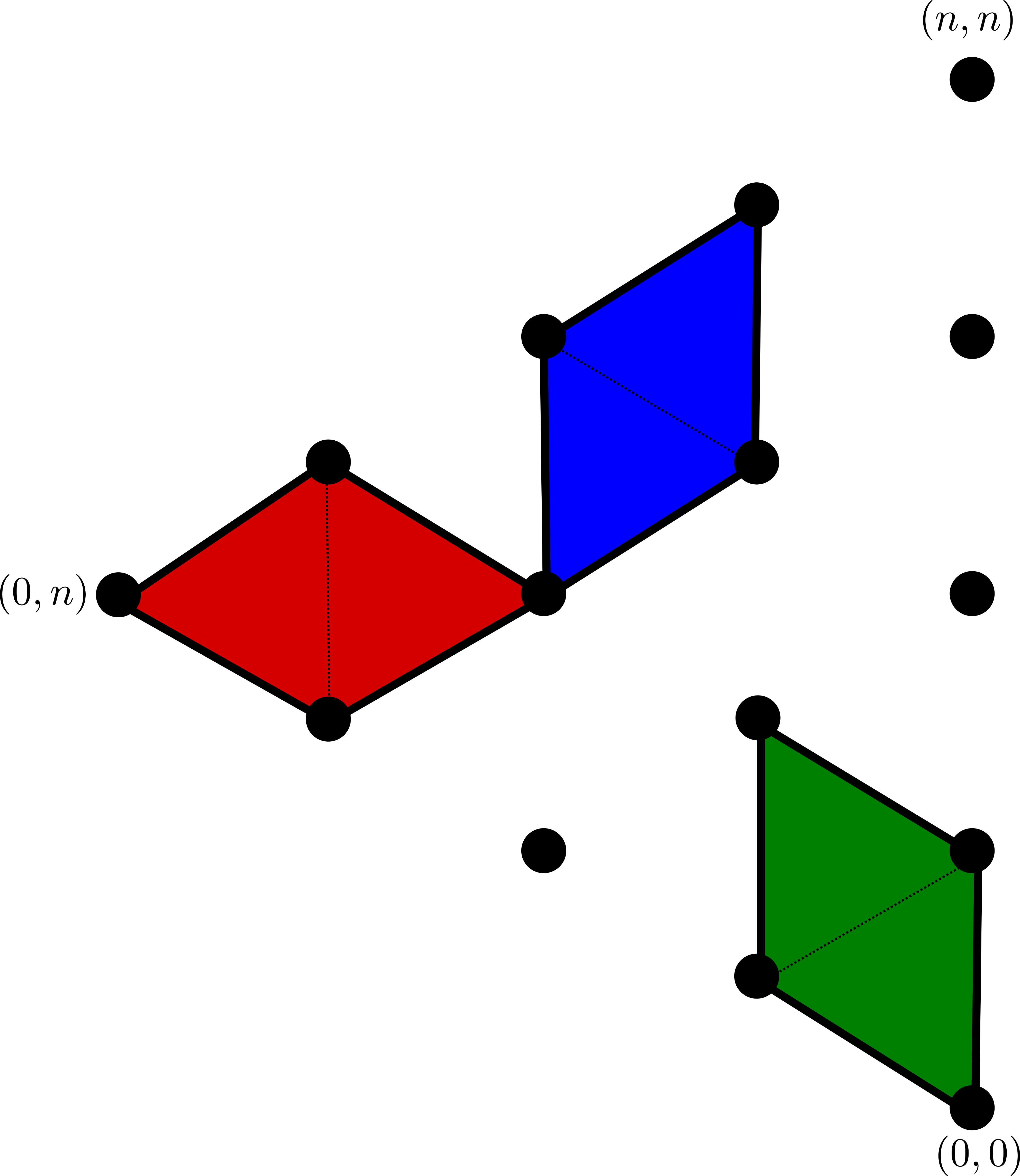}
    \caption{The weights of the green, blue, and red lozenges are, respectively, $\frac{1}{3}\left(\la_{1, 3} - \la_{1, 4}\right)$, $\frac{1}{3}\left(\la_{3, 3} - \la_{2, 2}\right)$, and $\frac{1}{3}\left(-\Lambda_1 + \Lambda_2 - \la_{1,1} + \la_{1, 2}\right)$.}
    \label{fig:rhombus}
  \end{figure}
  
\begin{figure}
  \begin{center}
  \includegraphics[scale=0.40]{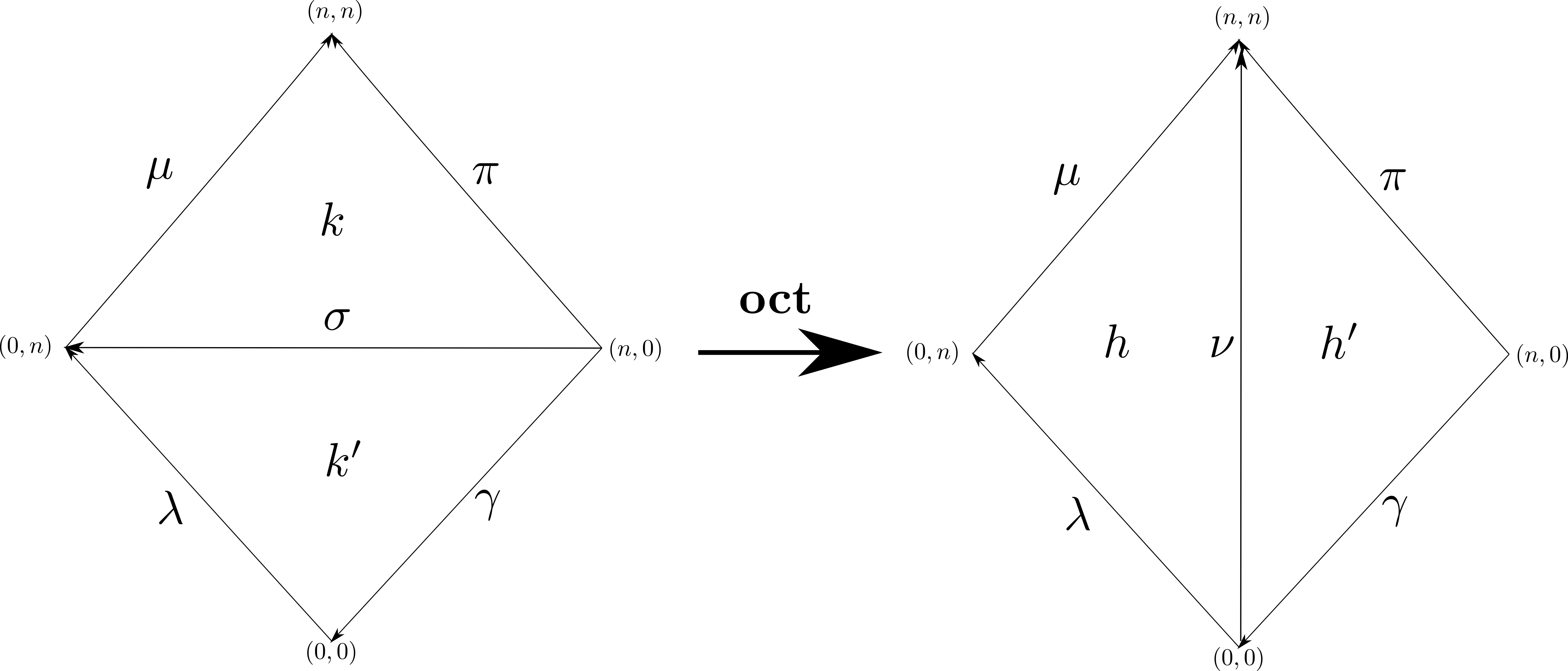}
  \caption{A schematic depiction from \cite{NST} of the octahedron recurrence that transforms one pair $(k,k') $ of hives into another $(h,h')$.  The hives $h,h',k,k'$ have been shifted to lie on triangles $T, T', U, U'$ respectively. }\label{fig:octahedron}
  \end{center}
  \end{figure}

\begin{observation}
In the lower trapezoid (contained in $k'$ in Figure~\ref{fig:octahedron}),
for a green lozenge $\edge$ that has labels $\la_{i, j}$ and $\la_{i, j+1}$ on two opposite sides, $$\wt(\edge) :=  \frac{1}{3}\left(\la_{i, j} - \la_{i, j+1}\right),$$ and a blue lozenge  $\edge$ that has labels $\la_{i, j}$ and $\la_{i-1, j-1}$ on two opposite sides, $$\wt(\edge) :=  \frac{1}{3}\left(\la_{i, j} - \la_{i-1, j-1}\right).$$
In the upper trapezoid (contained in $k$ in Figure~\ref{fig:octahedron}),
for a blue lozenge $\edge$ that has labels $\mu_{i, j}$ and $\mu_{i, j+1}$ on two opposite sides, $$\wt(\edge) :=  \frac{1}{3}\left(\mu_{i, j} - \mu_{i, j+1}\right),$$ and a green lozenge  $\edge$ that has labels $\mu_{i, j}$ and $\mu_{i-1, j-1}$ on two opposite sides, $$\wt(\edge) :=  \frac{1}{3}\left(\mu_{i, j} - \mu_{i-1, j-1}\right).$$
\end{observation}


  \begin{figure}[!h]
    \centering
    \includegraphics[width=0.1\linewidth]{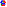}\\
    \caption{A random lozenge tiling for $n=4$, \cite{Gang}}
    \label{fig:lozenge_4}
  \end{figure}
  
  \begin{figure}[!h]
    \centering
    \includegraphics[width=0.15\linewidth]{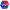}\\
    \caption{A random lozenge Tiling for $n=6$, \cite{Gang}}
    \label{fig:lozenge_6}
  \end{figure}
  
  \begin{figure}[!h]
    \centering
    \includegraphics[width=0.25\linewidth]{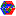}\\
    \caption{A random lozenge Tiling for $n=10$, \cite{Gang}}
    \label{fig:lozenge_10}
  \end{figure}
  
  \begin{figure}[!h]
    \centering
    \includegraphics[width=0.5\linewidth]{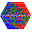}\\
    \caption{A random lozenge tiling of $\hexagon_{(10,10)}^{20}
    $, $n=20$, \cite{Gang}.}
    \label{fig:lozenge_20}
  \end{figure}

  \begin{figure}[!h]
    \centering
    \includegraphics[width=0.5\linewidth]{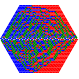}\\
    \caption{A random lozenge tiling for $n=50$, \cite{Gang}}
    \label{fig:lozenge_50}
  \end{figure}

  \begin{figure}[!h]
    \centering
    \includegraphics[width=0.5\linewidth]{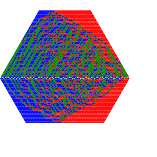}\\
    \caption{A random lozenge tiling for $n=100$, \cite{Gang}}
    \label{fig:lozenge_100}
  \end{figure}

  \begin{figure}[!h]
    \centering
    \includegraphics[width=0.5\linewidth]{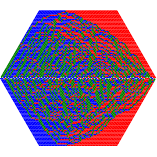}\\
    \caption{A random lozenge tiling of $\hexagon_{(50,50)}^{100}$, $n=100$, \cite{Gang}}
    \label{fig:lozenge_100_1}
  \end{figure}

  \section{Height functions of lozenge tilings}
  
  \begin{figure}[ht]
      \centering
      \begin{minipage}[b]{0.45\textwidth}
              \centering    
  \includegraphics[width=0.63\linewidth, height = 1\linewidth]{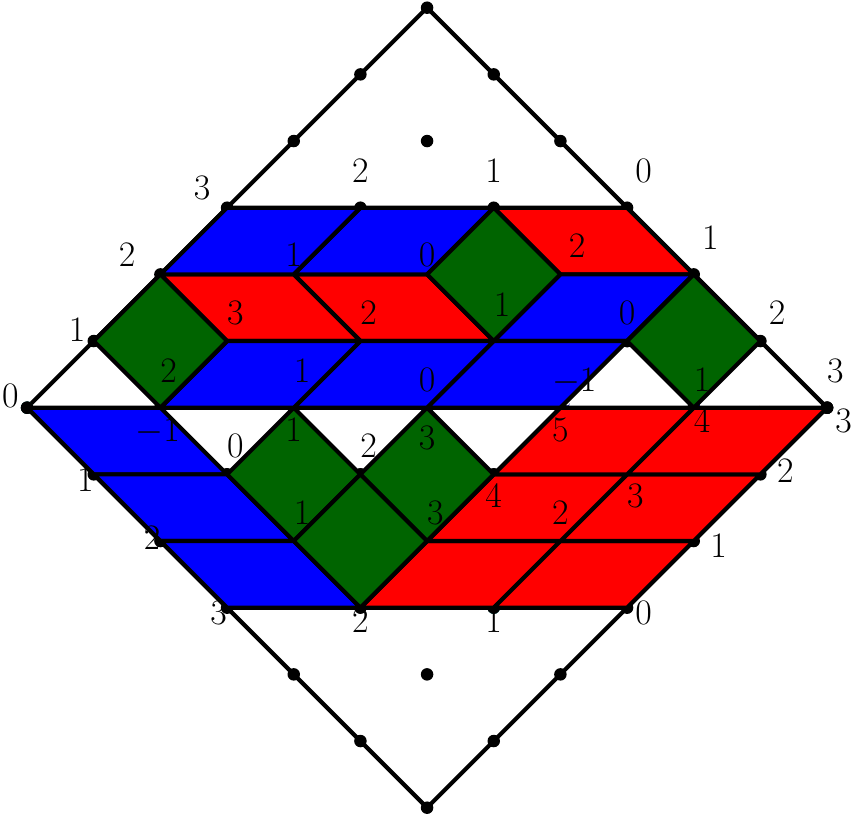}
  \caption{A typical lozenge tiling of $\hexagon_{(3,3)}^6$, with the height functions $f_{6,\mathtt{up}}$ and $f_{6,\mathtt{lo}}$ on the upper and lower trapezoids. More generally, in $\hexagon^n_{(a, b)}$ the westmost  corner in the upper trapezoid gets the height $2b-a$ while the same point, viewed as belonging to the lower trapezoid, gets a height $2a-b$.}\label{fig:typical-hex(3,3)}
            \end{minipage}
      \hfill
      \begin{minipage}[b]{0.30\textwidth}
          \centering
          \includegraphics[width=\textwidth]{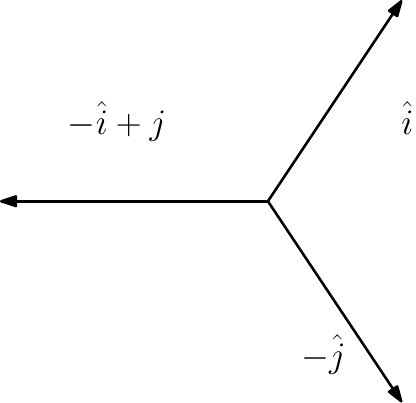} 
          \caption{positive directions}
          \label{fig:image2}
      \end{minipage}
      \label{fig:ij}
  \end{figure}
  
  Let $\hat{i}$ and $\hat{j}$ denote the unit vectors along the $x$ and $y$ axes respectively (adjusted to be unit vectors in the triangular lattice which we will identify with $\Z^2$; see Figure~\ref{fig:typical-hex(3,3)}).    Let the following be called unit vectors in ``positive directions'':  $\{\hat{i}, - \hat{j}, - \hat{i} + \hat{j}\}$.   Let positive multiples of these vectors be called positive vectors. 
  
  Let $\mathcal{R}$ be a domain whose boundary $\partial \RR$ is a piecewise linear non-self-intersecting curve, whose linear pieces are parallel to positive directions and have end points that are lattice points. 
  A piecewise linear path in $\mathcal R$ is called positively oriented if each constituent directed piece is along a positive direction.
  
  For two points $u, v \in \mathcal{R}$, let the asymmetric distance function $d_{\RR}(u, v)$ be defined as (see page 10 of \cite{GorinBook} for further discussion)  the minimal total length over all positively oriented paths from $u$ to $v$ staying within (or on the boundary) of $\mathcal R$.
  If all the piecewise linear segments of $\partial \mathcal R$ have endpoints in the triangular lattice $\L$, we call $\mathcal R$ a {\bf lattice domain}. If $\RR$ is also simply connected, we call $\RR$ a {\bf simply connected lattice domain}.
  
  \begin{definition}[Height function for a simply connected lattice domain]\lab{def:Htfn}
  We say that an integer valued Lipschitz function $f$ on ${\mathcal R}$ is a  height function if \ben \item $f(u') - f(u)$ is less or equal to $d_\RR(u, u')$ for any points $u$ and $u'$ in $\mathcal{R} \cap \L.$
  \item $f(u') - f(u)$ is equal to $1$ for any $u, u' \in \partial{\mathcal R}\cap \L$ such that $u'-u$ is a unit vector in a positive direction.
  \een
  Given a tiling of $\RR$, a height function $h(v)$ can be defined by a local rule: if $u \ra v$ is a positive direction, then $h(v) = h(u) + 1$, if we follow an edge of a lozenge, and $-2$ if we cross a lozenge diagonally.  The height function is unique up to a constant shift. 
  \end{definition}

  \begin{definition}[Asymptotic height function] Let $K$ be the convex hull of the vectors $2(\hat{i} - \hat{j})$, $2 \hat{j}$ and $-2\hat{i}$. We say that a Lipschitz function $f$ on a domain is an asymptotic height function if $\nabla f \in K$ at all points $x$ where
  $f$ is differentiable. \end{definition}
  

  \begin{figure}
  \begin{center}
  \includegraphics[scale=0.40]{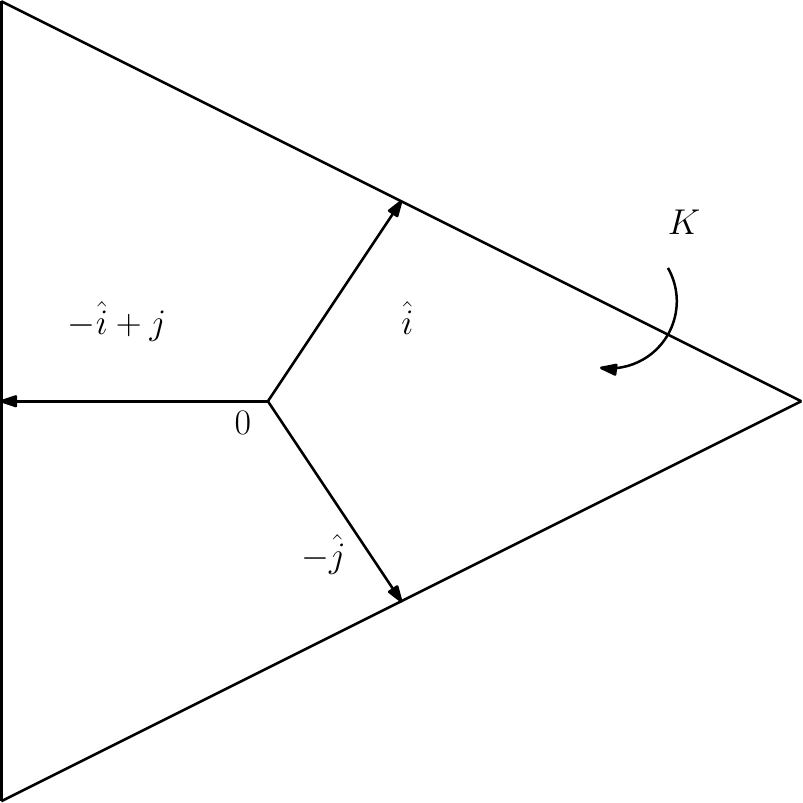}
  \caption{The convex set $K$ that the gradients of an asymptotic height function must belong to.}\label{fig:ij2}
  \end{center}
  \end{figure}

\begin{definition}[Height function on $\hexagon_{v}^n$] We look at the upper trapezoid $Z_{n, \mathtt{up}}$ and the lower trapezoid $Z_{n, \mathtt{lo}}$ whose union is $\hexagon_{v_n}^n$ and whose intersection is contained in the equator. We define height functions $f_{n,\mathtt{up}}$ and $f_{n,\mathtt{lo}}$  on the lattice points of a triangular mesh restricted respectively to $Z_{n,\mathtt{up}}$ and $Z_{n,\mathtt{lo}}$ as above for any lozenge (and triangle) tiling of $\hexagon_{v_n}^n$, shifted so that
such that their value is  $0$ at the leftmost point (see in Figure~\ref{fig:image2}).  We extend it to all real points in $\hexagon_{v}^n$ in a piecewise linear manner using the triangles of $\hexagon_{v}^n$ and call  this pair of functions  $f_n = (f_{n,\mathtt{up}}, f_{n,\mathtt{lo}})$ a height function pair. 
We denote the set of all such function pairs $f_n$ by $\HT_{v}^n$.
\end{definition}

\begin{obs}
By the definition of height functions  $ \left(f_{n,\mathtt{up}} + f_{n,\mathtt{lo}}\right) (x, y)$ is linear when restricted to $Z_{n,\mathtt{up}} \cap Z_{n,\mathtt{lo}} $  with a slope that equals $1$.
\end{obs}

\begin{definition}[Asymptotic height function on $\hexagon_v^\infty$] \lab{def:AHT} We say that $f = (f_{\mathtt{up}}, f_{\mathtt{lo}})$ is an asymptotic height function pair on $\hexagon_v^\infty$ if 
\ben \item $f_{\mathtt{up}}$ is Lipschitz on the upper trapezoid $Z_{\mathtt{up}}$,  $f_{\mathtt{lo}}$ is Lipschitz on the upper trapezoid $Z_{\mathtt{lo}}$, and wherever these functions are differentiable, their gradients belong to $K$. 
\item $f$ satisfies the boundary conditions dictated by Figure~\ref{fig:image2} and $ \left(f_{\mathtt{up}} + f_{\mathtt{lo}}\right) (x, y)$ is linear when restricted to $Z_{\mathtt{up}} \cap Z_{\mathtt{lo}} $  with a derivative that equals $1$.
\een

We denote the set of all such function pairs $f$ on $\hexagon_v^\infty$ by $\AHT_v^\infty$.
\end{definition}

\section{Log-concavity}

The following is a result of Pr\'{e}kopa (\cite{prekopa}, Theorem 6).
\begin{theorem}[Pr\'{e}kopa]\lab{theorem:prekopa}
Let $f(x, y)$ be a function of $\R^n \oplus \R^m$ where $x \in \R^n$ and 
and $y \in \R^m$. Suppose that $f$ is log-concave in $\R^{n+m}$ and let
$A$ be a convex subset of $\R^m$. Then the function of the variable x:
$$\int_A f(x, y) dy$$
is log-concave in the entire space $\R^n$.
\end{theorem}

 The following lemma appeared as \cite[Lemma 3]{NST}. Its proof was based on the work of Klartag \cite{Klartag-KLS} on the KLS conjecture.
 \begin{lemma}\lab{lem:17}
Let $\eta$ be a log-concave measure in $\R^D$ that is supported on a $d$-dimensional convex subset for some $d \leq D.$ Let $\ell_1, \dots, \ell_t$ be affine functions such that $$\var_\eta \ell_i \leq \mathbf v$$ for all $i \in [t].$ Then $$\var_\eta \left(\max_{i \in [t]} \ell_i\right) \leq C{\mathbf{v}} \log (2 + d).$$
\end{lemma}

\begin{lemma}\lab{lem:reuse}
Let $X = (x_1, \dots, x_m)$ be a random variable, whose distribution  $\eta$ is a log-concave probability measure supported on $\R^m_{\geq 0}$.  Suppose $\de \in (0, 1)$ and  $\|\E X\|_\infty \leq \de^{-\frac{1}{2}}$.
Then, $$\E\max_{|S| \leq \de m } \sum_{i \in S} x_i \leq C\de^{\frac{1}{4}} m. $$
\end{lemma}
\begin{proof}
\begin{eqnarray} \E\max\limits_{|S| \leq \de m} \sum\limits_{i \in S} x_i & \leq & \frac{\de m}{\de^{\frac{3}{4}}} + \E \sum_{i = 1}^m \max(0, x_i - \de^{- \frac{3}{4}})\\
& \leq &  \frac{\de m}{\de^{\frac{3}{4}}}  + C'm \de^\frac{1}{4}\\
& \leq & C\de^{\frac{1}{4}} m.\lab{eq:7.5p1}
\end{eqnarray}

\end{proof}

 \section{Marked Bead Model}

  Let $\Box_{m}$ be the lattice square that is equal to $[-m, m]^2 \cap \Z^2.$
   Let $\x \in \{(x_1, x_2)|0 < x_1 < x_2\}$. Let $\MM =\{\nu_{ij}\}_{1 \leq i \leq j}$ be an infinite GUE minor process, wherein $\nu_{ij}$ is the $i^{th}$ largest eigenvalue of a successively defined $j \times j$ random GUE matrix $X_j$ (our normalization conventions for a random GUE matrix are identical to those stated at the beginning of Subsection~\ref{ssec:GUEhives}). 
   Thus the non-diagonal $ik^{th}$ entry of $X_j$ is a complex Gaussian with mean zero and variance of the real and complex parts equal to $\frac{1}{2}$. The diagonals are real Gaussians with variance $1$ $X_j^* = X_j$, but the non-upper triangular entries are all independent. 
   Further, the if $j < j'$ then $X_j$ is the leftmost ($j \times j$) principle minor of $X_j'$.  Let $\ell$ be a positive integer. Let $\lceil \ell \x\rceil := (\lceil \ell x_1\rceil, \lceil \ell x_2 \rceil)$. Define $\MM|_{\lceil \ell \x\rceil + \Box_m}$ to be the restriction of $\MM$ to the square $(\lceil \ell x_1\rceil, \lceil \ell x_2 \rceil) + \Box_m$, and denote $\one_{\Box_m}$ to be the vector in $\R^{\Box_m}$ whose every entry is $1$, where $m$ is a fixed positive integer.
   
   Similarly, let $l_m$ be the lattice line segment joining $(-m, 0)$ and $(m, 0)$. Let $\y \in \{(y_1, 0)|0 < y_1\}$ and $\ell$ be a positive integer. Define $\MM|_{\lceil \ell \y\rceil + l_m}$ to be the restriction of $\MM$ to the line segment $(\lceil \ell y_1\rceil, 0) + l_m$.
   \begin{observation} By known results on the spectra of GUE matrices (see for example \cite{TaoVu-conc}), the following holds.
    As $\ell \ra \infty$  the $\R^{l_m}$ valued random variable $\ell^{-\frac{1}{2}}\left(\MM|_{\lceil \ell \y\rceil + l_m}\right)$ converges in distribution to a constant vector $2y_1 \one_{2m+1}$ where $\one_{2m+1}$ is the vector of all ones of length $2m+1$.
   \end{observation}
   \begin{definition}
The \( W_1 \) Wasserstein metric is defined as:

\[
W_1(\mu, \nu) = \inf_{\pi \in \Pi(\mu, \nu)} \int_{X \times X} d(x, y) \, d\pi(x, y),
\]

where:
\begin{itemize}
    \item \( \mu \) and \( \nu \) are probability measures on a metric space \( (X, d) \),
    \item \( \Pi(\mu, \nu) \) is the set of couplings of \( \mu \) and \( \nu \), i.e., the set of probability measures \( \pi \) on \( X \times X \) such that:
    \[
    \int_X \pi(x, y) \, d y = \mu(x), \quad \int_X \pi(x, y) \, d x = \nu(y).
    \]
\end{itemize}
\end{definition}
  The  proof of the following theorem was provided by Terence Tao in personal communication, and has been placed in  Appendix~\ref{ssec:Tao4.4}.
   \begin{theorem}[Tao, \cite{TaoComm}] \lab{theorem:beadR}
   Let $m$ be a fixed positive integer. As $\ell \ra \infty$  the $\R^{\Box_m}$ valued random variable $\ell^\frac{1}{2}\left(\MM|_{\lceil \ell \x\rceil + \Box_m} - \MM(\lceil\ell\x\rceil)  \one_{\Box_m}\right)$ converges in $W_1$ to a random variable $\MM_{m, \x}$ taking values in $\R^{\Box_m}$.
   \end{theorem}

Given $a_m \in B_{m}(a_\infty,  m\eps)$, let $\Xi(a_m)$ denote the lozenge tiling corresponding to the height function $a_m$.
Given a realization of the random variable $\MM_{m, \x}$, 
and $a_m \in B_{m}(a_\infty,  m\eps)$, 
define the  {\it total lozenge weight} of $\Xi(a_m)$ with respect to $\MM_{m, \x}$, to be $$\langle \partial \MM_{m, \x}, \partial a_m\rangle :=  2 \sum_{\edge \in \Xi(a_m), \mathrm{blue }: \edge \subseteq \Box_m} \wt(\edge) + \sum_{\edge \in \Xi(a_m), \mathrm{green }: \edge \subseteq \Box_m} \wt(\edge).$$ 
Since the lozenge weights are almost surely nonpositive, this quantity is almost surely nonpositive. 

\section{Key definitions}
\begin{definition}\lab{def:8.1}

Suppose $\rho \in \R_{>0}^2$,  $a_\infty \in K$, $\eps >0$ and $m \in \N$.

Let \begin{eqnarray*}\ \sigma_{m, \edge}(\rho, a_\infty, \eps) &:=&  -\left(\frac{1}{m^2}\right) \E \max_{a_m \in B_{m}(a_\infty,  m\eps)}  \langle \partial \MM_{m, \x(\rho)}, \partial a_m\rangle.\end{eqnarray*}

\end{definition}

 \begin{definition}\lab{def:psi}
 For $m > 0$, let  $\psi(m)$ be equal to $(\log m)^{-\frac{1}{10}}$.  \end{definition}
 \begin{observation}\lab{obs:psi}
Note for future use, that this is a monotonically decreasing function of $m$ such that 
 firstly, $\lim_{m \ra \infty} \psi(m) = 0$ and secondly,  $m \psi(m)$ is monotonically increasing with  $\lim_{m \ra \infty} m \psi(m) = \infty.$
\end{observation}

  \begin{definition}\lab{def:8.6}
  
 For  $\rho \in \R_{>0}^2$ and $a_\infty \in K$,  let
 $$ \sigma_{m, \edge}(\rho, a_\infty) := \sigma_{m, \edge}(\rho, a_\infty, \psi(m)).$$  
  \end{definition}
 
 \begin{definition}\lab{def:8.7}
 
 For $\rho \in \R_{> 0}^2$ and $a_\infty \in K^o$,
 let $$\sigma_{\edge}(\rho, a_\infty) := \liminf_{m \ra \infty} \sigma_{m, \edge}(\rho, a_\infty).$$
 For $\rho \in \R_{> 0}^2$ and $a_\infty \in \partial K \setminus\{2(\hat{i} - \hat{j}), -2\hat{i}, -2 \hat{j}\}$, define 
 $\sigma_\edge(\rho, a_\infty)$ by taking a directional liminf of $\sigma_\edge(\rho, a'_\infty)$ as $a'_\infty \ra a_\infty$ along the line segment joining $a_\infty$ to the vertex of $K$ that is not on the boundary line segment containing $a_\infty$. 
 For the three vertices in $\partial K$, define the value of $\sigma_\edge(\rho, a_\infty)$ to be the liminf  of $\sigma_\edge(\rho, a'_\infty)$ as $a'_\infty \ra a_\infty$ for  $a'_\infty \in K\setminus\{2(\hat{i} - \hat{j}), -2\hat{i}, -2 \hat{j}\}$. 
 \end{definition}

\begin{definition}
Let $$S_{v,\edge}(f) := (-1) \int_{(\hexagon_v^\infty \setminus Q)^\circ} \sigma_\edge(\rho(\x),  \partial f(\x))d\x.$$ Here $\rho(\x)$  is the vector in $\R^2$ whose coordinates are respectively, the expectations of the two interlacing gaps for the corresponding marked bead model (for $(x, y)$ that is not on the equator). 
\end{definition} 

\begin{definition}
Let $g_n = (g_{n, \up}, g_{n, \lo})$ be a pair of Gelfand-Tsetlin patterns arising from two independent GUE minor processes with variance parameters $\var_\mu$ and $\var_\la$ respectively (see the left diagram in Figure~\ref{fig:octahedron}).
Suppose the vertex $v$ has coordinates $(i, j)$ where $j \geq i$, which is the case when we are using the octahedron recurrence to recover the value of a hive vertex in $h$ (rather than $h'$). Thus the vertices of the hexagon are $(0, n), (i, n), (n+ i-j, j), (n+i-j, j-i), (i, j-1), (0, j)$ (Figure~\ref{fig:typical} is for the value of a hive vertex in $h'$, rather than $h$).
Given a height function pair $(f_{n, \up}, f_{n, \lo})$ for $\hexagon_v^n$, we define 
$b_n(k) := - f_{n, \up}(k, n-k) + f_{n, \up}(k+1, n-k-1)$ and 
\beq\lab{eq:bn} \langle g_n, \partial f_n \rangle_\De := \sum_{k = 0}^{n+i-j-1} \left((\frac{1 + b_n(k)}{3}) \frac{\mu_{1, n-k}}{3} + (\frac{2 - b_n(k)}{3}) (- \frac{\la_{n-k, n-k}}{3})\right).\eeq
\end{definition}

\begin{definition}\lab{defn:8.14}
Corresponding to a point $\y \in Q$, let $\tau_\up n$ and $\tau_\lo n$ denote the asymptotic expected eigenvalue  at the  location $n\y$ on the equator in the upper and lower trapezoid respectively in $\hexagon_v^n$.

Let  $$ \sigma_\Delta(\tau, b_\infty) :=  (-1)\left((\frac{b_\infty +1}{3})\frac{\tau_\up}{3} + (\frac{2 - b_\infty}{3})(-\frac{\tau_\lo}{3})\right).$$
\end{definition}

\begin{definition} Let $\tau$ be defined as in Definition~\ref{defn:8.14}.
Let 
$$S_{v,\Delta}(f) := (-1) \int_{ Q} \sigma_\Delta(\tau(\y), \partial f(\y))d\y.$$
\end{definition}

\begin{definition}Let  $$S_{v,\hexagon} := \limsup_{n \ra \infty} n^{-2} \E_n \weight'(\hexagon_v^n).$$ \end{definition}

\begin{definition} \lab{def:Sv} Let $$S_v(f) := S_{v,\edge}(f) + S_{v,\Delta}(f) + S_{v,\hexagon}.$$
\end{definition}

\begin{definition}\lab{def:14}
Given two asymptotic height functions $f$ and $g$ for $\hexagon_v^n$ we define $f \lor g$ to be the function that in the upper trapezoid takes the maximum of $f$ and $g$ and in the lower trapezoid takes the minimum of $f$ and $g$.
  Likewise, we define $f \land g$ to be the function that in the upper trapezoid takes the minimum of $f$ and $g$ and in the lower trapezoid takes the maximum of $f$ and $g$.
  For a collection $\FF$ of asymptotic height functions, we define $\bigvee_{f \in \FF} f$ to be the function that in the upper trapezoid takes the maximum of the functions in $\FF$ and in the lower trapezoid takes the minimum of the functions in $\FF$. Likewise, we define $\bigwedge_{f \in \FF} f$ to be the function that in the upper trapezoid takes the minimum of the functions in $\FF$ and in the lower trapezoid takes the maximum of the functions in $\FF$.
   We say $f \succeq g$ and $g \preceq f$ if $f \lor g = f$ and $f \land g = g$.
  
  Note that if $\FF$ is a set of asymptotic height functions for $\hexagon_v^n$, then $\bigvee_{f \in \FF} f$ and $\bigwedge_{f \in \FF} f$ are also asymptotic height functions. 
\end{definition}

\begin{definition}\lab{def:HT} Given $f \in \HT_v^\infty$ and positive $\eps$, let $\HT_v^n(f, \eps n)$ denote the set of all functions in $f_n \in \HT_v^n$, such that $| \frac{f_n}{n} (n x) - f(x)| < \eps|$ for all $x \in \hexagon^\infty_v$.
\end{definition}
\section{Key propositions and theorems}

\subsection{Bead model}
   \begin{proposition}[Local concentration bounds at fixed index]\lab{prop:8.17} 
Let $g_{m}$ denote an $(2m+1)\times (2m+1)$ patch corresponding to the eigenvalues in a $(2m+1)\times (2m+1)$ box of a GUE minor process around some fixed $(\lceil i_0 n \rceil, \lceil j_0 n\rceil)$, where $(i_0, j_0) \in (0, 1)^2$ and $  j_0 < i_0$, and $m = o(n).$ Thus, the square consists of all indices $(i, j)$ such that  $\max(|i- \lceil i_0 n\rceil|, |j-\lceil j_0 n\rceil|) \leq m.$
Let $B_m(a_\infty, m\eps)$ denote the set of all height functions on a $(2m+1)\times (2m+1)$ piece of the triangular lattice that are contained in an $L^\infty$ radius of $n\eps$ centered around $a_\infty$. Denote by $\partial g_m$, the array of interlacing gaps of the form $\la_{ij} - \la_{i-1\, j}$, where $\la_{ij}$ is in $g_{m}$. For any fixed $\eps > 0$ and fixed $C > 0$, for all affine functions $a_\infty \in K $, 
$$   \var\left( \max_{a_m \in B_m(a_\infty,  m\eps)}  \langle \partial g_{m}, \partial a_m\rangle\right) = O\left(\frac{m^4}{\log^C  m}\right).$$
\end{proposition}
The proof of this proposition appears in Section~\ref{sec:Bead}.

\subsection{Height functions}
  \begin{proposition}\lab{prop:6.4}
  Let $f$ be an asymptotic height function for $\hexagon_v^n$. There exists a height function $f_n$ for $\hexagon_v^n$ such that $\max_v |f(v) - f_n(v)| < 3$.
  \end{proposition}
  
  \begin{proposition} \lab{prop:3} Let $m \in \N$ and 
  let $\RR_0$ be a simply connected tileable lattice domain with boundary $\partial \RR_0$. Let $ 0 < 2 \eps' < \hat{\eps} \tilde{\eps}.$ 
  Let $\RR_1 \subseteq \RR_0$, where $\RR_1$  is a simply connected tileable lattice domain, such that every lattice point $x \in \partial\RR_1$ satisfies $d_{\RR}(x, \partial \RR_0) > \tilde{\eps}m$ and every lattice point $y \in \partial \RR_0$ satisfies $d_{\RR}(x, \partial \RR_1) > \tilde{\eps}m$. 
  Let $\RR = \RR_0\setminus \RR_1.$ Let $g_i: \partial \RR_i \ra \Z$ for $i \in \{0, 1\}$ be such that 
   $g_i$ extends to a height function on $\RR_i$ for each $i \in \{0, 1\}$ and 
    there exists an affine asymptotic height function $a_\infty \in (1 - \hat{\eps})K$ such that $\sup\limits_{\partial \RR_i} |g_i - a_\infty| < \eps' m$ for $i \in \{0, 1\}$.
 
  Then, there is a constant $\chi \in \{-1, 0, 1\}$ such that $g: \partial \RR \ra \R$ defined by $$g|_{\partial \RR_0} = g_0$$ and $$g|_{\partial \RR_1} = g_1 + \chi,$$ extends to a height function on $\RR$.
  \end{proposition}
 The proofs of Propositions~\ref{prop:6.4} and \ref{prop:3} appear in Section~\ref{sec:Height}.

   \subsection{Surface tension}
\begin{proposition}\lab{prop:conv-final}
 $\sigma_{m, \edge}(\rho, a_\infty)$ converges uniformly to $\sigma_{ \edge}(\rho, a_\infty)$ on compact subsets of $K^o$ for fixed $\rho$. Additionally,
 given  a compact subset $\C \subseteq \R_{>0}^2$, and $\eps' > 0$,  there exists an $m_0$ such that for all $m > m_0$ and all $(\rho, a_\infty) \in \C \times K,$ $$ \sigma_{m, \edge}(\rho, a_\infty)  \geq \sigma_\edge(\rho, a_\infty) - \eps'.$$
Further,  for each fixed $\rho \in \R_{>0}^2$, $\sigma_\edge(\rho, a_\infty)$ is a continuous nonnegative convex 
  function of $a_\infty$ on $K$ bounded above by $C_{|\rho|}$ and consequently,  for any $\de > 0$,  $\sigma_\edge(\rho, a_\infty)$ is a $C_{|\rho|}\de^{-1}$-Lipschitz function of $a_\infty$ on $(1 - \de)K$. 
\end{proposition}

\begin{proposition}\lab{prop:sig-unif}
Let $\rho \in \R_{> 0}^2$. Let $\bar{B}(\rho, \de)$ be the closed Euclidean ball of radius $\de$ around $\rho$. Then, the function $|\sigma_\edge(\rho', a'_\infty) - \sigma_\edge(\rho, a'_\infty)|$ defined for  $(\rho', a'_\infty) \in \bar{B}(\rho, \de) \times K$ is uniformly upper bounded by $o(1)$ as $\de \ra 0$.
\end{proposition}  
   
    The proofs of Propositions~\ref{prop:conv-final} and \ref{prop:sig-unif} appear in Section~\ref{sec:Surface}.
   \subsection{Quantitative differentiation}
   
  \begin{theorem}[Fefferman]\lab{thm:Feff}
  Suppose $\varepsilon, \eta > 0$ and $k \geq 1$. Assume that $\varepsilon^{\nnn+2} \eta k \geq C^\sharp$ (a large enough constant depending only on the dimension $\nnn$).
  Let $F: Q^\circ \ra \R$, where $Q^\circ$ is the unit cube in $\R^\nnn$. Suppose $F$ has Lipschitz constant $\leq 1$. Then we may decompose $Q^\circ$ into finitely many GOOD cubes and finitely many BAD cubes, with the following properties:
  \begin{enumerate}
  \item[(I)] The GOOD cubes have sidelength $\geq 2^{-k}$. For each GOOD cube $Q$, we have $|F - L_Q| \leq \eps \de_Q$ on $Q$, where $L_Q$ is the linear function that best approximates $F$ in $L^2(Q)$. 
  \item[(II)] The BAD cubes have sidelength exactly $2^{-k}$, and their total volume is at most $\eta$.
  \end{enumerate}
  
  \end{theorem}

  For the convenience of the reader, the proof due to Fefferman is reproduced in  Appendix~\ref{sec:Feff}.
  
 
  \begin{remark}
   Although the particular result quoted here is due to Fefferman,  this is merely one way to prove quantitative differentiation,  and 
    quantitative differentiation was well known.
  \end{remark}

\section{Convergence of GUE hives}

\begin{definition}
Let the equator in $\hexagon_v^n$ be denoted $Q_n,$ and the equator in $\hexagon_v$ be denoted $Q$.
Let $U_{n, \eps}$ denote the set of all points in $U_n$ at a distance at least $n\eps$ from $n \partial U$. Similarly, let $U'_{n, \eps}$ denote the set of all points in $U'_n$ at a distance at least $n\eps$ from $n \partial U'$.
Let $U_{ \eps}$ denote the set of all points in $U$ at a distance at least $\eps$ from $\partial U$. Similarly, let $U'_{\eps}$ denote the set of all points in $U'$ at a distance at least $\eps$ from $ \partial U'$.
 (See Figure~\ref{fig:octahedron} for a visualization of $U$ and $U'$).
\end{definition}

\subsection{Contribution of the equator and the hexagon}
  
Let $b_\infty: \R \ra \R$ be an arbitrary  affine function whose derivative lies in the interior of the projection of $K$ on the $x-$axis (see Figure~\ref{fig:ij2}), \ie the closed interval $[-1, 2]$.

Let $B_{m, \Delta}(b_\infty, \eps)$ denote the set of Lipschitz piecewise linear functions to $l_m$ whose derivatives belong to $\{-1, 2\}$ that are contained inside an $L^\infty$ $\eps m $-neighborhood of $b_\infty.$ 

Recall that the distribution of the minor process on $k$ restricts at level $j$ to a spectral distribution corresponding to  GUE matrices, rescaled so that the entry-wise variance is $\var_\mu$,
that on $k'$ is given by $\var_\la$.





Let $f^\sharp$ be an arbitrary asymptotic height function in $\AHT_v^\infty$.
\begin{lemma}\lab{lem:8}

$$\limsup_{n \ra \infty}  \E_n  n^{-2} \max_{f_n \in \HT_v^n(f^\sharp, \eps)}  \langle g_n, \partial f_n \rangle_\De  \leq  S_{v, \De}(f^\sharp) + o_\eps(1),$$ and
$$\liminf_{n \ra \infty}  \E_n  n^{-2} \min_{f_n \in \HT_v^n(f^\sharp, \eps)}  \langle g_n, \partial f_n \rangle_\De  \geq  S_{v, \De}(f^\sharp) - o_\eps(1).$$
\end{lemma}

\begin{proof}
From Lemma \ref{lem:TV}, we conclude in particular that
$$ \lambda_n = \E \lambda_n + O(n^{1/3} \log^{O(1)} n) $$
with probability greater than $1 - n^{-C}$ for an arbitrarily large absolute constant $C$.  Since the $k \times k$ minor of a GUE matrix is also a GUE matrix, we similarly have
$$ \lambda_{k,k} = \E \lambda_{k,k} + O(n^{1/3} \log^{O(1)} n)$$
with probability  greater than $1 - n^{-C}$ for an arbitrarily large absolute constant $C$ for all $1 \leq k \leq n$, where we recall that $\lambda_{1,k} \geq \dots \geq \lambda_{k,k}$ are the eigenvalues of the top left $k \times k$ minor of $A$ (and thus form the $k^{\mathrm{th}}$ row of $g$).  Similarly we have
$$ \mu_{1,k} = \E \mu_{1,k} + O(n^{1/3} \log^{O(1)} n)$$
with probability  greater than $1 - n^{-C}$ for an arbitrarily large absolute constant $C$ for all $1 \leq k \leq n$.  
Let $\Delta$ be a border triangle associated to a border edge $((i,n-i), (i+1,n-i-1))$ for some $0 \leq i < n$.  By inspecting the definitions, we see that the weight of this triangle is given by the formula
$$ \weight(\Delta) = \frac{1}{3} \mu_{1,n-i}$$
if $\Delta$ is an upward pointing triangle, and
$$ \weight(\Delta) = - \frac{1}{3} \lambda_{n-i,n-i}$$
if $\Delta$ is a downward pointing triangle.  
We conclude that for the supremum over lozenge tilings $\Xi_n$ of $\hexagon_v^n$, we have
$$ \sup_{\Xi_n} \sum_{\Delta \in \Xi_n} |\weight(\Delta) -  \E_n \weight(\Delta)| \leq \sum_{\textit{all}\, \Delta}  |\weight(\Delta) -  \E_n \weight(\Delta)| \leq O( n^{4/3} \log^{O(1)} n )$$
with probability greater that $1 - n^{-100}.$
The error term $O( n^{4/3} \log^{O(1)} n )$ is small enough for our purposes, so it suffices to prove the following  claim.
\begin{claim}\lab{cl:8.1} $$  n^{-2}  \max_{f_n \in \HT_v^n(f^\sharp, \eps)}  |\langle \E_n g_n, \partial f_n \rangle_\De   - S_{v, \De}(f^\sharp)|  < o_\eps(1).$$
\end{claim}
\begin{proof}
We rewrite the expression for $\langle \E_n g_n, \partial f_n \rangle_\De$ arising from  (\ref{eq:bn}) using summation by parts, as
$$ \sum_{k = 0}^{n+i-j-1} \left( \E_n (\frac{\mu_{1, n-k}}{9} - \frac{2 \la_{n-k, n-k}}{9}) + (\sum_{t \leq k} b_n(t)) (\E_n \frac{\mu_{1, n-k} + \la_{n-k, n-k} - \mu_{1, n-k-1} - \la_{n-k-1, n- k-1}}{9})\right).$$
Here we have adopted the convention that when the indices of $\mu$ and $\la$ go out of bounds, the value of the corresponding variable is $0$.
Note that  $f_{n, \up}(k, n-k) = \sum_{t \leq k} b_n(t).$
Writing $f_{n, \up} = f^\sharp_{n, \up} + (f_{n, \up} - f^\sharp_{n, \up})$ and using the triangle inequality, we see that the claim follows.
\end{proof}

This completes the proof of Lemma~\ref{lem:8}.
\end{proof}
\begin{lemma}\lab{lem:9}
$n^{-2} \weight'(\hexagon_v^n)$ is a log-concave random variable that concentrates around its expectation under $\p_n$ as $n \ra \infty$ with a standard deviation of $O( n^{-1} \log^{O(1)} n) $, and in the limit as $n \ra \infty$, $$\E_n n^{-2} \weight'(\hexagon_v^n)$$ converges.
\end{lemma}
\begin{proof}

The weight $\weight'(\hexagon_v)$ is a certain linear combination of the eigenvalues $\lambda_i, \mu_j$ with bounded coefficients, and is a log-concave random variable by Theorem~\ref{theorem:prekopa}.  By Lemma \ref{lem:TV}, we conclude that
\begin{align*}
 \weight'(\hexagon_v) &= \E \weight'(\hexagon_v) + O\left( \sum_{i=1}^n n^{1/3} \min(i, n - i + 1)^{-1/3} \log^{O(1)} n \right)\\
&= \E \weight'(\hexagon_v) + O( n \log^{O(1)} n) .
\end{align*}
The contribution of the $O( n \log^{O(1)} n )$ error is acceptable, which proves the concentration. That $n^{-2}\E_n \weight'(\hexagon_v)$ has a limit follows from \cite[Theorem 1.4]{Sosoe}.
\end{proof}

\subsection{Upper semicontinuity of $S_{v, \edge}$}

 Let $\lim_{n \ra \infty} v_n/n =  v.$ 

Let $f$ be an asymptotic height function pair on $\hexagon_v^\infty$. 
Following Cohn, Kenyon and Propp, \cite{CKP} for $\ell > 0$, we look at a mesh made up of equilateral triangles of side length $\ell$ (which we call an $\ell$-mesh). Consider any piecewise linear function $\tilde{f}$ that agrees with $f$ on the vertices of the mesh and is linear on each triangle.

\begin{lemma}[Lemma 2.2, \cite{CKP}]\lab{lem:6}
Let $f$ be an asymptotic height function, and let $\varepsilon > 0$. If $\ell$ is sufficiently small, then on at least a $1 - \varepsilon$ fraction of the triangles in 
the $\ell$-mesh that are contained in  $\hexagon_v \cap(U_\eps \cup U'_\eps)$, we have the following two properties. First, the piecewise linear approximation $\tilde{f}$ agrees with $f$ to within $\ell\varepsilon$. Second, for at least a $1 - \varepsilon$ fraction (in measure) of the points $x$ of the triangle, the tilt $f'(x)$ exists and is within $\varepsilon$ of $\tilde{f}'(x).$

\end{lemma}


\begin{definition}
  Let $X$ be a metric space.
  Recall that a function $g:X \ra \R$ is upper semicontinuous iff given $x \in X$ and $r > g(x)$ there is an open neighborhood of $U$ of $x$  such that $r > g(y)$  for each $y \in U$. 
Equivalently,  $g$ is upper semicontinuous iff $$\limsup_{x \ra y} g(x)  \leq g(y)$$ for all $y \in X.$
\end{definition}

\begin{proposition}\lab{prop:upper-semi} 
Let $\AHT_v^\infty$ be equipped with the $L^\infty(\hexagon_v^\infty)$ metric. Then, $S_{v, \edge}$ is an upper semicontinuous functional on $\AHT_v^\infty$.
\end{proposition}
\begin{proof}

We note that $\sigma_\edge(\rho, a_\infty)$ is continuous and convex in $a_\infty$ when $\rho \in \R_{> 0}^2$ is fixed,  by Proposition~\ref{prop:conv-final}.  Also by Proposition~\ref{prop:sig-unif}, the function $|\sigma_\edge(\rho', a'_\infty) - \sigma_\edge(\rho, a'_\infty)|$ defined for  $(\rho', a'_\infty) \in \bar{B}(\rho, \de) \times K$ is uniformly upper bounded by $o(1)$ as $\de \ra 0$, when $\rho \in \R_{> 0}^2$. The proof of this theorem is now essentially the same as the proof of \cite[Theorem 7.5]{GorinBook} but we provide details below. 
By Lemma~\ref{lem:6} (and retaining its notation), \beq \lim_{\eps \ra 0}\left(\int_{\hexagon_v^\infty}(-1)\sigma_\edge(\rho(\x), \partial f(\x))d\x - \int_{\hexagon_v^\infty}(-1)\sigma_\edge(\rho(\x), \partial \tilde{f} (\x))d\x \right) = 0.\lab{eq:cor:gorin}\eeq
Suppose now that $\eps > 0$ is fixed and the domain $\RR^*$ is an equilateral triangle $T$ of side length $\ell $ contained in  $\hexagon_v^\infty \cap U_\eps$ or $\hexagon_v^\infty \cap U'_\eps. $

\begin{claim}\lab{cl:7.8} Let $f$ be an asymptotic height function for $T$ and let $\tilde{f}$ be a linear function satisfying $|f - \tilde{f}| <  \varepsilon \ell$ on $\partial T$. Then, $$\frac{\int_T(-1) \sigma_\edge(\rho(\x), \partial f(\x))d\x}{|T|} \leq \frac{\int_T(-1) \sigma_\edge(\rho(\x), \partial \tilde{f}(\x))d\x}{|T|} + o(1)$$
as $\varepsilon \ra 0$ and $\ell  \ra 0$.
\end{claim}
\begin{proof}
By the concavity of $-\sigma_\edge(\rho, a_\infty)$ for fixed $\rho$ in the second argument and using Proposition~\ref{prop:sig-unif} to control the effect of the variation in the first argument, we have

\begin{equation}
\frac{\int_T (-1)\sigma_\edge(\rho(\x), \partial f(\x)) \, d\x}{|T|} \leq \frac{\int_T (-1)\sigma_\edge(\rho(\x), \text{Avg}(\partial f(\x)))d\x}{|T|} + o(1), \lab{eq:7.4}
\end{equation} 
as $\ell \ra 0$,
where $\text{Avg}(\partial f)$ is the average value of $\partial f$ on $T$. Also, we have
\[
\|\text{Avg}(\partial f) - \text{Avg}(\partial \tilde{f})\| = O(\varepsilon),
\]
which can be proven by reducing the integral over $T$ in the definition of the average value to the integral over $\partial T$ using the fundamental theorem of calculus and then using $\|f - \tilde{f}\| < \varepsilon \ell$. Therefore, by continuity of $\sigma_\edge$ in the second argument, uniformly for all $\x \in \hexagon_v^\infty \cap (U_\eps \cup U'_\eps)$, we have
\[
(-1)\sigma_\edge(\rho(\x), \text{Avg}(\partial f)) = (-1)\sigma_\edge(\rho(\x), \text{Avg}(\partial \tilde{f})) + o(1).
\]
Since $\tilde{f}$ is linear, its gradient is constant on $T$ and we can remove the average value in the right-hand side of the last identity. Combining with (\ref{eq:7.4}), the claim is proved. 
\end{proof}

Let $f$ be an asymptotic height function for $\hexagon_v^\infty$. Then for each $\gamma > 0$, we want to show that there exists $\delta > 0$ such that whenever $f$ and $h$ differ by at most $\delta$ everywhere, we should have
\[
\int_{\mathcal{R}^*}   (-1)\sigma_\edge(\rho(\x), \partial h)\, d\x \leq \int_{\mathcal{R}^*}  (-1)\sigma_\edge(\rho(\x), \partial f) d\x + \gamma.
\]
Take a piecewise linear approximation $\tilde{f}$ for $f$ from Lemma~\ref{lem:6}, and choose $\delta < \varepsilon \ell$. Let us call a triangle $T$ of the mesh ``good," if the two approximation properties of Lemma~\ref{lem:6} hold on it. Then as $\varepsilon \to 0$, we have by Claim~\ref{cl:7.8}
\[
\int_T  (-1)\sigma_\edge(\rho(\x), \partial h) \, d\x \leq \int_T  (-1)\sigma_\edge(\rho(\x),\partial \tilde{f}) \, d\x + \text{Area}(T) \cdot o(1),
\]
and thus by summing over all the good triangles $T$
\[
\int_{\text{good triangles in } \mathcal{R}^*}  (-1)\sigma_\edge(\rho(\x), \partial h) \, d\x \leq \int_{\text{good triangles in } \mathcal{R}^*}  (-1)\sigma_\edge(\rho(\x), \tilde{f}) \, d\x + \text{Area}(\mathcal{R}^*) \cdot o(1).
\]
Because $ (-1)\sigma_\edge(\rho(\x), a_\infty)$ is bounded when $x \in \RR^*$, the bad triangles only add another $O(\varepsilon)$. Thus, using also (\ref{eq:cor:gorin}), we get
\[
\int_{\mathcal{R}^*}  (-1)\sigma_\edge(\rho(\x),\partial h) \, d\x \leq \int_{\mathcal{R}^*}  (-1)\sigma_\edge(\rho(\x),\partial f) \, d\x + \text{Area}(\mathcal{R}^*) \cdot o(1).
\]


\end{proof}

\subsection{Main theorem}
 We have the following main theorem, which will be proved in the sequel.
\begin{theorem}\lab{theorem:main}
Let $\la_n:= (\la_{1, n}, \dots, \la_{n, n})$ and $\mu_n :=(\mu_{1, n}, \dots, \mu_{n, n})$ be respectively the eigenvalues of two independent random matrices $X_n$ and $Y_n$ such that $X_n/( \sqrt{\var_\la n})$ and $Y_n/( \sqrt{\var_\mu n})$ have the distribution of a GUE.  
Let $a_n$ be a random sample from the normalized Lebesgue measure on $\AHIVE_{\diag(\lambda^{(n)} \boxplus \mu^{(n)} \rel \ast) \rel \ast}.$ Let $\lim_{n \ra \infty} v_n/n = v,$ where $v_n$ is a vertex  of $\AHIVE_{\diag(\lambda^{(n)} \boxplus \mu^{(n)} \rel \ast) \rel \ast}$ in $T$. Then,
$\lim_{n \ra \infty}n^{-2} \E_n h_n(v_n)$ exists and equals  $\sup\limits_{f^\sharp \in \AHT_v^\infty} S_{v}(f^\sharp).$
\end{theorem}
Recall that the main result of \cite{NST} stated that under the same conditions, $\lim_{n \ra \infty}n^{-4} \var\, h_n(v_n) = 0.$

Therefore, we have the following corollary.
\begin{corollary}
  Let $h_n$ be the hive part of a random augmented hive with GUE boundary conditions, chosen from the measure $\p_n$. Then, $n^{-2}  h_n(v_n)$ converges in probability to 
  $\sup\limits_{f^\sharp \in \AHT_v^\infty} S_{v}(f^\sharp).$
\end{corollary}

\subsection{Upper bound on $h(v)$}
{\it Definition~\ref{def:HT} restated:} Given $f \in \HT_v^\infty$ and positive $\eps$, let $\HT_v^n(f, \eps n)$ denote the set of all functions in $f_n \in \HT_v^n$, such that $| \frac{f_n}{n} (n x) - f(x)| < \eps|$ for all $x \in \hexagon^\infty_v$.

\begin{lemma}\lab{lem:12}
$$\limsup_{n \ra \infty}  \E_n  n^{-2} \max_{f_n \in \HT_v^n(f^\sharp, \eps n)}  \langle \partial g_n|_{\hexagon_v^n \cap (U_{n, \eps} \cup U'_{n, \eps})}, \partial f_n|_{\hexagon_v^n \cap (U_{n, \eps} \cup U'_{n, \eps})} \rangle  \leq  S_{v, \edge}(f^\sharp) + o_\eps(1).$$
\end{lemma}
\begin{proof}

As in Theorem~\ref{thm:Feff}, let $\varepsilon, \eta > 0$ and $k \geq 1$ where $\varepsilon^4 \eta k \geq C^\sharp$ (a large
enough absolute constant
). We will further assume that $\eta = \varepsilon$ and that both tend to $0$ as  $\eps \ra 0$ at some sufficiently slow rate to be determined later. We will consider two length scales: the first is $\ell_1 = n\eps_1$. For fixed $\eps$, $\eps_1$ is fixed. The fact that this mesh is not a lattice mesh will not affect us.  The second length scale given by an absolute constant $C_0$ (equal to a positive  integer  power of $2$) that is sufficiently large that  $2^k < C_0^{\frac{1}{2}}.$ 

Now  subdivide the upper and lower trapezoids of $\hexagon_v^n$ (separately) into dyadic triangles of an $\ell_1$-mesh. Further divide this $\ell_1$-mesh into a $C_0$-mesh. 
 In general, we will be considering the maximal subset of a triangle of the $\ell$-mesh which can be subdivided into dyadic $C_0$-triangles. The remaining portion of  will have a normalized area that is $O(n^{-1})$, which will be handled along with the BAD triangles in Theorem~\ref{thm:Feff}.
Given any dyadic triangle $T'$, let  $\sidelength(T')$ denote the sidelength of $T'$.
Let $\HT_{v, T'}^n(L, \varepsilon \sidelength(T'))$ denote the set of height functions defined on the dyadic triangle $T'$ that are within $\varepsilon \sidelength(T')$ (in $\sup$ norm) of an affine function whose linear component equals $L$.


In the inequality in Claim~\ref{cl:3} below, the summation over $T'$ is over all dyadic triangles that have sidelength in the interval $[2^{-k} C_0, C_0]$ and are proper subsets of some triangle of the $\ell_1$-mesh.  The maximum over $L_{T'}$ in the subscript is over the set of linear asymptotic height functions corresponding to possible linear interpolations of height functions over $T'$ that agree with some height function on the vertices. Further, $\varepsilon $ is chosen to be 
$\psi(2^k) = \frac{1}{k^{\frac{1}{10}}}$ by Definition~\ref{def:psi}.

Given a random variable $Z$, let $Z_+$ denote $\max(0, Z).$
Let $X$ denote the random variable
  $$\left(\max_{f_n \in \HT_v^n(f^\sharp, \eps n)}  \langle \partial g_n|_{\hexagon_v^n \cap (U_{n, \eps} \cup U'_{n, \eps})}, \partial f_n|_{\hexagon_v^n \cap (U_{n, \eps} \cup U'_{n, \eps})} \rangle -  n^2 S_{v, \edge}(f^\sharp)\right)_+.$$ Let $\NN(T')$  be the set of linear approximations of height functions on $T'$ that agree with some height function on the vertices.
  Let $$Y = \sum_{T'} \max_{\partial L_{T'} \in \NN(T')} \left| \max_{\bar{f}_n|_{T'} \in \HT_{v, T'}^n(L_{T'}, \varepsilon  \sidelength(T'))}   \langle \partial g_n|_{T'}, \partial \bar{f}_n|_{T'} \rangle  - \E_n  \max_{\bar{f}_n|_{T'} \in \HT_{v, T'}^n(L_{T'},  \varepsilon  \sidelength(T'))}  \langle \partial g_n|_{T'}, \partial \bar{f}_n|_{T'} \rangle\right|.$$

\begin{claim} \lab{cl:3}  \begin{equation}\lab{eq:**}  \E X \leq  \E Y + o_\eps(n^2). \end{equation}  \end{claim}
\begin{proof}

  Let $f_n$ denote an arbitrary function in $\HT_v^n$. 
  Set $F$ in Theorem~\ref{thm:Feff} (after the necessary rescaling of domain) to the piecewise linear extension of $f_n$  restricted to a fixed $C_0$-triangle $\tilde{T}$. By an application of Theorem~\ref{thm:Feff}, 
  every triangle $\tilde{T}$ of the $C_0$-mesh can be subdivided dyadically into GOOD and BAD triangles (denoted $T'$ and $T''$ respectively),  such that  
  \begin{enumerate}
  \item[(I)] The GOOD triangles have sidelength $\geq 2^{-k} C_0$. For each GOOD triangle $T'$, we have $|F - L_{T'}| \leq \varepsilon \de_{T'}$ on $T'$, where $L_{T'}$ is the linear function that agrees with  $F$ on the vertices of $T'$. 
  \item[(II)] The BAD triangles $T''$ have sidelength exactly $2^{-k}C_0$, and their total area is at most $\eta$.
  \end{enumerate}
 
  Let the magnitudes (\ie negatives) of the weights of the blue and green rhombi be denoted $x_1, \dots, x_m$. Their joint distribution is log-concave by Theorem~\ref{theorem:prekopa}. Let $k \leq C \eta n^2$ be the total number of blue and green rhombi that have a non-trivial intersection with the BAD triangles $T''$ of some height function $f_n$. Then, by Theorem~\ref{theorem:Tao1} and interlacing, $\E x_i \leq C\eps^{-\frac{1}{2}}.$ Therefore, by Lemma~\ref{lem:reuse},
\begin{equation}
 \E\max\limits_{|S| \leq k} \sum\limits_{i \in S} x_i  \leq  o_\eps(n^2).\lab{eq:7.5}
\end{equation}
  Let $L_{T', f_n}$ denote the linear approximant to  $f_n$ on $T'$ that agrees with $f_n$ on the vertices of $T'$. 
 
 Then by Theorem~\ref{thm:Feff} and (\ref{eq:7.5}), 
$$ \E_n\max_{f_n \in \HT_v^n(f^\sharp, \eps n)}  \langle \partial g_n|_{\hexagon_v^n \cap (U_{n, \eps} \cup U'_{n, \eps})}, \partial f_n|_{\hexagon_v^n \cap (U_{n, \eps} \cup U'_{n, \eps})} \rangle \leq $$ 
$$ \E_n\max_{f_n \in \HT_v^n(f^\sharp, \eps n)} \sum_{\substack{T' \text{\, GOOD}\\\text{w.r.t\,} f_n}}  \max_{\bar{f}_n \in \HT_{v, T'}^n(L_{T', f_n},  \varepsilon  \sidelength(T'))}  \langle \partial g_n|_{T'}, \partial \bar{f}_n|_{T'} \rangle  + o_\eps(n^2) \leq$$
$$  \E_nY + \max_{f_n \in \HT_v^n(f^\sharp, \eps n) }\sum_{\substack{T' \text{\, GOOD}\\\text{w.r.t\,} f_n}}  \E_n  \max_{\bar{f}_n \in \HT_{v, T'}^n(L_{T', f_n},  \varepsilon  \sidelength(T'))}  \langle \partial g_n|_{T'}, \partial \bar{f}_n|_{T'} \rangle + o_\eps(n^2). $$ 
Let $f_{n, \ell_1}$ denote the piecewise linear extension of $f_n$ restricted to the $\ell_1$-mesh. 

The first inequality below is  because $2^{-k} C_0 \ra \infty$ as $\eps \ra 0$.  For any height function $f_n \in \HT_v^n(f^\sharp, \eps n)$, $f_{n, \ell_1}$ is an asymptotic height function, in which the piecewise linear pieces have sidelength $\ell_1$. For a fixed triangle $T$ of the $\ell_1$-mesh, let $\rho(T)$ denote the value of $\rho$ at the center of $T$, and let $a_T$ denote the tilt of the linear function on $T$ whose values agree with that of $f_n$ when restricted to the endpoints of $T$. 
 For a fixed triangle $T$ of the $\ell_1$-mesh,
\begin{eqnarray}\sum_{\substack{T' \text{\, GOOD}\\\text{w.r.t\,} f_n\\ T' \subseteq T}}  \E_n  \max_{\bar{f}_n \in \HT_{v, T'}^n(L_{T', f_n},  \varepsilon  \sidelength(T'))}  \langle \partial g_n|_{T'}, \partial \bar{f}_n|_{T'} \rangle  & \leq & \sum_{\substack{T' \text{\, GOOD}\\\text{w.r.t\,} f_n\\ T' \subseteq T}}(-1) |T'| \sigma_{2^{-k}C_0, \edge}(\rho(T'), a_{T'}) + \eps_1^2 o_\eps(n^2).\nonumber \end{eqnarray}
Using the local invariance (within $T$) of $\sigma_{2^{-k}C_0, \edge}$ with respect to $\rho$ that we have from Theorem~\ref{theorem:Tao3},
$$\sum_{\substack{T' \text{\, GOOD}\\\text{w.r.t\,} f_n\\ T' \subseteq T}}(-1) |T'| \sigma_{2^{-k}C_0, \edge}(\rho(T'), a_{T'}) \leq 
\sum_{\substack{T' \text{\, GOOD}\\\text{w.r.t\,} f_n\\ T' \subseteq T}}(-1) |T'| \sigma_{2^{-k}C_0, \edge}(\rho(T), a_{T'}) + \eps_1^2 o_\eps(n^2)  $$
The RHS above is (by Proposition~\ref{prop:conv-final}), is less or equal to 
$$\sum_{\substack{T' \text{\, GOOD}\\\text{w.r.t\,} f_n\\ T' \subseteq T}}(-1) |T'| \sigma_\edge(\rho(T), a_{T'}) + \eps_1^2 o_\eps(n^2).$$
However,

$$\sum_{\substack{T' \text{\, GOOD}\\\text{w.r.t\,} f_n\\ T' \subseteq T}}(-1) |T'| \sigma_\edge(\rho(T), a_{T'}) + \eps_1^2 o_\eps(n^2)\leq (-1)|T|\sigma_\edge(\rho(T), a_{T}) + \eps_1^2 o_\eps(n^2),$$ where the last inequality uses the convexity of $\sigma_\edge(\rho, a_{T'})$ as a function of $a_{T'}$ from Proposition~\ref{prop:conv-final}.
Using the local invariance (within $T$) of $\sigma_\edge$ with respect to $\rho$, by Proposition~\ref{prop:sig-unif}, and controlling the total contribution of the bad triangles as was done in (\ref{eq:7.5}) we see that the RHS is less or equal to  $$(-1)\int_T \sigma_\edge(\rho(\x), a_{T})d\x + \eps_1^2 o_\eps(n^2).$$


Next, by the upper semicontinuity of $S_{v, \edge}$ from Proposition~\ref{prop:upper-semi},
  \begin{eqnarray} n^2 S_{v, \edge}(f_{n, \ell_1})+ o_\eps(n^2) \leq  n^2 S_{v,  \edge}(f^\sharp)+ o_\eps(n^2).\end{eqnarray}
The claim follows.
\end{proof}

Therefore, 
$$ \E_n\left[ \left(\max_{f_n \in \HT_v^n(f^\sharp, \eps n)}  \langle \partial g_n|_{\hexagon_v^n \cap (U_{n, \eps} \cup U'_{n, \eps})}, \partial f_n|_{\hexagon_v^n \cap (U_{n, \eps} \cup U'_{n, \eps})} \rangle -  n^2 S_{v, \edge}(f^\sharp)\right)_+\right] \leq 
 \E_n Y
+ o_\eps(n^2).$$


In order to prove $\E_n Y \leq o_\eps(n^2)$, it suffices to take $\NN = \NN(T')$ and prove the following.
\begin{claim}\lab{cl:6}  
Let   $$Y' = \sum_{T'} \max_{\partial L_{T'} \in \NN} \left| \max_{\bar{f}_n|_{T'} \in \HT_{v, T'}^n(L_{T'}, \varepsilon  \sidelength(T'))}   \langle \partial g_n|_{T'}, \partial \bar{f}_n|_{T'} \rangle  - \E_n  \max_{\bar{f}_n|_{T'} \in \HT_{v, T'}^n(L_{T'},  \varepsilon  \sidelength(T'))}  \langle \partial g_n|_{T'}, \partial \bar{f}_n|_{T'} \rangle\right|,$$
where the summation over $T'$ is over all dyadic triangles that have sidelength in the interval $[2^{-k} C_0, C_0]$ and are proper subsets of some triangle of the $\ell_1$-mesh.
Then, $$\E_n Y'\leq o_\eps(n^2). $$
\end{claim}
\begin{proof}
The variance of $ \langle \partial g_n|_{T'}, \partial \bar{f}_n|_{T'} \rangle$, a log-concave random variable,  is less than $\frac{C |T'|^2}{\log^A |T'|}$ for any fixed constant $A$ by Proposition~\ref{prop:8.17}. The number of distinct gradients that the linear approximants in $\NN$ can have is bounded above by $C \sidelength(T')^2$. Therefore, the expectation of $$\max_{\partial L_{T'} \in \NN} \left| \max_{\bar{f}_n|_{T'} \in \HT_{v, T'}^n(L_{T'}, \varepsilon  \sidelength(T'))}   \langle \partial g_n|_{T'}, \partial \bar{f}_n|_{T'} \rangle  - \E_n  \max_{\bar{f}_n|_{T'} \in \HT_{v, T'}^n(L_{T'},  \varepsilon  \sidelength(T'))}  \langle \partial g_n|_{T'}, \partial \bar{f}_n|_{T'} \rangle\right|,$$
is bounded above by $\frac{C |T'|}{\log^{A'} |T'|}$ for any fixed constant $A'$. The claim now follows from the linearity of expectation and the fact that $\sum_{T'} \frac{C |T'|}{\log^{A'} |T'|}$ is less than $\frac{C n^2}{\log(C_0/2^k)^A}$ which is in turn less than $o_\eps(n^2).$
\end{proof}

The lemma is proved.
\end{proof}

We have the following lemma.
\begin{lemma}\lab{lem:cl6}
$$ \sum_{\edge \not \subset U_{n,\eps} \cup U'_{n, \eps}} |\weight(\edge)| \ll \eps^{1/3} n^2$$
with overwhelming probability, where the sum is over all blue or green lozenges in $U$ or $U'$ that are not contained in $U_{n, \eps}$ or $U'_{n, \eps}$. 
\end{lemma}
\begin{proof}
The proof follows calculations on \cite[Page 1149]{NST}.
 The sum in the claim telescopes to be bounded by a sum of 
\begin{itemize}
\item[(i)] $O(\eps n)$ expressions of the form $\lambda_{k,1} - \lambda_{k,k}$ or $\mu_{k,1} - \mu_{k,k}$ for various $1 \leq k \leq n$;
\item[(ii)] $O(n)$ expressions of the form $\lambda_{k,1} - \lambda_{k,i}$ or $\mu_{k,1} - \mu_{k,i}$ for various $1 \leq i \leq k \leq n$ with $i = O(\eps n)$; and
\item[(iii)] $O(n)$ expressions of the form $\lambda_{k,i} - \lambda_{k,k}$ or $\mu_{k,i} - \mu_{k,k}$ for various $1 \leq i \leq k \leq n$ with $k-i = O(\eps n)$.
\end{itemize}
By eigenvalue rigidity (Lemma \ref{lem:TV}), all the expressions in (i) are of size $O(n)$ with overwhelming probability (\ie probability $1 - O(n^{-C})$ for any $C > 0$), while all the expressions in (ii), (iii) are of size $O( \eps^{1/3} n )$ with overwhelming probability.  The lemma follows.
\end{proof}

\begin{lemma}\lab{lem:9.8}
$$\limsup_{n \ra \infty}  \E_n  n^{-2} h_n(v_n)  \leq  \sup_{f^\sharp \in \AHT_v^\infty} S_v(f^\sharp).$$
\end{lemma}
\begin{proof}
Note that by Lemma~\ref{lem:cl6},
$$|\limsup_{n \ra \infty}  \E_n  n^{-2} \max_{f_n \in \HT_v^n(f^\sharp, \eps n)}  \langle \partial g_n|_{\hexagon_v^n \setminus  (U_{n, \eps} \cup U'_{n, \eps})}, \partial f_n|_{\hexagon_v^n \setminus (U_{n, \eps} \cup U'_{n, \eps})} \rangle|  \leq  o_\eps(1).$$

Next, by Lemma~\ref{lem:12}, 
$$\limsup_{n \ra \infty}  \E_n  n^{-2} \max_{f_n \in \HT_v^n(f^\sharp, \eps n)}  \langle \partial g_n|_{\hexagon_v^n \cap (U_{n, \eps} \cup U'_{n, \eps})}, \partial f_n|_{\hexagon_v^n \cap (U_{n, \eps} \cup U'_{n, \eps})} \rangle  \leq  S_{v, \edge}(f^\sharp) + o_\eps(1).$$

Also, by Lemma~\ref{lem:8},
$$\limsup_{n \ra \infty}  \E_n  n^{-2} \max_{f_n \in \HT_v^n(f^\sharp, \eps n)}  \langle \partial g_n|_{Q_n}, \partial f_n|_{Q_n} \rangle_\De  \leq  S_{v, \De}(f^\sharp) + o_\eps(1).$$

Finally, by Lemma~\ref{lem:9}, $\E_n n^{-2} \weight'(\hexagon_v^n)$ converges as $n \ra \infty$.

Putting the preceding four facts together, we see that 
\beq \Lambda := \limsup_{n \ra \infty}  (\E_n  n^{-2} \max_{f_n \in \HT_v^n(f^\sharp, \eps n)}  \langle \partial g_n|_{\hexagon_v^n }, \partial f_n|_{\hexagon_v^n } \rangle & + & \nonumber \\
 \E_n  n^{-2} \max_{f_n \in \HT_v^n(f^\sharp, \eps n)}  \langle \partial g_n|_{Q_n}, \partial f_n|_{Q_n} \rangle_\De &+& \nonumber\\
\E_n n^{-2} \weight'(\hexagon_v^n)) & \leq & S_v(f^\sharp) + o_\eps(1).\lab{eq:3.2}\eeq

Recall the following.
\begin{theorem}[Arzelà–Ascoli]
    Let \( K \) be a compact metric space, and let \( \{ f_n \} \) be a sequence of real-valued functions on \( K \). If \( \{ f_n \} \) is uniformly bounded and equicontinuous, then there exists a subsequence \( \{ f_{n_k} \} \) that converges uniformly to a continuous function \( f: K \to \mathbb{R} \).
\end{theorem}

Together with Claim~\ref{cl:6}, (\ref{eq:3.2}) and the fact that $\AHT_v^\infty$ is compact in $L^\infty(\hexagon_v^\infty)$ (by the fact that $\AHT_v^\infty$ is closed and  the theorem of Arzelà–Ascoli, since these asymptotic height functions are Lipschitz and bounded) gives us, for any positive $\eps,$ $$\limsup_{n \ra \infty}  \E_n n^{-2} h_n(v_n) \leq   \sup_{f^\sharp \in \AHT_v^\infty} S_v(f^\sharp) + o_\eps(1).$$
This implies that 
$$\limsup_{n \ra \infty}  \E_n n^{-2} h_n(v_n) \leq  \sup_{f^\sharp \in \AHT_v^\infty} S_v(f^\sharp).$$

\end{proof}
\subsection{Lower bound on $h(v)$}

\begin{definition}\lab{def:dagger}
Let $f^\ddagger$ be the (unique) asymptotic height function on $\hexagon_v^\infty$ that is linear on every horizontal line segment contained in $\hexagon_v^\infty$. Likewise let $f^\ddagger_n$ be the asymptotic height function on $\hexagon_v^n$ that is linear on every horizontal line segment contained in $\hexagon_v^n$.
\end{definition}
Let $f^{\sharp, \de}$ denote the asymptotic height function on $\hexagon_v^\infty$, which equals $\de f^\ddagger + (1 - \de)f^\sharp$. 
Similarly, 
let $f_n^{\sharp, \de}$ denote the asymptotic height function on $\hexagon_v^n$, which equals $\de f_n^\ddagger + (1 - \de)f^\sharp_n$.
\begin{observation}\lab{obs:dagger}
Using the terminology of Definition~\ref{def:14} for the meaning of $f\vee g$, where $f$ and $g$ are height functions, we define $f \vee g$ to be the function that in the upper trapezoid takes the maximum of $f$ and $g$ and in the lower trapezoid takes the minimum of $f$ and $g$. Let the standard height function $f_{standard}$ be that corresponding to the tiling in Figure~\ref{fig:standard}. Restricted to the upper trapezoid, $f_{standard}$ is a pointwise lower bound, while restricted to the lower trapezoid, $f_{standard}$ is a pointwise upper bound on all possible height functions in $ \HT_v^n$.
We observe that for any positive $\eps$ and $\de$, for sufficiently large $n$, $\HT_v^n(f^{\sharp, \de}, \eps)$ is nonempty because
there exists a height function $f_n$ (which doesn't meet the boundary conditions) such that   $f_n - f_n^{\sharp, \de}$ on the upper trapezoid of $\hexagon_v^n$ belongs to  $[-3, -6]$  and on the lower trapezoid of $\hexagon_v^n$ belongs to $[3, 6]$ by Proposition~\ref{prop:6.4}. Then, $f_n' = f_{standard} \vee  f_n \in \HT_v^n(f^{\sharp, \de}, \eps)$ (and, in particular, does meet the boundary conditions).
\end{observation}

 \begin{figure}
  \begin{center}
  \includegraphics[scale=0.60]{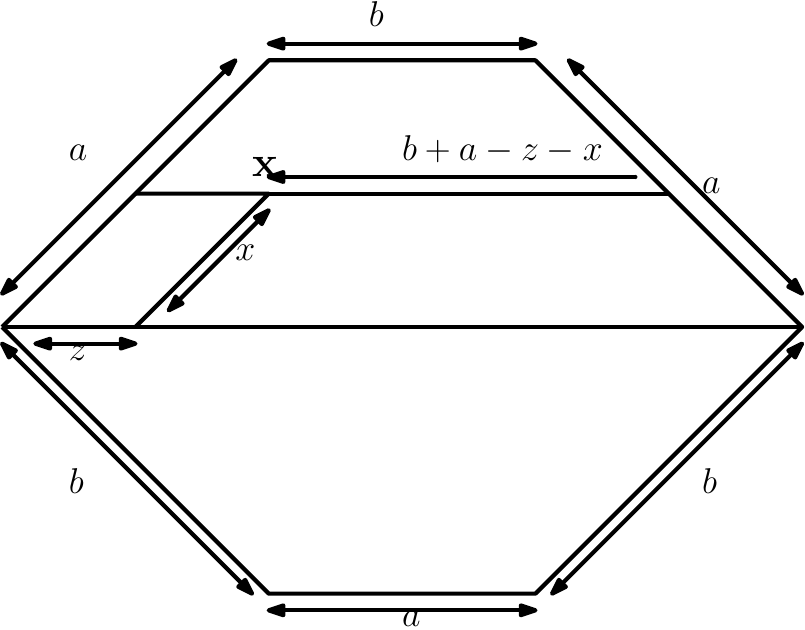}
  \caption{ Diagram accompanying the calculation of the gradient of $f^\ddagger$ at $\x$ in Lemma~\ref{lem:9.9}.} 
  \label{fig:dagger}
  \end{center}
  \end{figure}

\begin{lemma} \lab{lem:9.9} Assume that both the upper and lower trapezoids in $\hexagon_v^\infty$ contain Euclidean balls of positive radius. Let $\de = \frac{\eps}{10}$. 
Then the following inequality holds.
$$\liminf_{n \ra \infty}  \E_n n^{-2} \max_{f_n \in \HT_v^n(f^{\sharp, \de}, \eps)}  \langle \partial g_n|_{\hexagon_v^n \cap  (U_{n,\eps} \cup U'_{n,\eps})}, \partial f_n|_{\hexagon_v^n \cap (U_{n,\eps} \cup U'_{n,\eps})} \rangle  \geq  S_{v, \edge}(f^\sharp) - o_\eps(1).$$
\end{lemma}
\begin{proof}
  We consider a lengthscale $C_0$ (that is a large absolute constant for fixed $\eps$ as $n \ra \infty$) that is a power of $2$ (for this proof, we assume that we have $\eps C_0  > \psi(C_0)$), and take a piecewise linear interpolation $\bar{f}_n$ using a triangulation at the lengthscale $C_0$ into dyadic triangles and the value of $f^{\sharp, \de}_n$  for the endpoints in the triangulation.   
When $n$ is large, $$ \HT_v^n(f^\sharp_n, n\eps) \supset \HT_v^n(\bar{f}_n, \eps C_0),$$ therefore 
$$ \E_n n^{-2} \max_{f_n \in \HT_v^n(f^\sharp_n,n \eps)}  \langle \partial g_n|_{\hexagon_v^n \cap  (U_{n,\eps} \cup U'_{n,\eps})}, \partial f_n|_{\hexagon_v^n \cap  (U_{n,\eps} \cup U'_{n,\eps})} \rangle  \geq$$
$$ \E_n n^{-2} \max_{f_n \in \HT_v^n(\bar{f}_n, \eps C_0)}  \langle \partial g_n|_{\hexagon_v^n \cap (U_{n,\eps} \cup U'_{n,\eps})}, \partial f_n|_{\hexagon_v^n \cap  (U_{n,\eps} \cup U'_{n,\eps})} \rangle. $$
 Given a dyadic triangle $\De$, denote by $|\De|$ the number of lattice points in  $\De.$ 
 
 \begin{claim} For two subsets $X$ and $Y$ of $\R^2$, let  $$\dist(X, Y):= \inf\limits_{(x, y)  \in X\times Y} \|x - y\|.$$
 For any dyadic $C_0$-triangle $T \subseteq \hexagon_v^\infty\cap (U_{n,\eps} \cup U'_{n,\eps})$, we have $\dist(L_{T, \bar{f}}, \partial K) > c\eps^2.$ 
 \end{claim}
 \begin{proof}
 This follows from our assumption that both the upper and lower trapezoids in $\hexagon_v^n$ contain Euclidean balls of radius $10 \eps$, together with a computation involving the diagram in Figure~\ref{fig:dagger}. An explicit calculation shows for $\x$ in the upper trapezoid (the situation when $\x$ is in the lower trapezoid is analogous) that $$ f^\ddagger(\x) = x + \frac{z(2a-b - 2x)}{a+b - x} = x + 2z - \frac{3bz}{a+b-x}.$$
 When $x < a - \eps$, the directional derivative $\partial_x f^\ddagger(\x) < 1 - \eps$, and the directional derivative $\partial_z f^\ddagger(\x) = 2 - \frac{3b}{a+b-x} > 2 - \frac{3b}{b+\eps} > -1 + \frac{\eps}{10}.$ Together with symmetry about the vertical axis, and the fact that 
 $f^{\sharp, \de}$ denotes the asymptotic height function, which equals $\de f^\ddagger + (1 - \de)f^\sharp$,
 these inequalities imply the claim.
 \end{proof}
 
 On all triangles $T$ on which the tilt $L_{T, \bar{f}}$ of $\bar{f}$ is at a distance $\bar{\eps} > \eps$ from the boundary of $K$,  we can patch in pieces of height functions corresponding the maximum weight matchings of that tilt using Proposition~\ref{prop:3} and an application of Proposition~\ref{prop:conv-final} to yeild a height function with $$\E_n \max_{\bar{f}_n \in \HT_{v, T}^n(L_{T, f_n}, \eps C_0)}  \langle \partial g_n|_{T}, \partial \bar{f}_n|_{T} \rangle \geq  \int_{\x \in T}(-1) \sigma_\edge(\rho(\x), L_T) d\x - |T|o_\eps(1).$$  Therefore, summing over all $T$ and bounding the contributions from the remaining rhombi by $o_\eps(n^2)$ as in the previous proof and using 
 Proposition~\ref{prop:conv-final},  we see that 
 \beqs \E_n n^{-2} \max_{f_n \in \HT_v^n(f^\sharp_n, \eps C_0)}  \langle \partial g_n|_{\hexagon_v^n \cap (U_{n,\eps} \cup U'_{n,\eps})}, \partial f_n|_{\hexagon_v^n \cap (U_{n,\eps} \cup U'_{n,\eps})} \rangle  & \geq & (n^{-2}\sum_T  \int_{\x \in T}(-1) \sigma_\edge(\rho(\x), L_T) d\x) - o_\eps(1).\eeqs

 By Lemma~\ref{lem:6}, 
 the RHS is greater or equal to $$(n^{-2}\int_{\x \in \hexagon_v^n\cap (U_{n,\eps} \cup U'_{n,\eps})}(-1) \sigma_\edge(\rho(\x), \partial f^\sharp(x)) d\x) - o_\eps(1),$$
 which in turn is greater or equal to $S_{v, \edge}(f^\sharp) - o_\eps(1).$
 
\end{proof}

\begin{lemma}\lab{lem:9.10}
$$\liminf_{n \ra \infty}  \E_n  n^{-2} h_n(v_n)  \geq  \sup_{f^\sharp \in \AHT_v^\infty}S_v(f^\sharp).$$
\end{lemma}
\begin{proof}

It suffices to prove that for any fixed $f^\sharp \in \AHT_v^\infty$, we have $ \liminf_{n \ra \infty}  \E_n  n^{-2} h_n(v_n)  \geq S_v(f^\sharp).$ We proceed to prove this.
By Lemma~\ref{lem:8},
$$\liminf_{n \ra \infty}    n^{-2} \min_{f_n \in \HT_v^n(f^\sharp, \eps n)}  \langle \E_n g_n, \partial f_n \rangle_\De  \geq  S_{v, \De}(f^\sharp) - o_\eps(1).$$
By Lemma~\ref{lem:9}, $n^{-2} \weight'(\hexagon_v^n)$ is a log-concave random variable that concentrates around its expectation under $\p_n$ as $n \ra \infty$ with a standard deviation of $O( n^{-1} \log^{O(1)} n) $, and in the limit as $n \ra \infty$, $$\E_n n^{-2} \weight'(\hexagon_v^n)$$ converges.
Lemma~\ref{lem:cl6} together with Lemma~\ref{lem:9.9} completes the proof.
\end{proof}

\begin{proof}[Proof of Theorem~\ref{theorem:main}] This is immediate from Lemma~\ref{lem:9.8} and Lemma~\ref{lem:9.10}. 
\end{proof}
    \section{Proofs involving the bead model}\lab{sec:Bead}
   We will need to work with a ``fixed index" version of a family of point processes on $\mathbb{R} \times \mathbb{Z}$, known as bead processes that were defined by Boutillier \cite{boutillier}. 
  The configurations that constitute the sample space are defined below.
   
  The configurations are countable subsets of $\R \times \Z$ (elements of which will be called beads) that satisfy the following two properties:
  \begin{enumerate}
  \item They are locally finite in that the number of beads in each finite interval of a thread is finite.
  \item Between two consecutive beads on a thread, there is exactly one bead on each neighboring thread. 
  \end{enumerate}
  
  

  It has been proved by Adler, Nordenstam and Van Moerbeke in \cite[Corollary 1.4]{ANV14} that the GUE minor process converges to a version of the bead model in any square of constant size around a fixed energy in the bulk.

\begin{lemma}[Covariance bounds for patches centered at a fixed index]\lab{lem:8.16}
Let $g_{m}$ denote an $2m+1\times 2m+1$ patch corresponding to the eigenvalues in a $2m+1\times 2m+1$ box of a GUE minor process around some fixed $(\lceil i_0 n \rceil, \lceil j_0 n\rceil)$, where $(i_0, j_0) \in (0, 1)^2$ and $  j_0 < i_0$, and $m = o(n).$ Thus, the square consists of all indices $(i, j)$ such that  $\max(|i- \lceil i_0 n\rceil|, |j-\lceil j_0 n\rceil|) \leq m.$ Geometrically, this square is similar to a green lozenge in one of the tilings, such as Figure~\ref{fig:typical}.
 Let $M$ be the $4m^2 \times 4m^2$ covariance matrix of $\partial g_m$. Then, for any $C > 0$,  $$\|M\|_{op} = O\left(\frac{m^2}{\log^C m}\right).$$ 
\end{lemma}
\begin{proof}
This follows from Theorem~\ref{theorem:Tao4}, taking the parameter $A$ in it to be sufficiently large.

\end{proof}

\noindent {\it Proposition~\ref{prop:8.17} restated:} Let $g_{m}$ denote an $(2m+1)\times (2m+1)$ patch corresponding to the eigenvalues in a $(2m+1)\times (2m+1)$ box of a GUE minor process around some fixed $(\lceil i_0 n \rceil, \lceil j_0 n\rceil)$, where $(i_0, j_0) \in (0, 1)^2$ and $  j_0 < i_0$, and $m = o(n).$ Thus, the square consists of all indices $(i, j)$ such that  $\max(|i- \lceil i_0 n\rceil|, |j-\lceil j_0 n\rceil|) \leq m.$
Let $B_m(a_\infty, m\eps)$ denote the set of all Lipschitz height functions on a $(2m+1)\times (2m+1)$ piece of the triangular lattice that are contained in an $L^\infty$ radius of $n\eps$ centered around $a_\infty$. Denote by $\partial g_m$, the array of interlacing gaps of the form $\la_{ij} - \la_{i-1\, j}$, where $\la_{ij}$ is in $g_{m}$. For any fixed $\eps > 0$ and fixed $C > 0$, for all affine functions $a_\infty \in K $, 
$$   \var\left( \max_{a_m \in B_m(a_\infty,  m\eps)}  \langle \partial g_{m}, \partial a_m\rangle\right) = O\left(\frac{m^4}{\log^C  m}\right).$$

\begin{proof}[Proof of Proposition~\ref{prop:8.17}]
Let $M_n$ denote the covariance matrix of the interlacing gaps of the form $\la_{ij} - \la_{i-1\, j}$ in $\partial g_{n}$. By Theorem~\ref{theorem:prekopa}, $\partial g_n$ is log-concave and so are all its marginals. By Lemma~\ref{lem:17}  
and Cauchy-Schwartz,  
$$ \var\left( \max_{a_m \in B_m(a_\infty,  m\eps)}  \langle \partial g_{m}, \partial a_m\rangle\right) < Cm^2 \|M_m\|_{op} \log (m + 2)  .$$
It thus suffices to prove that $\|M_m\|_{op} = O(\frac{m^{2} }{\log^C m}),$ for some sufficiently large $C$ which follows from Lemma~\ref{lem:8.16}.
\end{proof}


  

  


 \section{Proofs involving height functions}\lab{sec:Height}
  
  The following result is from Thurston  \cite{Thurston90} (see Theorem 1.4 of \cite{GorinBook}).
  
  \begin{proposition}[Thurston] \lab{prop:1}Let $\RR$ be a simply connected domain, with boundary $\partial \RR$ on the triangular lattice. Then $\RR$ is tileable if and only if both conditions hold:
  \ben \item  One can define $f$ on $\partial \RR$, so that $f(v) - f(u)=1$ when ever $ u \ra  v$ is an edge of $\partial \RR$, such that $v - u$  is a unit vector in one of the positive directions.
  \item  The above $f$ satisfies
  $\forall  u,v \in \partial \RR : f(v) - f(u) \leq d_\RR(u,v).$
  \een
  \end{proposition}
  
  We note the following result from (Corollary 1.5, \cite{GorinBook}).
  \begin{proposition}\lab{prop:2} Let $\RR$ be a possibly non-simply connected domain, with boundary $\partial \RR$ on the triangular lattice. As in the simply connected case, for two points $u, v$ in the same connected component of $\RR$, let the asymmetric distance function $d_{\RR}(u, v)$ be defined as  the minimal total length over all positively oriented paths from $u$ to $v$ staying within (or on the boundary) of $\RR$.
  Suppose  $f$ is defined along $\partial \RR$ satisfying the conditions 1. and 2. in Proposition~\ref{prop:1}; this is uniquely defined up to constants $c_1,c_2,\dots, c_\ell$ corresponding to constant shifts along the $\ell$ pieces of $\partial \RR$. Then $f$ (for some fixed values of the constants $c_i$) extends to a height function on $\RR$ if and only if for every $u$ and $v$ in $\partial \RR$:
  \ben \item[(A)] $d_\RR(u,v) - f(v)+f(u) \ge 0,$ \item[(B)] $d_\RR(u,v) - f(v)+f(u) \equiv 0 \,  (\mathrm{mod} \,3).$\een
  \end{proposition}


  \begin{figure}
      \begin{center}
      \includegraphics[scale=0.60]{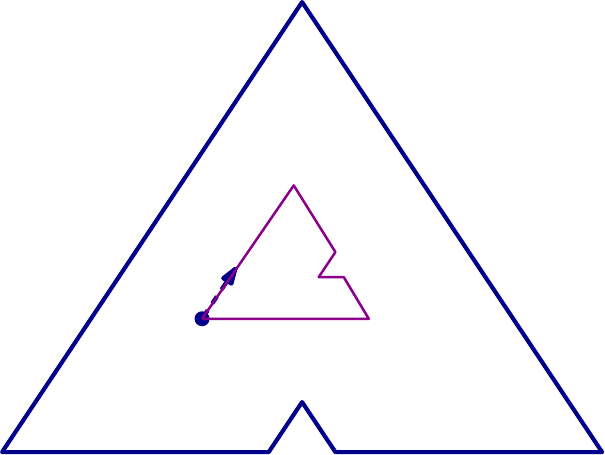}
      \caption{Values are specified on the boundary of an annular region.}\label{fig:geodesic1}
      \end{center}
  \end{figure}

{\it   Definition~\ref{def:14} restated:}  Given two asymptotic height functions $f$ and $g$ for $\hexagon_v^n$ we define $f \lor g$ to be the function that in the upper trapezoid takes the maximum of $f$ and $g$ and in the lower trapezoid takes the minimum of $f$ and $g$.
  Likewise, we define $f \land g$ to be the function that in the upper trapezoid takes the minimum of $f$ and $g$ and in the lower trapezoid takes the maximum of $f$ and $g$.
  For a collection $\FF$ of asymptotic height functions, we define $\bigvee_{f \in \FF} f$ to be the function that in the upper trapezoid takes the maximum of the functions in $\FF$ and in the lower trapezoid takes the minimum of the functions in $\FF$. Likewise, we define $\bigwedge_{f \in \FF} f$ to be the function that in the upper trapezoid takes the minimum of the functions in $\FF$ and in the lower trapezoid takes the maximum of the functions in $\FF$.
  We say $f \succeq g$ and $g \preceq f$ if $f \lor g = f$ and $f \land g = g$.
  
  Note that if $\FF$ is a set of asymptotic height functions for $\hexagon_v^n$, then $\bigvee_{f \in \FF} f$ and $\bigwedge_{f \in \FF} f$ are also asymptotic height functions. 
  
\noindent {\it Proposition~\ref{prop:6.4} restated:}  Let $f$ be an asymptotic height function for $\hexagon_v^n$. There exists a height function $f_n$ for $\hexagon_v^n$ such that $\max_v |f(v) - f_n(v)| < 3$.
  \begin{proof}[Proof of Proposition~\ref{prop:6.4}]
    This proof is modeled on that of \cite[Proposition 3.2]{CKP} by Cohn, Kenyon and Propp.
    We begin with the following claim.
    \begin{claim}
    Let $v$ be a lattice point in the upper trapezoid. Then there exists a height function $f_n$ such that $f_n \preceq f$  and $f(v) \geq f_n(v) > f(v) - 3$. Similarly, if $v$ is a lattice point in the lower trapezoid, then there exists a height function $f_n$ such that $f_n \preceq f$ and $f(v) \leq f_n(v) < f(v) + 3$.
    \end{claim}
    \begin{proof}
    We prove the claim when $v$ is in the upper trapezoid, since the other case in analogous. Let $f_{standard}$ denote the height function corresponding to the standard tiling of $\hexagon_v^n$. 
    Define a function $g$ (which will later be shown to be a height function) as follows. Let the value of $g_\up$ at $v$ to be the unique value in $(f(v) - 3, f(v)]$ congruent modulo $3$ to $f_{standard}(v)$. Let $\RR$ be the upper trapezoid of $\hexagon_v^n$. We define $g_\up$ 
    on $\RR$ by $g_\up(w) = \min_{u \in \{v\} \cup (\partial \RR\setminus Q_n)} f(u) + d_\RR(u, w)$, $Q_n$ being the equator of $\hexagon_v^n$. On the lower trapezoid we define $g_\lo$ to be the unique maximal extension that extends the boundary values (including the equator, there the values are such that their sum with $g_\up$ is a linear function of slope $1$ taking integer values on lattice points). Such an extension always exists because of Proposition~\ref{prop:2}.
    
    We define $g$ to equal $g_\up$ on the upper trapezoid and $g_\lo$ on the lower trapezoid (while on the equator, it takes two not necessarily distinct values, one corresponding to its inclusion in the upper trapezoid and another corresponding to its inclusion in the lower trapezoid).
    \end{proof}
  Let $\FF$ be the set of all height functions $g_n$ on $\hexagon_v^n$ such that $g_n \preceq f$ and let $f_n = \bigvee_{g \in \FF} g$. Then $f_n$ is a height function for $\hexagon_v^n$ and by the above claim, $\max_v |f(v) - f_n(v)| < 3$.
  \end{proof}

  \noindent {\it Proposition~\ref{prop:3} restated:}  Let $m \in \N$ and 
  let $\RR_0$ be a simply connected tileable lattice domain with boundary $\partial \RR_0$. Let $ 0 < 2 \eps' < \hat{\eps} \tilde{\eps}.$ 
  Let $\RR_1 \subseteq \RR_0$, where $\RR_1$  is a simply connected tileable lattice domain, such that every lattice point $x \in \partial\RR_1$ satisfies $d_{\RR}(x, \partial \RR_0) > \tilde{\eps}m$ and every lattice point $y \in \partial \RR_0$ satisfies $d_{\RR}(x, \partial \RR_1) > \tilde{\eps}m$. 
  Let $\RR = \RR_0\setminus \RR_1.$ Let $g_i: \partial \RR_i \ra \Z$ for $i \in \{0, 1\}$ be such that 
   $g_i$ extends to a height function on $\RR_i$ for each $i \in \{0, 1\}$ and 
    there exists an affine asymptotic height function $a_\infty \in (1 - \hat{\eps})K$ such that $\sup\limits_{\partial \RR_i} |g_i - a_\infty| < \eps' m$ for $i \in \{0, 1\}$.
  Then, there is a constant $\chi \in \{-1, 0, 1\}$ such that $g: \partial \RR \ra \R$ defined by $$g|_{\partial \RR_0} = g_0$$ and $$g|_{\partial \RR_1} = g_1 + \chi,$$ extends to a height function on $\RR$.
  \begin{proof}[Proof of Proposition~\ref{prop:3}]
  By Proposition~\ref{prop:2}, $g$ extends to a height function $f$ on $\RR$ if and only if for every $u$ and $v$ in $\partial R$:
  \ben \item[(A)] $d_\RR(u,v) - g(v)+g(u) \ge 0,$ \item[(B)] $d_\RR(u,v) - g(v)+g(u) \equiv 0 \,  (\mathrm{mod} \,3).$\een
  
  Let us check the above two conditions for every $u, v \in \partial \RR.$ Firstly, suppose both $u$ and $v$ belong to $\partial \RR_0$. Then, because $g_0$ extends to a height function on $\RR_0$, conditions (A) and (B) are satisfied. Similarly, if both $u$ and $v$ belong to $\partial \RR_1$, then because $g_1$ extends to a height function on $\RR_1$, conditions (A) and (B) are satisfied.
  We will now prove conditions (A) and (B) when $u \in \partial \RR_1$ and $v \in \partial \RR_0$. 
  We proceed to check condition (A). 
  
  Observe that by condition 2. in the statement of Proposition~\ref{prop:3}, \beq a_\infty(v) - a_\infty(u) <  (1 - \hat{\eps})d_\RR(u, v).\eeq
  and  \beq g_0(v) - g_1(u) & = & (g_0(v) - a_\infty(v)) + (a_\infty(v) - a_\infty(u)) + (a_\infty(u) - g_1(u)).\eeq Therefore, 
  \beq g(v) - g(u) = g_0(v) - g_1(u) - \chi & < & (a_\infty(v) - a_\infty(u)) + 2 \eps' m - \chi\\
  & < & (1 - \hat{\eps})d_\RR(u, v) + 2 \eps' m - \chi\\
  & < & d_\RR(u, v) - (\hat{\eps}\tilde{\eps} - 2 \eps')m - \chi.\eeq
  By the condition  that $0 < 2 \eps' < \hat{\eps} \tilde{\eps}$, 
  $g(v) - g(u) < d_\RR(u, v) - \chi$ and since $g$ and $d_\RR$ are integer valued, condition (A) is satisfied.
  We now check condition (B).
  
  Let $f_0$ be a height function that is an extension of $g_0$ to $\RR_0$ and $f_1$ be a height function that is an extension of $g_1$ to $\RR_1$.
  \begin{claim}  $f_0 - f_1$ is a constant modulo $3$ on $\RR_1$.
  \end{claim}
  \begin{proof}
  Since $\RR_1$ is connected, there exists an undirected lattice path between any $u'$ and $v'$ that are contained in $\RR_1$. Therefore, it suffices to show that if $u'$ and $v'$ are two adjacent lattice points, then $f_0(u') - f_1(u') \equiv f_0(v') - f_1(v') \, (\mathrm{mod} \, 3).$ WLOG $u' \ra v'$ is a positive direction. By the definition of the height function corresponding to a tiling, both $f_0(u') - f_0(v')$ and the $f_1(u') - f_1(v')$ above  belong to the set $\{1, -2\}$, and hence are congruent modulo $3$. Therefore, the claim follows.
  \end{proof}
  Note that by the definition of a height function corresponding to a tiling, given any $w \in \partial \RR_1$, $d_\RR(v,w) - f_0(w) + f_0(v) \equiv 0 \,  (\mathrm{mod} \,3)$.  This, by the above claim, implies that 
  for all $w \in \partial \RR_1$, $d_\RR(v,w) - f_1(w) + f_0(v) \,  (\mathrm{mod} \,3)$ is independent of $v$ and $w$ as well.
  This proves that condition (B) is satisfied.
  The case of $u \in \partial \RR_0$ and $v \in \partial \RR_1$ is analogous.
  \end{proof}


  
\section{Proofs involving the surface tension function $\sigma_\edge(\rho, a_\infty)$}\lab{sec:Surface}

\begin{lemma}\lab{lem:1}
Suppose $\rho \in \R_{>0}^2$ and $a_\infty \in (1 - \de)K$. Suppose $m_1, m_2, m_3$ are positive integers larger than $C_\de$.
Then, 
\ben
\item[A.] $  |m_3^2 \sigma_{m_3, \edge}(\rho,  a_\infty, \eps)- m_2^2\sigma_{m_2, \edge}(\rho,  a_\infty, \eps)| \leq   C_{|\rho|}|m_3^2 - m_2^2|.$ 
\item[B.] Suppose $m_1 m_2 + C_\de m_1m_2\eps =  m_3$. Then, for any $\eps_3 \geq \frac{C'_\de +  m_1 \eps}{m_3}$, $$m_3^2 \sigma_{m_3, \edge}(\rho,  a_\infty, \eps_3) \leq  m_1^2 m_2^2\sigma_{m_1, \edge}(\rho,  a_\infty, \eps)  +  C_{|\rho|, \de}(m_1^2 m_2^2\eps).$$ 
\item[C.] If $\eps \leq \eps_0$, then $ \sigma_{m_1, \edge}(\rho,  a_\infty, \eps)  \geq  \sigma_{m_1, \edge}(\rho,  a_\infty, \eps_0).$ 
\item[D.] There  is a constant $C_{|\rho|} $  and an absolute constant $\tilde{C}$ such that for all  $m_1 \eps > \tilde{C}$, $\sigma_{m_1, \edge}(\rho,  a_\infty, \eps) < C_{|\rho|}.$
\item[E.] $\sigma_{m_1, \edge}(\rho,  a_\infty, \eps) \geq 0.$
\een
\end{lemma}
\begin{proof} 
\ben 
  \item[A.]  follows from taking Lipschitz extensions and using Proposition~\ref{prop:3}.

\item[B.] follows from subdividing $[m_3] \times [m_3]$  into $m_2^2$ smaller squares of size $m_1 \times m_1$, with a border of thickness $C_\de \eps m_1$, and appealing to the fact that any set of height functions $a_{m_1}\in B_{m_1}(a_\infty, m_1\eps)$  corresponding to free boundary matchings of an $m_1\times m_1$ square  can be patched together into a  free boundary matching of a $[m_3] \times [m_3]$ square that corresponds to a height function $a_{m_3} \in B(a_\infty, m_3 \eps_3)$, using Lipschitz extension theory via  Proposition~\ref{prop:3}.


 \item[C.] follows from the fact that when the set of matchings is shrunk, the minimum weight can only increase, or stay the same. 
 
 \item[D.] follows from the fact that the interlacing gaps are in expectation $\rho_1$ and $\rho_2$ and hence  the negative of the mean of the sum of the weights of {\it all} the green lozenges and twice {\it all} the blue lozenges in an $m_1 \times m_1$ patch  is  bounded above  $ 2(\rho_1 + \rho_2)m_1^2 (1 + o_{m_1}(1))$, and that if $\tilde{C} > 3$ (the proof of this fact is entirely analogous to the proof of Proposition~\ref{prop:6.4}) there is at least one matching in the set the minimum is being taken over.
 
 \item[E.] follows from the nonpositivity of the (stochastic) edge weights $\wt(\edge)$.
 \een

\end{proof}

\begin{lemma}\lab{lem:6.3}
For any fixed $\rho \in \R_{>0}^2$ and $a_\infty \in K^o$, $\sigma_{m, \edge}(\rho,  a_\infty, \psi(m))$ is an approximately monotonically decreasing function of $m$ in the following sense. There exists positive $C_{|\rho|}$ and $M$ such that for all  $m_0 > M$ and $m_3 > m_0^2/\psi(m_0)$, $$\sigma_{m_3, \edge}(\rho,  a_\infty, \psi(m_3)) \leq \sigma_{m_0, \edge}(\rho,  a_\infty, \psi(m_0)) + C_{|\rho|}\psi(m_0).$$
\end{lemma}
\begin{proof}
By Lemma~\ref{lem:1} B., we see that the following holds. If $m_1 m_0(1 + C\psi(m_0)) = m_3$, then 
 $$ m_3^2 \sigma_{m_3, \edge}(\rho,  a_\infty, \psi(m_3)) \leq  m_1^2 m_0^2 \sigma_{m_0, \edge}(\rho,  a_\infty, \psi(m_0))  +  C_{|\rho|}(m_1^2m_0^2 \psi(m_0)).$$ 
 This implies that 
 \beq (1 + C \psi(m_0))^2 \sigma_{m_3, \edge}(\rho,  a_\infty, \psi(m_3)) \leq \sigma_{m_0, \edge}(\rho,  a_\infty, \psi(m_0))  +  C_{|\rho|}\psi(m_0).\lab{eq:conv}\eeq
 However, for every $m_3 > \frac{m_0^2(1+C\psi(m_0))}{\psi(m_0)}$, choosing $m_1 = \lfloor \frac{m_3}{m_0}\rfloor$, we see that
 $m_1 m_0(1 + C'\psi(m_0)) = m_3$ and therefore,
  \beq (1 + C' \psi(m_0))^2 \sigma_{m_3, \edge}(\rho,  a_\infty, \psi(m_0)) \leq \sigma_{m_0, \edge}(\rho,  a_\infty, \psi(m_0))  +  C_{|\rho|}\psi(m_0).\lab{eq:conv1}\eeq

\end{proof}

\begin{observation}\lab{obs:3}
If, in Lemma~\ref{lem:6.3}, in addition, we have that $\psi(m_3) \leq \frac{\psi(m_0)}{2},$ then we have
$$\sigma_{m_3, \edge}(\rho,  a_\infty, \psi(m_3))+ 2C_{|\rho|}\psi(m_3)\leq \sigma_{m_0, \edge}(\rho,  a_\infty, \psi(m_0)) + 2C_{|\rho|}\psi(m_0).$$
\end{observation}

As  consequence of B. and C. in Lemma~\ref{lem:1}, we have the following.

\begin{corollary} \lab{cor:1} For  $\rho \in \R_{>0}^2$ and $a_\infty \in K^o$, and positive $\eps$,
  $$\sigma_\edge(\rho,  a_\infty, \eps) := \lim_{m \ra \infty}  \sigma_{m, \edge}(\rho,  a_\infty, \eps),$$ exists.
 \end{corollary}

 \begin{lemma}\lab{lem:8.7}
 For any $\rho \in \R_{>0}^2$ and $a_\infty \in K^o$, as $m \ra \infty$, $ \sigma_{m, \edge}(\rho, a_\infty)$
 converges to  $\sigma_{\edge}(\rho, a_\infty)$.
 Further, for all sufficiently large $m_0$, and \ben \item for all $a_\infty \in K^o$, $$\sigma_{\edge}(\rho,  a_\infty ) 
 \leq \sigma_{m_0, \edge}(\rho,  a_\infty) + C_{|\rho|}\psi(m_0)
 ,$$
 \item for all $a_\infty \in \partial K$,
 $$\sigma_{\edge}(\rho,  a_\infty ) 
 \leq \sigma_{m_0, \edge}(\rho,  a_\infty) + o(1)$$ as $m_0 \ra \infty$.

 \een
 
Lastly, the value of $\sigma(\rho, a_\infty)$ does not depend on the specific function $\psi$ so long as it satisfies the conditions in Observation~\ref{obs:psi}.
 \end{lemma}
\begin{proof} Fix $\de > 0$. Let $m_0$ be such that $\psi(m_0) < \de$ and that $\sigma_{m_0, \edge}(\rho, a_\infty)$ is within $\de$ of  $$s := \lim\inf_m \{ \sigma_{m, \edge}(\rho, a_\infty)\}.$$ Then, for any sufficiently large $m_3$,  by Lemma~\ref{lem:6.3}, $\sigma_{m_3, \edge}(\rho, a_\infty)$ is within $C_{|\rho|}\psi(m_0) + \de \leq  (C_{|\rho|} + 1)\de$ of $s$.  This proves the convergence to a limit.
We get the inequality 1.  by letting $m_3 \ra \infty$ in Lemma~\ref{lem:6.3}. 

We now prove inequality $2.$
If we consider any maximum weight lozenge tiling $\tiling(a_m)$ where $a_m \in B_{m}(a_\infty,  m\eps)$, we see (for example in a small piece of Figure~\ref{fig:lozenge_20}) that it can be viewed in three ways: as a set of non-intersecting blue-green paths, a set of non-intersecting red-blue paths or a set of non-intersecting green-red paths. Suppose $a_m$ belongs to the boundary of $\partial K$ then asymptotically, the fraction of lozenges of one or two of these colors vanishes. Suppose first that the fraction of exactly one color which we denote by $col_1$ vanishes (we denote the other two colors by $col_2$ and $col_3$). We view $a_m$ as a set of non-intersecting $col_2$-$col_3$ paths. We fix a small integer $\de^{-1}$ and given a tiling $a_m$ remove one in every $\de^{-1}$ $col_2$-$col_3$ paths successively (replacing it by $col_1$ lozenges). This process   replaces a height function corresponding to average tilt $a_\infty$ by one of tilt $a''_\infty = (1 - \de) a_\infty + \de a'_\infty + o(1)$ where $a'_\infty$ is the tilt corresponding to tilings with only $col_1$, and $m$ tends to infinity.
Since the fraction of lozenges replaced is only $O(\de)$, by Lemma~\ref{lem:reuse} this implies that $\sigma_{m, \edge}(\rho, a_\infty) \geq \sigma_{m, \edge}(\rho, a''_\infty) + o(1),$ as $\de \ra 0$ and for each $\de >0$, $m \ra \infty$.
Suppose next that both the fraction of $col_1$ and $col_2$ asymptotically vanish.  In this case, as well,
since the fraction of lozenges replaced is only $O(\de)$, by Lemma~\ref{lem:reuse}, this implies that $\sigma_{m, \edge}(\rho, a_\infty) \geq \sigma_{m, \edge}(\rho, a''_\infty) + o(1),$ as $\de \ra 0$ and for each $\de >0$, $m \ra \infty$.
To see the last part, suppose we have two candidates for $\psi$, namely $\psi_1$ and $\psi_2$, and for the sake of contradiction, $$\lim_{m\ra \infty} \sigma_{m, \edge}(\rho, a_\infty, \psi_1(m)) < \lim_{m\ra \infty} \sigma_{m, \edge}(\rho, a_\infty, \psi_2(m)).$$
There must be some positive $\de$ and some positive $m_0$ such that for all $m_1, m_2$ larger than  $ m_0$,  
\beq\lab{eq:de1}  \sigma_{m_1, \edge}(\rho, a_\infty, \psi_1(m_1)) + \de < \sigma_{m_2, \edge}(\rho, a_\infty, \psi_2(m_2)).\eeq


Choose $m_2$ such that $m_1 | m_2$ and $m_2 \psi_2(m_2) > m_1 \psi_1(m_1)$. We see by patching up, that 
$$ \sigma_{m_1, \edge}(\rho, a_\infty, \psi_1(m_1)) + C_{|\rho|} \psi_1(m_1) \geq \sigma_{m_2, \edge}(\rho, a_\infty, \psi_2(m_1)).$$ Letting $m_1 \ra \infty$ gives us a contradiction to (\ref{eq:de1}).
\end{proof}

 \begin{figure}
  \begin{center}
  \includegraphics[scale=0.60]{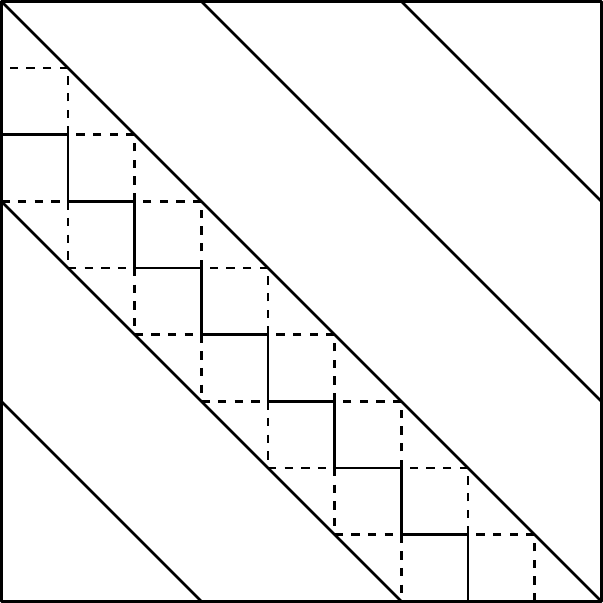}
  \caption{ Within each ``thin-strip" lattice domain bounded by transverse lines, there is a height function of close to constant slope. The (weighted) average of these slopes on the strips is the average slope on the triangle. 
  }\label{fig:washboard}
  \end{center}
  \end{figure} 
  \begin{figure}
  \begin{center}
  \includegraphics[scale=0.60]{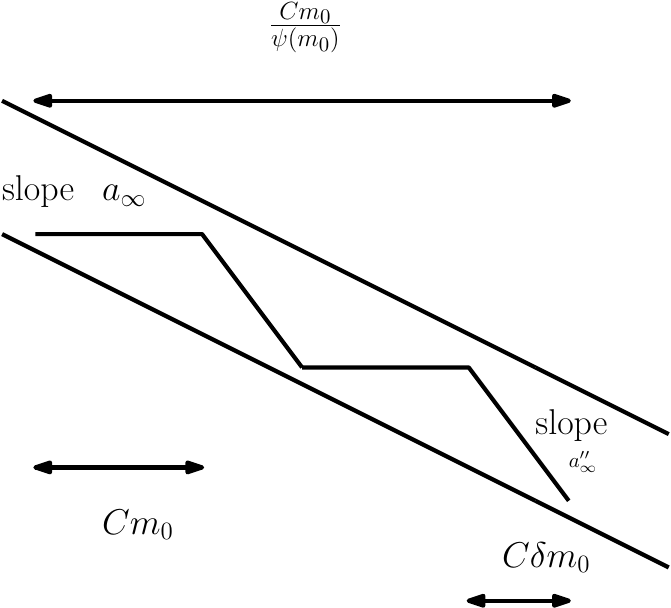}
  \caption{ Lengths of interleaving pieces in the washboard in the proof of Lemma~\ref{lem:7.13}.}
  \label{fig:m0psim0}
  \end{center}
  \end{figure} 

\begin{lemma}\lab{lem:1.5}
 For each fixed $\rho \in \R_{>0}^2$, $\sigma_\edge(\rho, a_\infty)$ is a nonnegative convex 
  function of $a_\infty$ on $K$ bounded above by $C_{|\rho|}$. Further, $\sigma_\edge(\rho, a_\infty)$ is a $C_{|\rho|}\de^{-1}$-Lipschitz function of $a_\infty$ on $(1 - \de)K$.   \end{lemma}
\begin{proof}
The nonnegativity of $\sigma_\edge(\rho, a_\infty)$ follows from the nonpositivity of lozenge weights.
In what follows, we will prove the convexity of  $\sigma_{ \edge}(\rho, a_\infty)$ for a fixed $\rho\in \R_{>0}^2$ and $a_\infty \in K^o$ by considering the ``washboard" surface in Figure~\ref{fig:washboard} (see also \cite[Figure 6.2]{Sheffield} by Scott Sheffield where this technique was used earlier). 
Let $a_\infty, a'_\infty$ and $a''_\infty$ belong to $K^o$, and suppose for some $\la \in (0, 1)$, we have $a_\infty = \la a'_\infty + (1 - \la)a''_\infty$. 
We consider an asymptotic height function on a  square of sidelength $m_3$ (that will tend to infinity) obtained by pasting together alternating bands of tilt $a'_\infty$ and $a''_\infty$, the ratio of their respective thicknesses being $\frac{\la}{1 - \la}$ (see  Figure~\ref{fig:washboard}).
Within each band, we use Proposition~\ref{prop:3} to paste in height functions corresponding to  maximum weight matchings where the boundary is only specified to lie within the prescribed neighborhood of the relevant asymptotic height function  in squares of size $m_1$ (here $m_1$  tends to infinity, but at a rate less than  $\min(\la, 1 - \la) m_3$).
We then obtain the relation $$\sigma_{m_3, \edge}(\rho, a_\infty) \leq \la \sigma_{m_1, \edge}(\rho, a'_\infty) + (1 - \la) \sigma_{m_1, \edge}(\rho, a''_\infty) + o(1)$$ as $m_1$ and $m_3$ tend to infinity. Taking limits with the thicknesses of the bands tending to infinity slower than $m_3\psi(m_3)$, we obtain the result that $\sigma_\edge(\rho, a_\infty)$ is convex on $K^o$.
To prove that $\sigma_{m, \edge}(\rho, a_\infty)$  is bounded by $C_{|\rho|}$, it suffices to bound the total expected weight of all the lozenges, which is upper bounded by $C_{|\rho|} m^2$, since there exist  periodic tilings of any rational tilt in $K$, if the period is allowed to be sufficiently large.  Any $C_{|\rho|}$-bounded nonnegative convex function is $C_{|\rho|}\de^{-1}$-Lipschitz on $(1 - \de)K$.  
\end{proof}

 \begin{lemma}\lab{lem:7.13}
 Let $\de > 0$.
Given $a_\infty, a'_\infty  \in (1-\de)K$ such that $|a_\infty - a'_\infty| < \de^2$ and any sufficiently large $m_0$, we have the following inequality,
\beq \lab{eq:unif} \sigma_{m_0, \edge}(\rho, a_\infty , \psi(m_0) ) + C_{|\rho|}\de \geq \sigma_{C+\frac{Cm_0}{\psi(Cm_0)}, \edge}(\rho, a'_\infty,  \psi(C+\frac{Cm_0}{\psi(Cm_0)})).\eeq

\end{lemma}
 \begin{proof}
 Choose $a''_\infty$ such that $a'_\infty = \de a_\infty + (1 - \de) a''_\infty.$ We consider a washboard surface as in the proof of the preceding lemma with alternating strips of thickness $Cm_0$ and $C\de m_0$ (see Figure~\ref{fig:m0psim0}). and of a total thickness $Cm_0/\psi(m_0)$. 
 This allows us to embed (using Proposition~\ref{prop:3}) squares of sidelength $m_0$ and asymptotic slope $a_\infty$ into a larger square of sidelength $\frac{Cm_0}{\psi(m_0)}$, where the fraction of squares (of sidelength $\de m_0$) that have asymptotic slope $a''_\infty$ is asymptotically less or equal to $\de$. This implies the Lemma. 
 \end{proof}

\begin{lemma} $\sigma_{m, \edge}(\rho, a_\infty)$ converges uniformly to $\sigma_{ \edge}(\rho, a_\infty)$ on compact subsets of $K^o$ for fixed $\rho$.       \lab{lem:uniform}
\end{lemma}
\begin{proof}
Any compact subset of $K^o$ is contained inside $(1 - \de)K$ for some positive $\de$.
Let $\L$ be the lattice spanned by $\hat{i}$ and $\hat{j}$.
Take a $\de^2$-net $\NN = (1 - \de)K \cap (\de/2)\L$ of $(1 - \de)K$. On $\NN$, we have uniform convergence, so there exists $m(\de)$ such that  for all $m > m(\de)$,
$\sigma_{m, \edge}(\rho, a_\infty)$ is within $\de^2$ of $\sigma_{ \edge}(\rho, a_\infty)$, $\forall a_\infty \in \NN$ for fixed $\rho$. The lemma now follows from an application of Lemma~\ref{lem:7.13}.

\end{proof}

\begin{lemma}\lab{lem:8.6}
Given  a compact subset $\C \subseteq \R_{>0}^2$, and $\eps' > 0$,  there exists an $m_0$ such that for all $m > m_0$ and all $(\rho, a_\infty) \in \C \times K,$ $$ \sigma_{m, \edge}(\rho, a_\infty)  \geq \sigma_\edge(\rho, a_\infty) - \eps'.$$
 \end{lemma}
 \begin{proof} Fix $\de > 0$.
 Let $\L$ be the lattice spanned by $\hat{i}$ and $\hat{j}$.
Take a $c\de$-net $\NN = (1 - \de)K \cap (\de/2)\L$ of $(1 - \de)K$. On $\C\times \NN$, we have uniform convergence of $\sigma_{m, \edge}(\rho, a_\infty)$ to $\sigma_\edge(\rho, a_\infty)$ by $1.$ of Lemma~\ref{lem:8.7}, so there exists $m_{\de, \C}$ such that  for all $m > m_{\de, \C}$,
$\sigma_{m, \edge}(\rho, a_\infty)$ is within $\de$ of $\sigma_{ \edge}(\rho, a_\infty)$, $\forall (\rho, a_\infty) \in \C\times \NN$.

Now consider an arbitrary $a_\infty \in K$.   Let   $a'_\infty \in K$  be such that the line segment joining $a_\infty$ to $a'_\infty$ extended, intersects $\partial K$ in the line segment joining $col_1$ and $col_2$ (two vertices of $K$) at a tilt $a''_\infty$, and $|a_\infty - a'_\infty| < C\de$.  
\begin{claim}\lab{lem:8.6-cl:3}
There exists $M$ such that for all $m > M$, uniformly over $(\rho, a_\infty) \in \C \times K$, $ \sigma_{m, \edge}(\rho, a_\infty)  \geq \sigma_{m, \edge}(\rho, a'_\infty) - o_\de(1)$.
\end{claim}
\begin{proof}
Let $a''_\infty = \de'' col_1 + (1 - \de'') col_2$ and  $a'_\infty = \de'a''_\infty + (1 - \de')a_\infty$ for some $\de'', \de' \in [0, 1]$.
Thus $$a'_\infty = a_\infty + \de'\de''(col_1 - a_\infty) + \de'(1 - \de'')(col_2 - a_\infty).$$
We first view $a_m$ as a set of non-intersecting $col_2$-$col_3$ paths. Remove one in every $(\de'\de'')^{-1}$ adjacent $col_2$-$col_3$ paths (replacing them by $col_1$ lozenges). This process   replaces a height function corresponding to average tilt $a_\infty$ by one of tilt $a_\infty = \de'\de''(col_1 - a_\infty)  a_\infty + o(1)$ as  $m$ tends to infinity. Let this height function be called $\tilde{a}_m$.
We then view $\tilde{a}_m$ as a set of non-intersecting $col_3$-$col_1$ paths.
Remove one in every $(\de'(1 -\de')')^{-1}$ adjacent $col_3$-$col_1$ paths (replacing them by $col_2$ lozenges). This process   replaces a height function corresponding to average tilt $\tilde{a}_\infty + o(1)$ by one of tilt $a'_\infty +o(1)$ as  $m$ tends to infinity. 
Since the fraction of lozenges replaced is only $O(\de)$, by Lemma~\ref{lem:reuse}, we see that their total weight in expectation does not exceed $C_{|\rho|}m^2 \de^{\frac{1}{4}}.$

This implies that $$\inf_{\rho\in \C} \sigma_{m, \edge}(\rho, a_\infty) -  \sigma_{m, \edge}(\rho, a'_\infty) \geq - o_\de(1),$$ as $m \ra \infty$.

\end{proof}
There exists at least two corners of $K$ which we arbitrarily fix and denote by  $col_1$ and $col_2$ (which will depend on $a_\infty$)  such that $|a_\infty - col_1| > c$ and $|a_\infty - col_2| > c$.
Note that there exists an element of $a'_\infty \in \NN$ (a function of $a_\infty$) such that the line segment joining $a_\infty$ to $a'_\infty$ extended, intersects $\partial K$ in the line segment joining $col_1$ and $col_2$ at a tilt $a''_\infty$, and $|a_\infty - a'_\infty| < C\de$.  
Next, let $m$ in the above claim be larger than $\max(M, m_{\de, \C})$. If we then choose $\de$ small enough that $C\de + C\de\log \frac{C}{\de} < \eps'/10$ and $$\inf_{\rho\in \C}\inf_{a_\infty \in K} \sigma_{m, \edge}(\rho, a_\infty) - \sigma_{m, \edge}(\rho, a'_\infty) > -  \frac{\eps'}{10}.$$ 
Using 1. in Lemma~\ref{lem:8.7}, we see that this gives us 
$$\inf_{\rho\in \C}\inf_{a_\infty \in K} \sigma_{m, \edge}(\rho, a_\infty) - \sigma_\edge(\rho, a'_\infty) >  - \frac{\eps'}{10} - C_{|\rho|} \psi(m).$$ 

But for each fixed $\rho$, $\sigma_\edge(\rho, a'_\infty)$ is convex and bounded above by $C_{|\rho|}$ by Lemma~\ref{lem:1.5}. It follows that 
$$\sigma_\edge(\rho, a_\infty) - \sigma_\edge(\rho, a'_\infty) \geq C\de( \sigma_\edge(\rho, a_\infty)  - \sigma_\edge(\rho, a''_\infty)) \geq - C\de C_{|\rho|}.$$
Choosing $C\de C_{|\rho|} < \frac{\eps'}{10},$ and $C_{|\rho|} \psi(m) < \frac{\eps'}{10}$,
we thus obtain 
$$\inf_{\rho\in \C}\inf_{a_\infty \in K} \sigma_{m, \edge}(\rho, a_\infty) - \sigma_\edge(\rho, a_\infty) >  - \eps'.$$ 
 \end{proof}
 \begin{lemma}\lab{lem:8.12}
 For fixed $\rho \in \R_{> 0}^2$, $\sigma_\edge(\rho, a_\infty)$  is a continuous function of $a_\infty$ on $K.$ Further, the value at the extreme vertices of $K$ equals 
 $2\rho_1$, $\rho_2$ or $0$, depending on the color of the most populous lozenges corresponding to that vertex being blue, green or red respectively.
 \end{lemma}
 \begin{proof} 
 Suppose $a \in (1 - \de)K$ for some positive $\de$. By Proposition~\ref{prop:conv-final},  for fixed  $\rho \in \R_{> 0}^2$, $\sigma_\edge(\rho, a)$ is a $C_{|\rho|}\de^{-1}$-Lipschitz function  on $(1 - \de)K$ and is therefore continuous at $a$. 
 Suppose next that $a$ is on $\partial K \setminus\{2(\hat{i} - \hat{j}), -2\hat{i}, -2 \hat{j}\}$.
  Given a side of $K$ perpendicular to the outer normal $u$, and sufficiently small $\eps$ (belonging to a neighborhood $V$ depending on $a$),
 we define the function $f_\eps(a):= \sigma_{\edge}(\rho, a - \eps u)$ for $a$ that is on $\partial K \setminus\{2(\hat{i} - \hat{j}), -2\hat{i}, -2 \hat{j}\}$. The pointwise limit as $\eps$ tends to $0$ exists because $\sup_{\eps \in V} f_\eps(a)$ is bounded and $f_\eps(a)$ is a  convex function of $a$ for fixed $\eps$ and is pointwise convergent as $\eps \ra 0$. 
 The limit of bounded locally Lipschitz functions that converge pointwise is  locally Lipschitz, and the convergence is locally uniform, therefore $\lim_{\eps \ra 0} f_\eps(a)$ is continuous at $a$.
 
 Finally, suppose  $a$ is in $\{2(\hat{i} - \hat{j}), -2\hat{i}, -2 \hat{j}\}$. In an $\eps$-neighborhood of  $a$,
 the corresponding lozenge tilings have a $O(\eps)$ fraction of lozenges of two colors and a $1 - O(\eps)$ fraction of the third color.  Note that the negatives of the weights are nonnegative. 
 We wish to prove that the total weight of the lozenges of these two minority colors tends to $0$ in expectation as $\eps \ra 0$.
 Since red lozenges have $0$ weight,  without loss of generality, we assume blue and green lozenges are the minority.
 
  Let the magnitudes (\ie negatives) of the weights of the blue and green lozenges be denoted $x_1, \dots, x_t$. Let $k \leq C \eps m^2$ be the total number of blue and green rhombi . 
 Also, the density of the GUE minor process is log-concave (see \cite{NST}) and so the pushforward under any linear map is also a log-concave density by \cite[Theorem 6]{prekopa}. Consequently, each $x_i$ is a nonnegative random variable, that has a density that is log-concave (being a limit of log-concave random variables).
 Therefore, by Lemma~\ref{lem:reuse},
\begin{eqnarray} \E\max\limits_{|S| \leq k} \sum\limits_{i \in S} x_i  \leq  o_\eps(m^2).\lab{eq:7.5o}
\end{eqnarray}
 

 Therefore, the  minimum weight matchings with these proportions converge to $2\rho_1$, $\rho_2$ or $0$, depending on the color of the most populous lozenges being blue, green or red respectively, and so for all sufficiently large $m$, $\liminf_{a' \ra a} \sigma_{m, \edge}(\rho, a)$ within an error of $o_\eps(1)$, equals the appropriate one among  $2\rho_1$, $\rho_2$ and $0$.
 \end{proof}

Let $\HT^m(a_\infty, m \psi(m))$ denote the set of height functions in an $m \times m$ square similar in orientation to a green lozenge (linearly transformed to be a square), that are within $m \psi(m)$ of the affine asymptotic height function $a_\infty$ in the sup norm. Suppose $m_1 | m_2$ and that $m_2$ is an integer power of $2$.
Let $ f \in \HT^{m_2}(a_\infty, \psi(m_2)).$ Let $\CZ(f, m_1)$ denote the decomposition of $\Box_{m_2}$ obtained by the following procedure.
Subdivide the domain of $f$ \ie $\Box_{m_2}$ into congruent squares of side length $m_1$.
Set $\varepsilon =\eta = \psi(m_1)$, and $k = C^\sharp \log^{\frac{1}{2}}m_1.$ These values satisfy the condition
$\varepsilon^4 \eta k \leq C^\sharp$ imposed by Theorem~\ref{thm:Feff}.
On each square of side length $m_1$, we apply the CZ decomposition
from Theorem~\ref{thm:Feff} (after rescaling it to a unit square, and rescaling the restriction of $f$ to the square by the same factor).
Now define \beq \phi_{m_2, m_1}(\rho, f) := \sum_{\Box'} (\frac{m'^2}{m_2^2}) \sigma_{m', \edge}(\rho, L_{\Box'}, 2 \psi(m'))\lab{eq:m2}\eeq where the sum is over dyadic squares in the CZ decomposition and $L_{\Box'}$ is the average slope of  $f$ on $\Box'$, \ie $$\frac{\int_{\Box'} \partial f(\x) d\x}{m'^2}, $$   and  $m'$ is the side length of $\Box'$. 
\begin{definition}
Let $\AHT^{m}(a_\infty, m \eps)$ denote the asymptotic height functions on a $[m] \times [m]$ square that lie within an $L^\infty$ ball of radius $m\eps$ around a linear asymptotic height function $a_\infty$.
\end{definition}

Let \beq \gamma_{m_1}(\rho, a_\infty) := \liminf\limits_{\substack{m_2 \ra \infty\\m_1|m_2}} \inf\limits_{f \in \AHT^{m_2}(a_\infty, m_2 \psi(m_2))} \phi_{m_2, m_1}(\rho, f).\lab{eq:g1}\eeq

\begin{lemma}\lab{lem:sigma-cont2} There exists a function $\bar{\psi}: \N \ra [0, 1]$ such that $\lim_{m \ra \infty} \bar{\psi}(m) = 0$, such that 
for all $\rho \in \R_{>0}^2$ and $a_\infty \in K^o$, for all $m_1$ that are integer powers of $2$,
$$\sigma_\edge(\rho, a_\infty) \geq\gamma_{m_1}(\rho, a_\infty)  - C_{|\rho|} \bar{\psi}(m_1).$$
\end{lemma}
\begin{proof}
Let $m_2$ be a power of $2$ that is an integer multiple of $m_1$.
We will cover $\HT^{m_2}(a_\infty, m_2 \psi(m_2))$ using a small number of sets defined below.
Consider $f \in \HT^{m_2}(a_\infty, m_2 \psi(m_2))$. We associate an invariant $\chi_{m_1}(f)$ to it as follows. Look at all the squares in $\CZ(f, m_1)$. To each such square $\Box'$, associate a slope $L_{\Box'}$ by the above procedure. Now for all squares of side length $\sidelength(\Box')$, and (rounded) slope $L_{\Box'}$, maintain the count $\#( L_{\Box'}, \log(m_1/ \sidelength(\Box')))$, indexed by the approximate slope in which is an element of a set $\NN_K$ (having at most $Cm_1^2$ elements, because $m_1$ is an integer power of $2$) and the depth of $\Box'$ compared to the $m_1 \times m_1$ square. This invariant of $f$, which we denote by $\chi_{m_1}(f)$ is a map from $ \NN_K \times \N$ to $\N$.  Observe that $\phi_{m_2, m_1}(\rho, f) = \phi_{m_2, m_1}(\rho, f')$ if $\chi_{m_1}(f) = \chi_{m_1}(f')$. Next observe that by Theorem~\ref{thm:Feff}, the smallest possible square that can appear in $\CZ(f, m_1)$ has side length greater or equal to $2^{-k}m_1$, where $k\varepsilon^4 \eta \leq C^\sharp$.
Therefore, the number of possible distinct invariants  $\chi$  (for fixed $m_1, m_2$) is bounded above by 
\beq \lab{eq:card}  \left(m_2^2\right)^{|\NN_K| \frac{C^\sharp}{\varepsilon^4 \eta}}.\eeq
Given $\chi = \chi_{m_1}(f)$, we define $\FF_\chi$ to be the set of all height functions $f' \in  \HT^{m_2}(a_\infty, m_2 \psi(m_2))$ such that $\chi_{m_1}(f') = \chi$. In the rest of this proof, we abbreviate $\HT^{m_2}(a_\infty, m_2 \psi(m_2))$ to $\FF$.

Recall from Definition~\ref{def:8.1} that \begin{eqnarray*} \sigma_{m_2, \edge}(\rho, a_\infty) &=&  -\left(\frac{1}{m_2^2}\right) \E \max_{a \in \FF}  \langle \partial \MM_{m_2, \x(\rho)}, \partial a\rangle.\end{eqnarray*}

By Proposition~\ref{prop:8.17},  the variance of $\langle \partial \MM_{m_2, \x(\rho)}, \partial a\rangle$ is bounded above by $O(\frac{m_2^4}{(\log m_2)^A})$  for every absolute constant $A$. Further, 
$ \MM_{m_2, \x(\rho)}$ is log-concave, so by  \cite[Theorem 6]{prekopa},  so is 
$ \partial \MM_{m_2, \x(\rho)}$, and for since $a \in \FF$ is $m_2^2$-Lipschitz, 
for each $\chi$, $ \max_{a \in \FF_\chi}  \langle \partial \MM_{m_2, \x(\rho)}, \partial a\rangle$ is a subexponential random variable
whose $\psi_1$-norm (see Definition~\ref{def:psi1}) is bounded above by $\frac{m_2^2}{(\log m_2)^\frac{A}{2}}$ by \cite[Theorem 1.5]{Emanuel}.
Note that 
$$ -\left(\frac{1}{m_2^2}\right) \max_\chi \max_{a \in \FF_\chi}  \langle \partial \MM_{m_2, \x(\rho)}, \partial a\rangle = 
 -\left(\frac{1}{m_2^2}\right) \max_{a \in \FF}  \langle \partial \MM_{m_2, \x(\rho)}, \partial a\rangle.$$
The preceding discussion implies that 
\beqs m_2^{-2} \left( \E \max_{a \in \FF}  \langle \partial \MM_{m_2, \x(\rho)}, \partial a\rangle - \max_\chi (\E \max_{a \in \FF_\chi}  \langle \partial \MM_{m_2, \x(\rho)}, \partial a\rangle) \right) \eeqs is bounded above by the contributions $\mathrm{contrib}(\eta)$ of the BAD squares in Theorem~\ref{thm:Feff} added to a fluctuation term, namely \beq  \mathrm{contrib}(\eta) + O\left(
(\log m_2)^{-\frac{A}{2}}  \log   \left(\left(m_2^2\right)^{|\NN_K| \frac{C^\sharp}{\varepsilon^4 \eta}}\right)\right),\eeq
where the first term is bounded above by $C_{|\rho|} \bar{\psi}(m_1)$ by Lemma~\ref{lem:reuse}, and the second is due to subexponential tail bounds and (\ref{eq:card}). 

As a consequence, we see that $$ \sigma_{m_2, \edge}(\rho, a_\infty) \geq  \inf_{f \in \FF} \chi_{m_1}(f) - C_{|\rho|} \bar{\psi}(m_1).$$ Taking $m_2 \ra \infty$, we see that 
 $$ \sigma_{ \edge}(\rho, a_\infty) \geq  \gamma_{m_1}(\rho, a_\infty) - C_{|\rho|} \bar{\psi}(m_1),$$ which proves the lemma.
\end{proof}
\begin{lemma}\lab{lem:8.14}
For any fixed $a_\infty \in K$, $\sigma_{\edge}(\rho, a_\infty)$ is a continuous function of 
$\rho$ for $\rho \in \R_{>0}^2$.
\end{lemma}
\begin{proof}
\begin{claim}
When $a_\infty \in K^o$, the lemma is true.
\end{claim}
\begin{proof} 
Suppose $a_\infty \in (1-\de)K$. 
Because of Theorem~\ref{theorem:Tao3} and Proposition~\ref{prop:conv-final}  and Lemma~\ref{lem:sigma-cont2}, 
it suffices to show that given any $\eps'>0$, for all $m$ a large enough power of $2$, $\gamma_m(\rho, a_\infty) \geq \sigma_{m, \edge}(\rho, a_\infty) - \eps'$.
We already know by Proposition~\ref{prop:conv-final} that the pointwise limit of $\sigma_{m, \edge}(\rho, a'_\infty)$ as $m \ra \infty$ is bounded below by $\sigma_{\edge}(\rho, a'_\infty)$ (for all $a'_\infty \in K$) which is a convex function. It follows by considering a supporting hyperplane (which by convexity of $\sigma_\edge(\rho, a'_\infty)$ as a function of $a'_\infty$ and boundedness has a bounded slope) to the epigraph of $\sigma_\edge(\rho, a'_\infty)$ and applying Proposition~\ref{prop:conv-final} that given any $\eps'>0$, for all $m$ large enough, $\gamma_m(\rho, a_\infty) \geq \sigma_{m, \edge}(\rho, a_\infty) - \eps'$. The claim is proved.
\end{proof}
Suppose now that $a_\infty \in \partial K\setminus \{-2\hat{i}, 2\hat{j}, 2(\hat{i} - \hat{j})\}$.
We proceed by contradiction. Suppose there exists some $\eps > 0$ such that every neighborhood of $\rho$
contains some $\rho'$ such that $|\sigma_\edge(\rho', a_\infty) - \sigma_\edge(\rho, a_\infty)| > \eps.$ Let $a'_\infty \ra a_\infty$ along a straight line $L(a_\infty) $ normal to $\partial K$. Suppose $\rho'' \ra \rho$ is such that \beq\lab{eq:sigmaeps} |\liminf_{\rho'' \ra \rho} \sigma_\edge(\rho'', a_\infty) - \sigma_\edge(\rho, a_\infty)| > \eps.\eeq 

By Lemma~\ref{lem:8.12},  we can fix a sufficiently small $\de$ such that for $a'_\infty \in L(a_\infty)$ and $|a'_\infty - a_\infty| = \de$,   \beq \lab{eq:sigmaeps1} |\sigma_\edge(\rho, a'_\infty) - \sigma_{ \edge}(\rho, a_\infty)| < \frac{\eps}{4}.\eeq                                                                 Then, by (\ref{eq:sigmaeps}) and (\ref{eq:sigmaeps1}), for all sufficiently large $m_1$ (depending on $\de$), \beq \lab{eq:8.9} |\liminf_{\rho'' \ra \rho} \sigma_{m_1, \edge}(\rho'', a'_\infty) - \sigma_{ \edge}(\rho, a_\infty)| > \eps/2.\eeq

Now, $$|\liminf_{\rho'' \ra \rho} \sigma_{m_1, \edge}(\rho'', a'_\infty) - \sigma_{\edge}(\rho, a_\infty)| < $$ $$|\liminf_{\rho'' \ra \rho} \sigma_{m_1, \edge}(\rho'', a'_\infty) - \sigma_{m_1, \edge}(\rho, a'_\infty)| +| \sigma_{m_1, \edge}(\rho, a'_\infty) - \sigma_{ \edge}(\rho, a_\infty)|=$$
$$| \sigma_{m_1, \edge}(\rho, a'_\infty) - \sigma_{ \edge}(\rho, a_\infty)|<$$
$$|\sigma_{m_1, \edge}(\rho, a'_\infty) - \sigma_\edge(\rho, a'_\infty)| + 
|\sigma_\edge(\rho, a'_\infty) - \sigma_{ \edge}(\rho, a_\infty)|.$$
However, by the first statement of Lemma~\ref{lem:8.7}, 
this is bounded above by
$$o(1) + \frac{\eps}{4}$$ as $m_1 \ra \infty$.  But this contradicts (\ref{eq:8.9}).

Finally, suppose $a_\infty \in \{-2\hat{i}, 2\hat{j}, 2(\hat{i} - \hat{j})\}.$ Then, the result follows by Lemma~\ref{lem:8.12}.
\end{proof}

\noindent {\it Proposition~\ref{prop:conv-final} restated:}
 $\sigma_{m, \edge}(\rho, a_\infty)$ converges uniformly to $\sigma_{ \edge}(\rho, a_\infty)$ on compact subsets of $K^o$ for fixed $\rho$. Additionally,
 given  a compact subset $\C \subseteq \R_{>0}^2$, and $\eps' > 0$,  there exists an $m_0$ such that for all $m > m_0$ and all $(\rho, a_\infty) \in \C \times K,$ $$ \sigma_{m, \edge}(\rho, a_\infty)  \geq \sigma_\edge(\rho, a_\infty) - \eps'.$$
Further,  for each fixed $\rho \in \R_{>0}^2$, $\sigma_\edge(\rho, a_\infty)$ is a continuous nonnegative convex 
  function of $a_\infty$ on $K$ bounded above by $C_{|\rho|}$ and consequently,  for any $\de > 0$,  $\sigma_\edge(\rho, a_\infty)$ is a $C_{|\rho|}\de^{-1}$-Lipschitz function of $a_\infty$ on $(1 - \de)K$. 

Proposition~\ref{prop:conv-final}  is the consolidation of  Lemma~\ref{lem:uniform},  Lemma~\ref{lem:8.6}, Lemma~\ref{lem:1.5} and Lemma~\ref{lem:8.12} into a single proposition.

\noindent {\it Proposition~\ref{prop:sig-unif} restated: }
Let $\rho \in \R_{> 0}^2$. Let $\bar{B}(\rho, \de)$ be the closed Euclidean ball of radius $\de$ around $\rho$. Then, the function $|\sigma_\edge(\rho', a'_\infty) - \sigma_\edge(\rho, a'_\infty)|$ defined for  $(\rho', a'_\infty) \in \bar{B}(\rho, \de) \times K$ is uniformly upper bounded by $o(1)$ as $\de \ra 0$.
\begin{proof}[Proof of Proposition~\ref{prop:sig-unif}]

Suppose to the contrary. Then, there exists a positive $\eps$ such that for every $\de > 0$, there is $(\rho_\de, a'_{\infty, \de})$ such that $|\rho_\de - \rho| < \de$ and \beq\lab{eq:last1} |\sigma_\edge(\rho_\de, a'_{\infty, \de}) - \sigma_\edge(\rho, a'_{\infty, \de})| \geq \eps.\eeq
Since $K$ is compact, there is a subsequential limit $(\rho, a'_{\infty, 0})$ of these $(\rho_\de, a'_{\infty, \de})$ as $\de \ra 0$.  As $\de \ra 0$, $|\sigma_\edge(\rho_\de, a'_{\infty, \de}) - \sigma_\edge(\rho_\de, a'_{\infty, 0})| \ra 0$ by Lemma~\ref{lem:8.12}. However, 

$$ |\sigma_\edge(\rho_\de, a'_{\infty, \de}) - \sigma_\edge(\rho, a'_{\infty, \de})| \leq  |\sigma_\edge(\rho_\de, a'_{\infty, \de}) - \sigma_\edge(\rho_\de, a'_{\infty, 0})| +  |\sigma_\edge(\rho_\de, a'_{\infty, 0}) - \sigma_\edge(\rho, a'_{\infty, 0})|$$
The last two statements imply that $\limsup_{\de \ra 0}  |\sigma_\edge(\rho_\de, a'_{\infty, 0}) - \sigma_\edge(\rho, a'_{\infty, 0})| > 0$. This  contradicts  Lemma~\ref{lem:8.14} (which implies that  $\sigma_\edge(\rho', a'_{\infty, 0})$ is continuous in $\rho'$ when it is in a $\de$-neighborhood of $\rho$ and the second argument is fixed).
\end{proof}

\section{Acknowledgements}
We wish to express our gratitude to Charles Fefferman, Scott Sheffield and Terence Tao.
Charles Fefferman supplied his notes on quantitative differentiation \cite{CF}, that stated and proved Theorem~\ref{thm:Feff}. 
Scott Sheffield  suggested the use of the surface tension $\sigma_\edge$ and discussed many points with us. 
Terence Tao answered our  many queries on the GUE minor process at fixed index  in \cite{TaoPreprint} and supplied the proof of Theorem~\ref{theorem:beadR}.

 Finally, we thank Aalok Gangopadhyay for permitting us to use his beautiful simulations of random tilings and hives from \cite{Gang}.

 We were supported by a Swarna Jayanti fellowship and a grant associated with the Infosys-Chandrasekharan virtual center for Random Geometry. We also acknowledge support from the Department
of Atomic Energy, Government of India, under project no. RTI4001.
Part of this research was performed while the author was visiting the Institute for Pure and Applied Mathematics (IPAM), which is supported by the National Science Foundation (Grant No. DMS-1925919).
\bibliographystyle{alphaabbrv}

\appendix

\section{Eigenvalues of GUE}
\subsection{Concentration of eigenvalues}\lab{ssec:rep}
By Corollary 15 of Tao and Vu, \cite{TaoVu-conc},  we know the following (which in fact holds for a larger class of Wigner matrices):
\begin{lemma}\lab{lem:TV}
  Let $\tilde{W}_n$ be a matrix  drawn from the GUE, normalized so that $\E \sum_i \la_{i}^2 = n^3$.  
Then for
  any $1 \leq i \leq n$ we have
$$\p(n^{-1/3} \min(i, n - i + 1)^{1/3}|\la_i - \E\la_i| \geq T )\ll  n^{O(1)} exp(-cT^c)$$
for any $T > 0. $
\end{lemma}

\subsection{Covariance of eigengaps in GUE}

We will denote small universal constants by $c, c'$ and large ones by $C, C'$ etc,  and their precise value may depend upon the occurrence. 
When $X < CY$, we will also at times write $ X = O(Y)$ and $X \ll Y.$ 

Let $n \geq 1$ be an integer (which we view as a parameter going off to infinity). An $n \times n$ GUE matrix $M_n$ is defined to be a random Hermitian $n \times n$ matrix $M_n = (\xi_{ij} )_{1\leq i,j \leq n},$ in which the $\xi_{ij}$ for $1 \leq i \leq j \leq n$ are jointly independent with $\xi_{ij} = \xi_{ji}$ (in particular, the $\xi_{ii}$ are real-valued), and each $\xi_{ij}$  has mean zero and variance one. 
The bulk distribution of the eigenvalues $\la_1(M_n), \dots , \la_n(M_n)$ of a GUE matrix is governed by the Wigner semicircle law. Indeed, if we let $N_I(M_n)$ denote the number of eigenvalues of $M_n$ in an interval $I$,  then with probability $1 - o(1),$ we have the asymptotic
$$N_{\sqrt{n}I}(M_n) = n \int_I \rho_{sc}(u) du + o(n)$$ 
 uniformly in $I,$ where $\rho_{sc}$  is the Wigner semi-circular distribution
 $$\rho_{sc}(u) := \frac{1}{2\pi} (4 - u^2)^{\frac{1}{2}}_+.$$
The Wigner semi-circle law predicts that the location of an individual eigenvalue $\la_i(M_n)$ of a GUE matrix $M_n$ for $i$ in the bulk region $\eps n \leq i \leq (1- \eps)n,$ should be approximately $\sqrt{n}u$
 where $u = u_{i/n}$ is the classical location of the eigenvalue, given by the formula
\beq \int_{-\infty}^u \rho_{sc}(y) dy = \frac{i}{n}.\lab{eq:5}\eeq

The following result is from \cite{Tao-gap}.
\begin{theorem} \lab{theorem:Tao1}

Let $M_n$ be drawn from GUE, and let $\eps n \leq i \leq (1 - \eps)n$ for some fixed $\eps > 0.$ Then one has the asymptotic 
$$\p\left(\frac{\la_{i+1}(M_n) - \la_i(M_n)}{1/(\sqrt{n}\rho_{sc}(u))} \leq s\right) = \int_0^s p(y)dy + o(1)$$ where $p$ is the Gaudin distribution
 for any fixed $s>0,$ where $u=u_{i/n}$ is given by (\ref{eq:5}).
\end{theorem}

Note that $p$ is log-concave by Theorem~\ref{theorem:prekopa}.

   
    \subsection{Distribution of eigenvalues of GUE and its minors at ﬁxed
  index}
  The next three results are from \cite{TaoPreprint} and are due to Terence Tao.
  
  Suppose that $\lambda_1 < \dots < \lambda_N$ are the eigenvalues of an $N \times N$ GUE matrix $H$.  If we let $\lambda'_1 < \dots < \lambda'_{N-1}$ be the eigenvalues of the top left $N-1 \times N-1$ minor of this matrix (which is again a GUE matrix, but with $N$ replaced by $N-1$), then we almost surely have the Cauchy interlacing law
  $$ \lambda_i < \lambda'_i < \lambda_{i+1}$$
  for $1 \leq i < N$.  The combined process $\tilde \Sigma \subset \R \times \{N-1,N\}$ defined by
  \begin{equation}\label{td-def}
    \tilde \Sigma \coloneqq \{ (\lambda_i, N): 1 \leq i \leq N\} \cup \{ (\lambda'_i, N-1): 1 \leq i \leq N-1\}
  \end{equation}
  forms two rows of the \emph{GUE minor process}.  Like the GUE eigenvalue process, this process is known to be determinantal \cite{ANV14}, \cite{JN06}.  Among other things, it was shown in \cite{ANV14} that under suitable rescaling, this kernel in the bulk converges to the kernel of a Boutillier bead process \cite{boutillier} (with a drift parameter determined by the location in the bulk).  
  
  As an application of the above estimates, one may now also obtain a universal law in the \emph{fixed index} sense.  Define the \emph{interlacing gaps}
  $$ \tilde g_i \coloneqq \sqrt{N/2} \rhosc(\gamma_{i/N}) (\lambda'_i-\lambda_i)$$
  then we have $0 < \tilde g_i < g_i$.  Informally, the quantities $g_i,\dots,g_{i+m}, \tilde g_i, \dots, \tilde g_{i+m}$ then describe the behavior of the GUE minor process near the index $i$.
  
  The main result for two rows from \cite[Theorem 1.2]{TaoPreprint} is that
  
  \begin{theorem}[Interlacing universality]\label{interlacing-univ} Let $i, j$ lie in the bulk region \eqref{bulk-def} for some fixed $\delta > 0$ with $i = j + o(N)$, let $m$ be fixed, and let $F: \R^{2m} \to \R$ be a fixed smooth, compactly supported function.  Then
    $$ \E F(g_i,\dots,g_{i+m}, \tilde g_i,\dots,\tilde g_{i+m}) = \E F(g_j,\dots,g_{j+m}, \tilde g_j,\dots,\tilde g_{j+m}) + o(1)$$
  as $N \to \infty$.
  \end{theorem}
  
  
  The above result was for two rows of the minor process, but an adaptation of the argument also applies to any fixed number of rows of the minor process as follows.  For any $1 \leq M \leq N$, and let $\lambda^{(M)}_1 < \dots < \lambda^{(M)}_{M}$ be the eigenvalues of the top left $N-k \times N-k$ minors of $H$, thus in the previous notation we have $\lambda^{(N)}_i = \lambda_i$ and $\lambda^{(N-1)}_i = \lambda'_i$.  We then have the normalized eigenvalue gaps
  $$ g^{(M)}_i \coloneqq \sqrt{M/2} \rhosc(\gamma_{i/M}) (\lambda^{(M)}_{i+1}-\lambda^{(M)}_i)$$
  as well as the interlacing gaps
  $$ \tilde g^{(M)}_i \coloneqq \sqrt{M/2} \rhosc(\gamma_{i/M}) (\lambda^{(M-1)}_{i}-\lambda^{(M)}_i)$$
  for any $1 \leq i < M \leq N$.
  Thus, for instance, $g^{(N)}_i = g_i$ and $\tilde g^{(N)}_i = g'_i$.  Also, for any fixed $m$ and any $1 \leq i \leq i+m+1 \leq N$, $g_i,\dots,g_{i+m},\tilde g_i,\dots,\tilde g_m$ can be expressed as a linear combination (with bounded coefficients) of the $g^{(M)}_j$ for $M=N-1,N$ and $i \leq j \leq i+m$, together with a single interlacing gap $\tilde g^{(N)}_i$. The following generalization of  Theorem \ref{interlacing-univ} appears in \cite[Theorem 1.3]{TaoPreprint}.
  
  \begin{theorem}[Interlacing universality, II]\label{theorem:Tao3} Let $i, j$ lie in the bulk region \begin{equation}\label{bulk-def}
    \delta n \leq i \leq (1-\delta)n
  \end{equation} for some fixed $\delta > 0$ with $i = j + o(N)$, let $k, m \geq 1$ be fixed, and let $F: \R^{(m+1)k + k-1} \to \R$ be a fixed smooth, compactly supported function.  Then
    $$ \E F( (g^{(M)}_{i'})_{\substack{N-k < M \leq N;\\ i \leq i' \leq i+m}}, (\tilde g^{(M)}_i)_{N-k+1 < M \leq N} ) = \E F( (g^{(M)}_{j'})_{\substack{N-k < M \leq N;\\ j \leq j' \leq j+m}}, (\tilde g^{(M)}_j)_{N-k+1 < M \leq N} ) + o(1)$$
  as $N \to \infty$.
  \end{theorem}

  As an application of this universality, some non-trivial reduction in variance of linear statistics of interlacing gaps is shown to occur in \cite[Theorem 1.4]{TaoPreprint}:

  \begin{theorem}[Interlacing statistics variance]\label{theorem:Tao4} Let $i$ lie in the bulk region \eqref{bulk-def} for some fixed $\delta > 0$, let $m \geq 1$ be fixed, and let $a_1,\dots,a_m$ be complex numbers.  Then
  $$ \Var \sum_{l=1}^m a_l \tilde g_{i+l} \ll_{\delta,A} (\frac{m}{\log^A(2+m)}+o(1)) \sum_{l=1}^m |a_l|^2$$
  for any $A>0$.
  \end{theorem}
  \subsection{Tao's proof of Theorem~\ref{theorem:beadR}}\lab{ssec:Tao4.4}
  
  \begin{proof}[Proof of Theorem~\ref{theorem:beadR}]
  We will use the convergence of the GUE minor process to the bead model with parameters $\rho(\x)$ (which can be explicitly computed from $\x$ using  \cite[Theorem 1.4] {Sosoe}) in any square of constant size around a fixed energy in the bulk \cite{ANV14}, together with Theorem~\ref{theorem:Tao3}.
  
  At a fixed energy $E$ and level $N$ let $M$ vary in $[N - m, N+m]$, and let $\la_{-m, E}^{(M)} > \dots > \la_{0, E}^{(M)}\geq E >  \lambda_{1,E}^{(M)} > \dots > \la_{m, E}^{(M)}$ be the eigenvalues at level $M$ around $E$. We will normalize the scaling so that the mean gap size is $1$. Let $F$ be some smooth compactly supported multivariate function of the variables $\la_{k, E}^{(M)} - \la_{0, E}^{(N)}$, which we organize as a vector and denote by $\Lambda^E_N$. 
  At a fixed index $i$ and level $N$ let $M$ vary in $[N - m, N+m]$, and let $\la_{-m+ i}^{(M)} - \la_{ i}^{(N)} > \dots > 0 >  \lambda_{1+i}^{(M)} - \la_{ i}^{(N)} > \dots > \la_{m+ i}^{(M)}- \la_{ i}^{(N)}$ be the shifted eigenvalues around index $i$ and level $N$, which we organize as a vector and denote by $\Lambda_{i, N}$.
  We claim that the fixed index gap statistic \(\E F(\Lambda_{i, N})\) for $i$ in the bulk is asymptotically equal to the fixed energy reweighted gap statistic
  \[
  {\E} \frac{F(\Lambda^E_N)}{\lambda_{0,E}^{(N)} - \lambda_{1,E}^{(N)}}.
  \]
  
  Because of the log-concavity of $\ell^\frac{1}{2}\left(\MM|_{\lceil \ell \x\rceil + \Box_m} - \MM(\lceil\ell\x\rceil)  \one_{\Box_m}\right)$, for any fixed $\ell$,  the distribution of any Lipschitz function is subexponential (see \cite[Theorem 1.2]{Emanuel}), and therefore the aforementioned asymptotic equality for smooth compactly supported functions $F$ implies convergence in the Wasserstein metric.
  
  Here, $E$ is the classical location of $\la_i$.
  
  Firstly, because of index universality (Theorem~\ref{theorem:Tao3}), one can approximate the fixed index statistic
  \[
  {\E} F(\Lambda_{i, N})
  \]
  by the averaged index statistic
  \[
  \frac{1}{\log^{10} N} \sum_{i \leq i' \leq i+\log^{10} N} \E F(\Lambda_{i', N}))
  \]
  up to small error. We choose $\log^{10} N$ here because this is larger than the oscillations given by eigenvalue rigidity. Now observe that $F(\Lambda^{E'}_N)$ is equal to $F(\Lambda_{i', N})$ whenever $E'$ lies in $(\lambda_{i'}^{(N)}, \lambda_{i'+1}^{(N)})$, in which case $\Lambda^{E'}_N$ is equal to $\Lambda_{i', N}$. As a consequence, we see that
  \[
  \sum_{i \leq i' \leq i+\log^{10} N} {\E} F(\Lambda_{i, N}) = \int^{\lambda_{i-1}^{(N)}}_{\lambda_{i+\log^{10} N}^{(N)}} \frac{F(\Lambda^{E'}_N)}{\lambda_{0,E'}^{(N)} - \lambda_{1,E'}^{(N)}} \, dE'.
  \]
  
  By eigenvalue rigidity, this averaged index expression is very close (within $o(\log^{10} N)$) of the averaged energy expression
  \[
  \int_{E-\log^{10} N}^{E} \frac{F(\Lambda^{E'}_N)}{\lambda_{0,E'}^{(N)} - \lambda_{1,E'}^{(N)}} \, dE'.
  \]
  From the Wegner estimate (see \cite[Theorem 3.5]{ESY}),  denoting by $\mathcal{N}_I$ the number of eigenvalues of a GUE (rescaled so that the mean gap at $E$ is $1$), in the interval $I$, we have that 
  \begin{equation}\label{la-2}
  \p[\mathcal{N}_{(-\eps + E', E' + \eps)} \geq 2] \leq C \eps^4.
  \end{equation}
  By fixed energy universality (\cite[Corollary 1.4]{ANV14}) and the above Wegner eigenvalue repulsion estimate,
  the expression
  \[
  {\E} \frac{F(\Lambda^{E'}_N)}{\lambda_{0,E'}^{(N)} - \lambda_{1,E'}^{(N)}}
  \]
  is basically 
   independent of $E'$. So we can pass from the averaged energy statistic to the fixed energy statistic
  \[
  {\E} \frac{F(\Lambda^E_N)}{\lambda_{0,E}^{(N)} - \lambda_{1,E}^{(N)}}.
  \]
  This proves the theorem.
  \end{proof}

\section{Fefferman's proof of Theorem~\ref{thm:Feff}}\lab{sec:Feff}
\subsection*{Basic Identity}

If \( u \) is Lipschitz on \( \mathbb{R}^n \) with compact support, then

\[
\int_{\mathbb{R}^n} |\nabla u(x)|^2 \, dx = c_n \int_{\substack{\xi \in \mathbb{R}^n\\R>0}} \left( \frac{1}{R^{n+2}} \int_{x \in B(\xi, R)} |u(x) - L_{\xi,R}(x)|^2 \, dx \right) \frac{d\xi \, dR}{R}
\]

where \( B(\xi, R) \) denotes the ball about \( \xi \) of radius \( R \), and \( L_{\xi,R} \) is the linear (including a constant term) function that best approximates \( u \) in \( L^2(B(\xi, R)) \).

See Subsection~\ref{ssec:CF-appendix} for a sketch of the proof.

Let \( Q^\circ \) denote the unit cube in \( \mathbb{R}^n \),  
and let \( Q \) denote a dyadic subcube of \( Q^\circ \). We write \( \de_Q \) for the sidelength of \( Q \)  
and \( |Q| \) for \( \de_Q^n \), the volume of \( Q \).

We write \( E(Q) \subset \mathbb{R}^n \times (0, \infty) \) to denote the set of all \((\xi, R)\) such that  
\( \xi \in Q \) and \( Q \subset B(\xi, R) \), \( 2R < (1.01) \cdot \text{diameter}(Q) \).

For \( u \in L^2_{\text{loc}}(\mathbb{R}^n) \), we write \( L_Q \) to denote the linear function that  
best approximates \( u \) in \( L^2(Q) \).

Note that \( E(Q) \) and \( E(Q') \) are disjoint if \( Q, Q' \) are distinct dyadic cubes.

Note also that  
\[
c |Q| < \int_{E(Q)} \frac{d\xi \, dR}{R^{n+1}} < C|Q|
\]
for any dyadic cube \( Q \).

Moreover, if $(\xi, R) \in E(Q)$ then  
$$c \de_Q < R < C\de_Q$$ and 
\[
c|Q| < R^n < C|Q|.
\]

\[
\|u - L_Q\|_{L^2(Q)} \leq \|u - L_{\xi, R}\|_{L^2(Q)} \leq \|u - L_{\xi, R}\|_{L^2(B(\xi, R))}.
\]
Therefore, if $u$ is Lipschitz and of compact support, then 
\[
(*) \,\,\,\,\,\,\,\,\, \sum_Q \frac{1}{\de_Q^{2}} \int_Q |u(x) - L_Q(x)|^2 dx \leq \sum_Q \int_{(\xi, R) \in E(Q)} \left( \frac{1}{R^{n+2}} \int_{B(\xi, R)} |u(x) - L_{\xi, R}(x)|^2 dx \right) \frac{d\xi \, dR}{R}
\]

$$ \leq \int_{\substack{\xi \in \R^n\\R>0}} (\frac{1}{R^{n+2}} \int_{B(\xi, R)} |u(x) - L_{\xi, R}(x)|^2 dx) \frac{d\xi\, dR}{R} = C \|\nabla u\|^2.$$

Now let \( F \) be a function on the unit cube \( Q^\circ \), taking the value 0  
at the center of \( Q^\circ \), and having Lipschitz constant \(\leq 1\).

By McShane's extension theorem, there exists a function \( u : \mathbb{R}^n \to \mathbb{R} \),  
such that \( u = F \) on \( Q^\circ \), \( u = 0 \) outside the double of \( Q^\circ \), and  
\( u \) has Lipschitz constant \(\leq C\).

Applying to \( u \) the chain of inequalities \((*)\), we obtain  
the following result:

\[
\sum_{Q} \frac{|Q|}{\de_Q^{n+2}} \int_Q |F(x) - L_Q(x)|^2 dx \leq C,
\]

where \( L_Q \) denotes the linear function that best approximates \( F \)  
in \( L^2(Q) \). Here, the sum is taken over all dyadic subcubes \( Q \subset Q^\circ \).

\bigskip

\noindent\textbf{Equivalently,} \\

\noindent\textbf{Main Estimate:}

\[
\int_{z \in Q^\circ} \left( \sum_{Q \ni z} \frac{1}{\de_Q^{n+2}} \int_Q |F(x) - L_Q(x)|^2 dx \right) dz \leq C.
\]

To prove this estimate, we assumed that $F=0$ at the center of $Q^\circ$. We may drop this assumption by replacing $F$ by $F - F(\text{center\,\,} \text{of\,\,} Q^\circ)$.
\subsection*{CZ Decomposition}

We pick parameters \( \varepsilon \), \( \eta > 0 \), and \( k \geq 1 \),  such that

\begin{equation}
(\smile )\,\,\,\,\,\,\,\,\,\, \varepsilon^{n+2} \eta k \geq C^\# 
\end{equation}

where \( C^\# \) is a large enough constant determined by the dimension \( n \).

We then make a CZ decomposition of \( Q^\circ \) by successively bisecting and stopping according to the following rules:

We stop cutting \( Q \) if either

\begin{equation} \tag{A}
\frac{1}{|Q|^{1 + \frac{2}{n}}} \int_Q |F(x) - L_Q(x)|^2 dx \leq \hat{c} \varepsilon^{n+2}
\end{equation}

\[
\text{\small \(\hat{c}\) is a constant to be picked later. It will depend only on \(n\).}
\]
or
\begin{equation} \tag{B}
\delta_Q \leq 2^{-k}
\end{equation}

If (B) holds, we call \( Q \) \textbf{BAD}; otherwise, \( Q \) is \textbf{GOOD}.

Thus, \( Q^\circ \) is partitioned into finitely many CZ cubes,  each having sidelength \( \geq 2^{-k} \).

\subsection*{The Good Cubes}

Let \( Q \) be a \textbf{Good CZ cube}. Then

\begin{equation} \tag{A}
\frac{1}{\delta_Q^{n+2}} \int_Q |F(x) - L_Q|^2 dx \leq \hat{c} \varepsilon^{n+2}
\end{equation}

Because \( F \) has Lipschitz constant \(\leq 1\) and \( L_Q \) is the best linear approximation to \( F \) on \( L^2(Q) \), we can find that \( F - L_Q \) has Lipschitz constant \(\leq C\).

Suppose \( |F(y) - L_Q(y)| \geq \varepsilon \delta_Q \) for some \( y \in Q \).

Then, because \( F - L_Q \) has Lipschitz constant \(\leq C\), we have

\[
|F - L_Q| \geq \frac{\varepsilon}{2} \delta_Q \quad \text{on} \quad Q \cap B(y, c \varepsilon \delta_Q)
\]

a set whose volume is \(\geq c (\varepsilon \delta_Q)^n\).

Therefore,

\[
\int_Q |F - L_Q|^2 dx > \left( \frac{\varepsilon}{2} \delta_Q \right)^2 \cdot c (\varepsilon \delta_Q)^n
\]

i.e.,

\begin{equation} \tag{!}
\frac{1}{\delta_Q^{n+2}} \int_Q |F - L_Q|^2 dx > c' \varepsilon^{n+2}.
\end{equation}

If we pick \(\hat{c}\) in (A) to be less than our present \( c' \), then (!) contradicts (A). We pick such a \(\hat{c}\) and conclude that

\[
|F - L_Q| \leq \varepsilon \delta_Q \quad \text{on } Q \quad \text{(for any Good CZ cube)}.
\]

\subsection*{The Bad Cubes}

Let \( \hat{Q} \) be a bad CZ cube. Then (A) didn't hold for any dyadic cube \( Q \) strictly containing \( \hat{Q} \),  
else we would have stopped cutting before reaching \( \hat{Q} \).

Therefore,  
\begin{equation}
\frac{1}{\delta_Q^{n+2}} \int_Q |F - L_Q|^2 dx > \hat{c} \varepsilon^{n+2}
\end{equation}
for all $Q \supseteq \hat{Q}$, $Q \neq \hat{Q}$.  
There are \( k \) such dyadic cubes \( Q \) for fixed \( \hat{Q} \).  
Consequently,  
\begin{equation}
(!!)\,\,\,\,\,\,\, \int_{z \in \hat{Q}} \left( \sum_{Q \ni z} \frac{1}{\delta_Q^{n+2}} \int_Q |F - L_Q|^2 dx \right) dz \geq k |\hat{Q}| \cdot \hat{c} \varepsilon^{n+2}
\end{equation}

Summing over all the bad CZ cubes, we find that  
\begin{equation}
\int_{z \in Q^\circ} \left( \sum_{Q \ni z} \frac{1}{\delta_Q^{n+2}} \int_Q |F - L_Q|^2 dx \right) dz \geq \hat{c} \varepsilon^{n+2} k \sum_{\hat{Q} \, \text{bad}} |\hat{Q}|
\end{equation}

Hence, by our \textbf{Main Estimate},  
\begin{equation}
\varepsilon^{n+2} k \sum_{\hat{Q} \, \text{bad}} |\hat{Q}| \leq C.
\end{equation}

If we pick \( C^\# \) in \((\smile)\) to be bigger than our present \( C \),  
then we conclude from \((\smile)\) and \((!!)\) that the total volume  
of all the bad CZ cubes is at most \(\eta\).

\subsection*{Recap}

Suppose \( \varepsilon, \eta > 0 \) and \( k \geq 1 \).

Assume that  
\[
\varepsilon^{n+2} \eta k > C^\# 
\]  
(a large enough constant depending only on the dimension \( n \)).

Let \( F : Q^\circ \to \mathbb{R} \), where \( Q^\circ \) is the unit cube in \( \mathbb{R}^n \).  
Suppose \( F \) has Lipschitz constant \(\leq 1\).

Then we may decompose \( Q^\circ \) into finitely many \textbf{Good} cubes, and finitely many \textbf{Bad} cubes, with the following properties:

\begin{enumerate}
    \item[(I)] The \textbf{Good} cubes have sidelengths \( \geq 2^{-k} \).

    For each good cube \( Q \), we have  
    \[
    |F - L_Q| \leq \varepsilon \delta_Q \quad \text{on} \quad Q,
    \]  
    where \( L_Q \) is the linear function that best approximates \( F \) in \( L^2(Q) \).

    \item[(II)] The \textbf{Bad} cubes have sidelength exactly \( 2^{-k} \), and their total volume is at most \( \eta \).
\end{enumerate}

\subsection{Sketch of the proof of the basic identity}\lab{ssec:CF-appendix}
Let 
\[
\varphi_0^R(x) = c_0\frac{\mathbf{1}_{B(0, R)}(x)}{R^{n/2}} \quad \text{and} \quad \varphi_j^R(x) = c_1 \, x_j \, \frac{\mathbf{1}_{B(0, R)}}{R^{1 + \frac{n}{2}}} \quad \text{for } j = 1, \dots, n,
\]
where \(\mathbf{1}\) denotes the indicator function and the constants \( c_0, c_1\) are picked so that \(\varphi_0^R, \varphi_1^R, \dots, \varphi_n^R\) have \( L^2 \)-norm 1 on \( \mathbb{R}^n \).

Then
\[
L_{\xi, R}(z) = \sum_{j=0}^n \int u(x) \, \varphi_j^R(\xi - x) \, dx \cdot \varphi_j^R(\xi - z),
\]
hence
\[
\int_{x \in B(\xi, R)} |u(x) - L_{\xi, R}(x)|^2 dx = \int_{x \in B(\xi, R)} |u(x)|^2 dx -  \int_{x \in B(\xi, R)} |L_{\xi, R}(x)|^2dx
\]

(Because \( L_{\xi, R} - u \perp u\) on \( B(g, R) \))

\[
= \int_{x \in B(g, R)} |u(x)|^2 dx - \sum_{j=0}^n \left( \varphi_j^R * u (\xi - x)dx \right)^2
\]

(Because \(\varphi_0^R, \dots, \varphi_n^R\) are orthonormal)
\[
= \int_{x \in B(\xi, R)} |u(x)|^2dx - \sum_{j=0}^n|\varphi_j^R \ast u(\xi)|^2. 
\]

Integrating over \( \xi \in \mathbb{R}^n \) for fixed \( R \), we find that

\[
\int_{\xi \in \mathbb{R}^n} \int_{x \in B(\xi, R)} |u(x) - L_{\xi, R}(x)|^2 dx \, d\xi = \tilde{c} \, R^n \left( \int_{\mathbb{R}^n} |u(x)|^2 dx - \sum_{j=0}^n \|\varphi_j \ast u\|_{L^2(\mathbb{R}^n)}^2 \right)
\]

where \( \tilde{c} \) is the volume of the unit ball in \( \mathbb{R}^n \).

By the Plancherel theorem, we have

\[
\int_{\xi \in \mathbb{R}^n} \int_{x \in B(\xi,  R)} |u(x) - L_{\xi, R}(x)|^2 dx \, d\xi = \tilde{c} R^n \int_{\mathbb{R}^n} |\hat{u}(\xi)|^2 d\xi - \sum_{j=0}^n \int_{\mathbb{R}^n} |\widehat{\varphi_j^R}(\xi) \hat{u}(\xi)|^2 d\xi
\]

\[
= \int_{\mathbb{R}^n} \left\{ \tilde{c} R^n - \sum_{j=0}^n |\widehat{\varphi_j^R}(\xi)|^2 \right\} |\hat{u}(\xi)|^2 d\xi
\]

Integrating against \( \frac{dR}{R^{n+3}} \), we find that

\[
\int_{\substack{\xi \in \mathbb{R}^n\\R > 0}} \int_{R > 0} \frac{1}{R^{n+2}} \int_{x \in B(\xi, R)} |u(x) - L_{\xi, R}(x)|^2 dx \, \frac{d\xi \, dR}{R}
\]

\[
= \int_{\mathbb{R}^n} \left[ \int_0^\infty \left\{ \tilde{c} R^n - \sum_{j=0}^n |\widehat{\varphi_j^R}(\xi)|^2 \right\} \frac{dR}{R^{n+3}} \right] |\hat{u}(\xi)|^2 d\xi \tag{\#}
\]

The quantity in square brackets may be computed; it equals

\[
c |\xi|^2 \quad \text{for a constant \( c \) determined by the dimension \( n \)}.
\]

(More on this later.)

Therefore,  
\[
\int_{\substack{\xi\in \mathbb{R}^n\\R > 0}}  \left\{ \frac{1}{R^{n+2}} \int_{x \in B(\xi, R)} |u(x) - L_{\xi, R}(x)|^2 dx \right\} \frac{d\xi \, dR}{R}
= c \int_{\mathbb{R}^n} |\xi|^2 |\hat{u}(\xi)|^2 d\xi = c \int_{\mathbb{R}^n} |\nabla u(x)|^2 dx,
\]

completing the proof of the \textbf{Basic Identity}.

\bigskip

\noindent\textbf{A Note on the Quantity in Square Bracket in (\#)}


We first note that
\[
\int_{\xi \in \mathbb{R}^n} \int_{R > 0} \left\{ \frac{1}{R^{n+2}} \int_{x \in B(\xi, R)} |u(x) - L_{\xi, R}(x)|^2 dx \right\} \frac{d\xi \, dR}{R} < \infty \tag{\#\#}
\]
when \( u \in C_0^\infty(\mathbb{R}^n) \).

If the quantity in square brackets were infinite for a set of \(\xi\) of nonzero measure, then (\#\#) would be impossible. Therefore, (\#) leads to a formula of the form:

\[
\int_{\xi \in \mathbb{R}^n} \int_{R > 0} \frac{1}{R^{n+2}} \int_{x \in B(\xi, R)} |u(x) - L_{\xi, R}(x)|^2 dx \, \frac{d\xi \, dR}{R} = \int_{\mathbb{R}^n} m(\xi) |\hat{u}(\xi)|^2 d\xi \tag{!!!}
\]

for some function \( m(\xi) \). Here, the left-hand side is invariant under rotation and has a homogeneity with respect to dilations.

It follows that \( m(\xi) \) is rotation-invariant and homogeneous of degree 2 under dilations.  
Therefore,  
\[
m(\xi) = c |\xi|^2
\]  
for some constant \( c \).

One has to compute in order to determine the constant \( c \), but here we use only that \( c \) is determined by the dimension \( n \), which is obvious.

Plugging our formula for \( m(\xi) \) into \((!!!)\), we obtain the \textbf{Basic Identity} without doing the computation of the quantity in square brackets in (\#).

\end{document}